\documentclass[letterpaper,12pt]{article}

\usepackage[sort&compress]{natbib}
 \bibpunct[, ]{[}{]}{,}{n}{}{,}%
\usepackage{basic_template}
\usepackage{multirow}

\graphicspath{
  {../Figures/} 
  {Figures/}  
}

\newtheorem{conjecture}{Conjecture}
\crefname{conjecture}{conjecture}{conjectures}
\Crefname{conjecture}{Conjecture}{Conjectures}

\def\bo{\mathbf{0}}

\def\bb{\mathbf{b}}
\def\be{\mathbf{e}}
\def\bff{\mathbf{f}}
\def\bg{\mathbf{g}}
\def\bp{\mathbf{p}}

\def\bu{\mathbf{u}}
\def\bv{\mathbf{v}}
\def\bw{\mathbf{w}}
\def\bx{\mathbf{x}}
\def\by{\mathbf{y}}
\def\bz{\mathbf{z}}

\def\bA{\mathbf{A}}
\def\bB{\mathbf{B}}
\def\bC{\mathbf{C}}
\def\bD{\mathbf{D}}

\def\bF{\mathbf{F}}
\def\bG{\mathbf{G}}
\def\bH{\mathbf{H}}

\def\bM{\mathbf{M}}
\def\bN{\mathbf{N}}

\def\bP{\mathbf{P}}
\def\bQ{\mathbf{Q}}
\def\bR{\mathbf{R}}

\def\bT{\mathbf{T}}
\def\bU{\mathbf{U}}

\def\bW{\mathbf{W}}
\def\bX{\mathbf{X}}

\def\cF{\mathcal{F}}
\def\cI{\mathcal{I}}

\def\cL{\mathcal{L}}
\def\cO{\mathcal{O}}
\def\cQ{\mathcal{Q}}
\def\cS{\mathcal{S}}
\def\cX{\mathcal{X}}
\def\mE{\mathbb{E}}

\def\mR{\mathbb{R}}
\def\mS{\mathbb{S}}

\def\th{{\boldsymbol{\theta}}}

\newcommand{\seq}[1]{\{ #1 \}}
\newcommand{\grad}{\nabla}
\newcommand{\abs}[1]{\left|#1\right|}
\newcommand{\norm}[1]{\left\|#1\right\|}
\DeclareMathOperator{\Tr}{Tr}
\DeclareMathOperator*{\argmin}{argmin}
\DeclareMathOperator{\Rank}{rank}
\DeclareMathOperator{\st}{s.t.}
\DeclareMathOperator{\diag}{diag}
\newcommand{\up}{\rule{0pt}{1.5em}}
\newcommand{\Halmos}{\qed}
\def\fprod#1{\left\langle#1\right\rangle}
\newcommand\numberthis{\addtocounter{equation}{1}\tag{\theequation}}
\newcommand{\bmat}[1]{\begin{bmatrix}#1\end{bmatrix}}

\begin{document}

\title{\textbf{Nonasymptotic Analysis of Accelerated Methods With Inexact Oracle Under Absolute Error Bound}}

\author{Yin Liu\footnotemark[1] \and Sam Davanloo Tajbakhsh\footnotemark[2]}

\footnotetext[1]{Beijing International Center for Mathematical Research, Peking University, Beijing, China, \texttt{yinliu@pku.edu.cn} }
\footnotetext[2]{The Ohio State University, Columbus, OH, USA, \texttt{davanloo.1@osu.edu}}


\date{\today}


\maketitle

\begin{abstract}
Performance analysis of first-order algorithms with inexact oracles has gained recent attention due to various emerging applications in which obtaining exact gradients is impossible or computationally expensive.
Previous research has demonstrated that the performance of accelerated first-order methods is more sensitive to gradient errors compared with non-accelerated ones. This paper investigates the nonasymptotic convergence bound of two accelerated methods with inexact gradients to solve deterministic smooth convex problems. Performance Estimation Problem (PEP) is used as the primary tool to analyze the convergence bounds of the underlying algorithms.  By finding an analytical solution to PEP, we derive novel convergence bounds of the Generalized Optimized Gradient Method (GOGM) and Generalized Fast Gradient Method (GFGM) \emph{with inexact gradient oracles following the absolute error bound.}
Next, we analyze the tradeoff between the vanishing term and the accumulated error in the convergence bound that guides finding the optimal stepsize. Furthermore, we determine the optimal strategy to set the gradient inexactness along iterations, ensuring that the accumulated error remains subordinate to the vanishing term. Finally, we establish a lower bound for accumulated error for a subclass of first-order methods satisfying two properties, which then motivated us to develop PEP solutions satisfying this lower bound for GOGM and GFGM methods.
\end{abstract}

\section{Introduction}
We consider the optimization problem
\begin{equation}\label{eq:main}
	\min_{\bx\in \mR^d} f(\bx),
\end{equation}
where \(f\) is convex and has a Lipschitz continuous gradient.
We assume the optimal value of \eqref{eq:main} is finite and attained, i.e., \(f_* >-\infty\) with \(f_* := \min_\bx f(\bx)\) and the solution set is nonempty.

We use \(\cF_{\mu,L}\) with \(\mu \geq 0\) to denote the class of \(\mu\) strongly convex functions with Lipschitz continuous gradient with constant \(L\). This paper focuses on the class of merely convex functions with Lipschitz continuous gradient \(f \in \cF_{0,L}\).
For this class of functions, the gradient descent (GD) method has the iteration complexity \(f(\bx_K) - f_* = \cO(K^{-1})\). This result can be improved to \(\cO(K^{-2})\) with Nesterov's fast gradient method (FGM) \cite{Nesterov1983MethodUnconstrainedConvex, Nesterov1988ApproachConstructionOptimal, Nesterov2005SmoothMinimizationNon, Nesterov2018LecturesConvexOptimization}. Recently, using computer-aided analysis, the Optimized Gradient Method (OGM) has been proposed \cite{Drori2014PerformanceFirstOrder, Kim2016OptimizedFirstOrder} that improves the complexity of FGM by a \(\sqrt{2} \) constant.

The above results are based on having access to the \emph{exact} gradient at any point, which is not the case in many applications. The performance of the algorithms deteriorates when there is an error in the gradient estimate. We consider the gradient error as
\begin{equation}\label{eq:inexact-oracle}
	\tilde{\grad}f(\bx) :=  \grad f(\bx) + \be, \qquad \norm{\be}\leq b = h(\boldsymbol{\eta}),
\end{equation}
where the error \(\be\) can be controlled by a parameter \(\boldsymbol{\eta}\in\mR^r\) through a positive function \(h(\cdot)\) and the bound holds either deterministically, which is the focus of this work, or with high probability in stochastic settings.

For example, the parameter \(\boldsymbol{\eta}\) can represent computational effort, the accuracy of an approximation, or the sampling size used to achieve a desired gradient accuracy.

\subsection{Applications with inexact gradient oracles.} \label{sec:applications}
In some applications, gradient inexactness is the result of the oracle not being evaluated at the desired point. For instance, in bilevel optimization, one block of coordinates of the upper-level problem is the solution of the lower-level optimization problem. However, in practice, the lower-level problem can only be solved to a suboptimal solution, which results in upper-level gradient inexactness. Similarly, in composition optimization, the desired point results from an expectation that in many scenarios can only be estimated, resulting in gradient inexactness. These two applications are discussed below.
\paragraph{Bilevel optimization.}
	Consider the bilevel optimization problem
	\begin{align*}
		\min_\bx & \quad f(\bx;\by^*(\bx))                    \\
		\st      & \quad \by^*(\bx) = \argmin_\by g(\bx,\by),
	\end{align*}
	with \(f\) being continuously differentiable and \(g\) being two times continuously differentiable and strongly convex. To solve the problem with a gradient-based method, the gradient of the upper-level problem with respect to \(\bx\) is
	\begin{equation*}
		\grad f(\bx;\by^*(\bx)) = \grad_\bx f(\bx;\by^*(\bx)) - \grad^2_{\bx\by} g(\bx,\by^*(\bx))[\grad^2_{\by\by} g(\bx,\by^*(\bx))]^{-1} \grad_\by f(\bx;\by^*(\bx)),
	\end{equation*}
	which requires solving the lower-level problem to optimality. When the lower-level problem is solved to a suboptimal point \(\tilde{\by}(\bx)\) and the gradient is evaluated at this point instead of \(\by^*(\bx)\), the resulting gradient is inexact.
Let \(\tilde{\by}_k(\bx)\) be the solution of the lower-level problem by the gradient descent method with stepsize \(\frac{2}{\mu_g + L_g}\) after \(k\) iterations. From the iteration complexity of GD for the function class \(\cF_{\mu,L}\), the gradient inexactness can be bounded as
	\begin{equation*}
		\norm{\tilde{\grad}f(\bx;\tilde{\by}_k(\bx)) - \grad f(\bx;\by^*(\bx))} \leq C \left(\frac{Q_g - 1}{Q_g + 1}\right)^k \norm{\by_0 - \by^*(\bx)},
	\end{equation*}
	where \(C\) is a constant and \(Q_g := L_g/\mu_g\). The above bound follows an exponential decay and represents the error bound in \eqref{eq:inexact-oracle} with \(\boldsymbol{\eta}\) being the iteration number \(k\) - For more details see~\ref{sec:detail_example}.

\paragraph{Composition optimization.}
	Consider the nested composition optimization problem
	\begin{equation*}
		\min_{\bx} \ f(\bx) := h(\bg(\bx)) \quad \text{with} \quad \bg(\bx) := \frac{1}{N}\sum_{i=1}^N \bg_i(\bx),
	\end{equation*}
    where \(h:\mR^k\rightarrow\mR\) and \(\bg:\mR^d\rightarrow\mR^k\) are continuously differentiable, and \(N\) is a large positive integer. From the chain rule, the exact gradient is \(\grad f(\bx) = \grad \bg(\bx) \grad h(\bg(\bx))\). However, assume one can only approximate \(\bg\) and \(\grad \bg\) through their minibatch samples \(\cS_{\bg}\subseteq [N]\) and \(\cS_{\grad}\subseteq[N]\), respectively, where \([N]:=\seq{1,2,\cdots,N}\). Then, under some Lipschitz continuity and bounded variance assumptions, as shown in \cref{lem:composition-error-bound}, with probability at least \(1 - \frac{1}{\epsilon}\), the inexact gradient estimate is bounded as
\begin{equation*}
		\norm{\tilde{\grad} \bg(\bx) \grad h (\tilde{\bg}(\bx)) - \grad \bg(\bx) \grad h(\bg(\bx))}^2 \leq \epsilon \left(\frac{C_1}{\abs{\cS_\grad}} + \frac{C_2}{\abs{\cS_\bg}} \right).
\end{equation*}
This bound represents the error bound in \eqref{eq:inexact-oracle} with \(\boldsymbol{\eta}=\left(\abs{\cS_\grad},\abs{\cS_\bg}\right)\). Furthermore, if \(\cS_\grad = \cS_\bg\), for any fixed \(\epsilon\), the bound follows a power law decay.

Another scenario is when \(\grad f\) is unattainable, but it can be approximated by inexact function values (zero-order information), denoted by \(\tilde{f}\). In many applications, it is often reasonable to assume \(\abs{\tilde{f}(\bx) - f(\bx)} \leq b_f\).
 Below, we discuss deterministic and stochastic zero-order methods that generate inexact gradient estimates with quantifiable error bounds~\cite{Berahas2022TheoreticalEmpiricalComparison}.

\paragraph{Gradient estimation via forward finite difference.}
	Let \(\bu_i \in \mR^d\) denote the unit vector with the \(i\)-th element equal to 1, and \(l>0\) be the finite difference interval. Define the \(i\)-th element of the gradient estimator as
	\begin{equation*}
		[\tilde{\grad}f(\bx)]_i = \frac{\tilde{f}(\bx + l \bu_i) - \tilde{f}(\bx)}{l}.
	\end{equation*}
When \(f(\bx)\) has Lipschitz continuous gradient, the error of the gradient estimate is shown by \cite{Berahas2022TheoreticalEmpiricalComparison} to be bounded as
	\begin{equation*}
		\norm{\tilde{\grad}f(\bx) - \grad f(\bx)} \leq \frac{\sqrt{d}L l}{2} + \frac{2\sqrt{d}b_f}{l},
	\end{equation*}
	which represents the error bound \eqref{eq:inexact-oracle} with \(\boldsymbol{\eta}=(l,b_f)\). Note that the forward finite difference requires querying the zero-order oracle \(d+1\) times to approximate the gradient at \(\bx\).

\paragraph{Gradient estimation via Gaussian smoothing.}
	Let \(\seq{\bv_i}_{i=1}^N\) be i.i.d. random directions following the standard multivariate normal distribution \(\bv_i\sim\mathcal N(0,I_d)\), and define
	\begin{equation*}
		\tilde{\grad}f(\bx) = \frac{1}{N}\sum_{i=1}^N \frac{\tilde{f}(\bx + l \bv_i) - \tilde{f}(\bx)}{l}\bv_i.
	\end{equation*}
	It is shown in \cite{Berahas2022TheoreticalEmpiricalComparison} that
	\[
        \norm{\mE_{\seq{\bv_i}_{i=1}^N}[\tilde{\grad}f(\bx)] - \grad f(\bx) } \leq \sqrt{d}Ll + \frac{\sqrt{d}b_f}{l},
	\]
 which represents the error bound \eqref{eq:inexact-oracle} with \(\boldsymbol{\eta}=(l,b_f)\).

\subsection{Effect of oracle inexactness on acceleration}
A careful review of the literature shows that the effect of oracle inexactness on acceleration is typically addressed through 1) development of a specific error condition, and 2) analysis of a common (or customized) accelerated method under that error condition with (or in a few cases without) extra assumption. A detailed literature review on accelerated methods under different gradient error conditions is relegated to Section~\ref{sec:related-work}.

This paper considers gradient inexactness in the form of the absolute error defined as~\cite{Polyak1987IntroductionOptimization}
\begin{align}
	\text{(absolute error)}  \qquad  \norm{\tilde{\grad} f(\bx) - \grad f(\bx) } \leq b. \tag{AE}  \label{eq:ae}
\end{align}
Motivated by the application discussed in Section~\ref{sec:applications}, \eqref{eq:ae} can be considered as one of the most common forms to quantify the gradient error. Furthermore, as discussed in Section~\ref{sec:related-work}, obtaining other error conditions from \eqref{eq:ae}, e.g., BIE, IFO, etc., requires extra assumptions that limit their applications.

Therefore, instead of translating \eqref{eq:ae} into a more restrictive error model, we directly study how this oracle interacts with accelerated first-order schemes.
Nesterov's Fast Gradient Method (FGM) and its variations have been extensively studied under various inexact gradient assumptions, which we refer to as iFGM. The standard FGM update generates sequences \(\seq{\by_k}\) and \(\seq{\bz_k}\), with the gradient evaluated at a convex combination of these points. The detailed update rule is
\begin{equation*}
\begin{cases}
    \by_{k+1} = \bx_k - \frac{1}{L}\tilde{\grad}f(\bx_k), \\
    \bz_{k+1} = \bz_k - \frac{1}{L}\alpha_k \tilde{\grad}f(\bx_k), \\
    \bx_{k+1} = (1- \frac{1}{\alpha_{k+1}})\by_{k+1} + \frac{1}{\alpha_{k+1}} \bz_{k+1},
\end{cases}
\end{equation*}
where \(\alpha_0=1\), \(\alpha_{k+1} = \frac{1+\sqrt{1+4\alpha_k^2}}{2}\), and \(\tilde{\grad} f(\cdot)\) denotes the inexact gradient. When the exact gradient is available, FGM achieves the optimal convergence rate of \(\cO(K^{-2})\). However, when the gradient is inexact, this fast convergence also accelerates error accumulation, undermining the method's performance.
This observation motivates the use of generalized accelerated schemes, where the acceleration steps can be tuned to balance exact-oracle convergence against error accumulation.

To address this trade-off, a generalized version of iFGM (iGFGM, \cref{alg:iGFGM}) is often considered. This approach controls the increase of \(\alpha_k\) to obtain a better tradeoff between the convergence rate and the accumulated error. For instance, under a specific assumption of the gradient inexactness (see details in \cref{sec:related-work}), \citet{Devolder2013FirstOrderMethods} established the convergence result of iGFGM as
\begin{equation*}
    \cO\left(\frac{1}{A_K} + \frac{\sum_{i=0}^K A_i \delta_{\bx_i}}{A_K} \right),
\end{equation*}
where \(\delta_{x_i}\) represents the gradient error at \(x_i\) and \(A_i=\sum_{j=1}^i \alpha_j\) with \(\alpha_i^2\leq A_i\).

While iGFGM has been widely investigated under error conditions \emph{other than} \eqref{eq:ae} (except for one work~\cite{ArtemVasin2023AcceleratedGradientMethods} which is discussed in Section~\ref{sec:related-work}), the Optimized Gradient Method (OGM) with an inexact gradient, i.e., iOGM, has received less attention.
A similar tradeoff between convergence rate and accumulated error exists for iOGM. In this paper, we consider the generalized version \cite{Kim2018GeneralizingOptimizedGradient}, which we term iGOGM (\cref{alg:iGOGM}).

We note that the difference between the iGFGM and iGOGM methods is in their step 3, where iGOGM's stepsize is two times larger. When \(\lambda_k = 1\), the \(\bx_{k+1}\) updates (step 5) in both iGFGM and iGOGM are simplified as
\(\bx_{k+1} = \left(1 - \frac{1}{\alpha_{k+1}}\right) \by_{k+1} + \frac{1}{\alpha_{k+1}} \bz_{k+1},\)
and algorithms reduce to iFGM and iOGM, respectively.

The properties of OGM and its generalization have been thoroughly studied. However, the convergence analysis of OGM (or its generalization) with an inexact gradient oracle has not been performed yet. Furthermore, the existing analysis of iGFGM with the absolute error assumption depends on the trajectory of sequences and requires fixed error throughout the iterations, as discussed in ~\cite{ArtemVasin2023AcceleratedGradientMethods} and summarized in Section~\ref{sec:related-work}. This paper aims to close these gaps.
The two generalized algorithms are presented side by side below to make this structural difference explicit.

\begin{minipage}[t]{0.48\textwidth}
    \begin{algorithm}[H]
        \caption{Inexact Generalized Fast Gradient Method (iGFGM)} \label{alg:iGFGM}
        \begin{algorithmic}[1]
            \Require \(\bz_0 = \bx_0 \in \mR^d \), \(A_0 = \alpha_0 =1 \), stepsize parameter \(\seq{\lambda_k}\), \(\lambda\in[0,1]\).
            \For{\(k=0, \cdots, K-1\)}
            \State \(\by_{k+1} = \bx_{k} - \frac{1}{L}\tilde{\grad} f(\bx_k)\)
            \State \(\bz_{k+1} = \bz_k  - \frac{1}{L} \alpha_k \tilde{\grad} f(\bx_k)\)
            \State \(\begin{aligned}[t]
                \alpha_{k+1} & =\frac{\lambda_{k+1}+\sqrt{4\lambda_{k+1}A_k+\lambda_{k+1}^2}}{2}, \\
                A_{k+1}      & =A_k+\alpha_{k+1}
            \end{aligned}\)
            \State \(\bx_{k+1} =(1- \frac{\alpha_{k+1}}{A_{k+1}})\by_{k+1} + \frac{\alpha_{k+1}}{A_{k+1}} \bz_{k+1}\)
            \EndFor
        \end{algorithmic}
    \end{algorithm}
\end{minipage}%
\hfill
\begin{minipage}[t]{0.48\textwidth}
    \begin{algorithm}[H]
        \caption{Inexact Generalized Optimized Gradient Method (iGOGM)} \label{alg:iGOGM}
        \begin{algorithmic}[1]
            \Require \(\bz_0 = \bx_0 \in \mR^d \), \(A_0 = \alpha_0 =1\), stepsize parameter \(\seq{\lambda_k}\), \(\lambda\in[0,1]\).
            \For{\(k=0, \cdots, K-1\)}
            \State \(\by_{k+1} = \bx_{k} - \frac{1}{L}\tilde{\grad} f(\bx_k)\)
            \State \(\bz_{k+1} = \bz_k  - \frac{2}{L} \alpha_k \tilde{\grad} f(\bx_k)\)
            \State \(\begin{aligned}[t]
                \alpha_{k+1} & =\frac{\lambda_{k+1}+\sqrt{4\lambda_{k+1}A_k+\lambda_{k+1}^2}}{2}, \\
                A_{k+1}      & =A_k+\alpha_{k+1}
            \end{aligned}\)
            \State \(\bx_{k+1} =(1- \frac{\alpha_{k+1}}{A_{k+1}})\by_{k+1} + \frac{\alpha_{k+1}}{A_{k+1}} \bz_{k+1}\)
            \EndFor
        \end{algorithmic}
    \end{algorithm}
\end{minipage}
\vspace{0.3cm}

\subsection{Contributions}
In this paper, we analyze nonasymptotic convergence bounds of two accelerated gradient methods, namely Generalized Fast Gradient Method (iGFGM,~\cref{alg:iGFGM}) and Generalized Optimized Gradient Method (iGOGM,~\cref{alg:iGOGM}), under \emph{inexact} gradient oracles satisfying the \emph{absolute error condition} \eqref{eq:ae}. The \eqref{eq:ae} error bound finds a range of applications, a few of which are discussed in Section~\ref{sec:applications}. These bounds are established through the Performance Estimation Problem (PEP) technique and, unlike previous works, are independent of unknown quantities--see the first bullet below for more details. Furthermore, our analysis allows variable error along the iterations of the algorithms; hence, the established bounds allow exploiting the tradeoff between the per-iteration cost to control the bias and the total cost to obtain the optimal oracle inexactness schedule. Our key contributions can be summarized as follows:

\begin{itemize}
\item As summarized in \cref{tab:algorithm-comparasion}, existing methods are limited by their reliance on strong assumptions (e.g., BIE, IFO, IFO-q, and RE) which are generally not easily verifiable in different applications. Furthermore, some previous analyses depend on unquantifiable terms, e.g., the dependence of the bound in \cite{ArtemVasin2023AcceleratedGradientMethods} on unquantifiable \(\tilde{R}_K\), the radius of the set containing the algorithm's trajectory. Hence, it is not possible to evaluate the accumulated error in \cite{ArtemVasin2023AcceleratedGradientMethods} and to determine the algorithm's optimal parameter settings. Our work closes this gap by deriving a quantifiable expression for the accumulated error under the \eqref{eq:ae} condition.
\item We provide the convergence bound of iGOGM and iGFGM under the absolute error \eqref{eq:ae} condition in Section~\ref{sec:main_thms}. The derived convergence bounds consist of two components: the diminishing component independent of the oracle's error and the component containing the accumulated error. Interestingly, the accumulated error is independent of the initial condition \(\norm{\bx_0 - \bx_*}\), and is determined solely by the Lipschitz constant and the stepsize. This result advances the convergence bound of \cite{ArtemVasin2023AcceleratedGradientMethods}, presented in \eqref{eq:Vasin-result}, in the sense that it eliminates the dependence on the unknown parameter \(\tilde{R}_k\). Furthermore, unlike \citet{Nabou2024ProximalGradientMethods}'s inexactness condition for the gradient to be a subgradient, we do not require such a condition, and we yet achieve comparable convergence guarantees.
\item In \cref{sec:rate-error-tradeoff}, we analyze the established convergence bounds from two different perspectives: i) we investigate the tradeoff between the vanishing term and the accumulated error; ii) we establish the optimal inexactness schedule, i.e., the minimal total cost to control the error, while preserving the accelerated convergence bound. The proofs of our main convergence results for iGOGM and iGFGM are presented in Sections~\ref{sec:iGOGM} and \eqref{appx:iGFGM}, respectively.
\item We derive an \emph{analytical feasible solution} to the dual of the relaxed semidefinite programming formulation of the Performance Estimation Problem (PEP) for first-order algorithms with inexact \eqref{eq:ae} oracles in \cref{sec:iGOGM}. The approach used to find this solution inspires our theoretical proof of the convergence bound. Notably, this proof can be understood without prior knowledge of the PEP technique and could be of independent interest.

\end{itemize}

The codes for our numerical experiments as well as those to verify the algebraic steps in our proofs are available at \href{https://github.com/Yin-LIU/Inexact-Acceleration-PEP}{https://github.com/Yin-LIU/Inexact-Acceleration-PEP}.

\paragraph{Notations}
Vectors and matrices are denoted by bold and capitalized bold letters, respectively. Sets are denoted by calligraphic letters. In the absence of additional instruction, a bold letter coupled with an arrow, e.g., \(\vec{\bu}\), generally denotes a standard basis vector, which has exactly one element equal to 1 while other elements are 0. \(\bM_{i,j}\) stands for the \((i,j)-\)th element of the matrix \(\bM\),  with row and column indices starting from \(1\). The notation \(\bM^{\backslash [i] }\) refers to the submatrix of \(\bM\) obtained by removing its \(i\)-th column and row. \([\bM]_{n\times m}\) indicates the shape of the matrix and \([\bM]\) is used to emphasize the item inside the brackets is a matrix. \(\norm{\cdot}\) denotes the \(l_2\) norm for vectors and Frobenius norm for matrices.
\(\log(\cdot)\) represents the natural logarithm.

\subsection{Related work}\label{sec:related-work}
Below, we discuss the literature related to accelerated methods under various oracle inexactness conditions.

\textbullet \ \textbf{Bounded Inner Product Error (BIE) }~\citet{dAspremont2008SmoothOptimizationApproximate} focuses on smooth convex optimization over a \textit{compact} convex set \(\cQ\), with the inexact gradient satisfying
\begin{align}
    \abs{\fprod{\tilde{\grad} f(\bx) - \grad f(\bx), \by - \bz}} \leq \delta,  \quad \forall \bx, \by,\bz \in \cQ. \tag{BIE} \label{eq:bie}
\end{align}
The BIE condition can be derived from the absolute error assumption when the feasible set is bounded. Specifically, if the gradient error satisfies \(\norm{\tilde{\grad} f(\bx) - \grad f(\bx)} \leq b\) for some constant \(b\), then
\begin{equation*}
	\abs{\fprod{\tilde{\grad} f(\bx) - \grad f(\bx), \by - \bz}} \leq b \norm{\by - \bz} \leq b\max_{\by,\bz \in \cQ}\norm{\by-\bz} = b \cdot D := \delta,
\end{equation*}
where \(D = \max_{\by,\bz \in \cQ} \norm{\by- \bz}\) represents the diameter of the feasible set. This derivation highlights that the error bound \(\delta\) is directly proportional to the size of the feasible region.

Under the BIE assumption, iFGM has the convergence rate of
\begin{equation*}
    f(\by_{K+1}) - f_* \leq \frac{L\norm{\bx_0 - \bx_*}^2}{A_K} + 3\delta .
\end{equation*}
Note that the convergence rate is determined by two components: the first term decreases with the accumulation parameter \(A_K\), and the second term is a constant that scales with the inner product error bound \(\delta\).

\textbullet \ \textbf{Inexact First-order \((\delta,L)\) Oracle (IFO)} ~\citet{Devolder2013FirstOrderMethods}
introduce the \((\delta, L)\) first-order oracle condition for problems with unbounded feasible set or those with nonsmooth objective functions. The inexactness of the oracle pair \((\tilde{f}_{\delta_\bx}(\bx),\tilde{\grad} f_{\delta_\bx}(\bx))\) is quantified by
\begin{align*}
    0 \leq f(\by) - (\tilde{f}_{\delta_\bx}(\bx) + \fprod{ \tilde{\grad} f_{\delta_\bx}(\bx),\by-\bx} \leq \frac{L}{2} \norm{\bx - \by}^2 + \delta_\bx, \tag{IFO} \label{eq:ifo}
\end{align*}
which is a relaxation of the first-order convexity and Lipschitz smoothness conditions.

Under the IFO assumption, \citet{Devolder2013FirstOrderMethods} shows that the sequence generated by iFGM satisfies
\begin{equation*}
    f(\by_{K+1}) - f_* \leq \frac{L \norm{\bx_0 - \bx_*}^2}{A_K} + \frac{\sum_{i=0}^K A_i \delta_{\bx_i}}{A_K}.
\end{equation*}
This result is similar to that of \eqref{eq:bie} condition, i.e., a term decreasing with iteration and an accumulated error term. However, \eqref{eq:ifo} condition also allows varying inexactness levels along the iterations. While \eqref{eq:ifo} condition allows an unbounded feasible set, obtaining it from the absolute error condition requires bounded feasible sets. Specifically, for an oracle with absolute errors \(\abs{\tilde{f}(\bx) - f(\bx)}\leq \Delta_1\) and \(\norm{\tilde{\grad} f(\bx) - \grad f(\bx)} \leq \Delta_2\). By defining \(\tilde{f}_{\delta_\bx} := \tilde{f}(\bx) - \Delta_1 - \Delta_2 D \) and \(\tilde{\grad} f_{\delta_\bx}(\bx) := \tilde{\grad} f(\bx) \), where \(\delta_\bx = 2\Delta_1 + 2\Delta_2 D\) (with \(D\) representing the feasible region's diameter), it can be shown the oracle with \eqref{eq:ae} condition is a \((\delta,L)-\)oracle~\cite{Devolder2013FirstOrderMethods}. The \eqref{eq:ifo} condition motivated \cite{zoll2026inexactly} to consider a family of ``inexactly smooth" functions, present the corresponding interpolation theorems, and design minimax optimal algorithms for important special cases.

\textbullet \ \textbf{Inexact First-order \((\delta,L)\) Oracle of Degree \(q\) (IFO-q)}
To be able to address oracle inexactness in minimization of smooth nonconvex or nonsmooth convex objectives, while sacrificing the inexact gradients to be subgradients, \citet{Nabou2024ProximalGradientMethods} propose the IFO-q oracle with an additional degree parameter \(q \in [0, 2)\) as
\begin{align*}
    0 \leq f(\by) - (f(\bx) + \fprod{\tilde{\grad}f_{\delta}(\bx),\by-\bx} \leq \frac{L}{2}\norm{\by-\bx}^2 + \delta \norm{\bx-\by}^q. \tag{IFO-q} \label{eq:ifoq}
\end{align*}
With an appropriate selection of the parameters \(q\) and \(\delta\), any inexact gradient satisfying the absolute error condition also satisfies the upper bound IFO-q condition even on unbounded feasible sets. However, the lower bound of this inequality requires the inexact gradient to be a subgradient of the objective function. This subgradient condition is not inherently guaranteed by the absolute error assumption, limiting the applicability of the IFO-q oracle when solely relying on absolute gradient errors. The convergence rate of iFGM {with an oracle satisfying IFO--q} is given as
\begin{equation}\label{eq:Nabou-result}
    f(\by_{K+1}) - f_* \leq \frac{4L\norm{\bx_0 - \bx_*}^2}{(K+1)(K+2)} + \frac{8^{\frac{q}{2}}\norm{\bx_0 - \bx_*}^q(K+3)}{((K+1)(K+2)(K+3))^{\frac{q}{2}}} \delta,
\end{equation}
where both terms in the upper bound depend on the initial condition.

\textbullet \ \textbf{Absolute Error (AE)} The absolute error condition~\eqref{eq:ae},
has recently been investigated in a couple of works. Notably, \citet{ArtemVasin2023AcceleratedGradientMethods} study a variant of the accelerated algorithm known as the Similar Triangles Method (STM) \citep{Gasnikov2018UniversalMethodStochastic} with inexact gradients, i.e., (iSTM). In contrast to FGM, STM requires only one projection per iteration, making it preferable for constrained problems. For unconstrained optimization, it is shown by \citet{ArtemVasin2023AcceleratedGradientMethods} that iSTM achieves the convergence rate of
\begin{equation}\label{eq:Vasin-result}
    f(\by_K) - f_* \leq \frac{8L\norm{\bx_0 - \bx_*}^2}{K^2} + 3\tilde{R}_K \delta +\frac{K \delta^2}{2L},
\end{equation}
where \(\tilde{R}_K := \max_{0\leq j\leq K} \big\{\norm{\bz_j - \bx_*}, \norm{\bx_j - \bx_*}, \norm{\by_j - \bx_*}\big\}\).
However, this convergence bound is only given under fixed inexactness level {along the iterations} and restricts the stepsize as \(\alpha_k^2 = A_k\). Furthermore, \(\tilde{R}_K\) is not explicitly quantifiable under the absolute error assumption, since it is only shown to be bounded by \(\|\bx_0 - \bx_*\|\) \emph{under the exact gradient oracle}. When the feasible region is bounded with diameter \(D\), one can set \(\tilde{R}_K = D\), thereby recovering the rate similar to previous cases. This coincides with the observation that both the BIE and IFO-(\(\delta,L\)) conditions can be derived from \eqref{eq:ae} assumption under a bounded feasible region.

\textbullet \ \textbf{Relative Error (RE)~\cite{Polyak1987IntroductionOptimization}} The relative error \eqref{eq:re} is a stronger assumption as it requires the error to decrease with the gradient norm and enforces the gradient oracle to be more accurate near the stationary points in smooth unconstrained optimization.
\begin{align}
	\text{(relative error)} \qquad \   \norm{\tilde{\grad} f(\bx) - \grad f(\bx) } \leq \delta \norm{\grad f(\bx)}, \quad \text{for some } \delta\in[0,1]. \tag{RE} \label{eq:re}
\end{align}
Under this condition, \citet{Kornilov2023IntermediateGradientMethods} establish the boundedness of \(\tilde{R}_K\) in \eqref{eq:Vasin-result} for the iSTM algorithm.  Specifically, they show that with appropriate step size selection, \(\tilde{R}_K\) can be bounded by \(2 \|\bx_0 - \bx_*\|\), eliminating the need for separate error terms in the convergence rate. The resulting convergence rate is
\begin{equation*}
    f(\by_K) - f_* \leq \frac{16 a L \norm{\bx_0 -\bx_*}^2}{(K+2)^p},
\end{equation*}
where \(p \in [1, 2]\) and \(a = \cO\left(\max \left\{1, K^{\frac{p}{4}} \sqrt{\delta}, K^{\frac{p}{2}} \delta, K^p \delta^2\right\}\right)\) is a parameter that depends on the relative error level \(\delta\) and the iteration count \(K\).

The inexact assumptions discussed above and their corresponding convergence rates are summarized in \cref{tab:algorithm-comparasion}.

\begin{table}[hbt]
\centering
\caption{Summary of gradient inexactness error conditions and the corresponding complexity bounds for accelerated methods. (BFS): Bounded Feasible Set; (BG): Bounded Gradient; (SG): inexact gradient being a SubGradient.
The second column indicates the relation (necessity or sufficiency) of the absolute error (AE) condition for other error models in the presence of additional conditions. The \(\hat{u}_k\) in ``this work" are functions of the stepsize sequence \(\seq{\alpha_i}_{i=1}^K\).}
\label{tab:algorithm-comparasion}
\begin{tabular}{ccccc}
    \hline \up
    \multirow{2}{*}{Error condition}                              & \multirow{2}{*} {Relation to AE} & Allows changing error & \multirow{2}{*}{Iteration complexity}                                                                          \\
                                                                  &                                  & along iterations?     &                                                                                                                \\
    \hline
    \eqref{eq:bie}
    \cite{dAspremont2008SmoothOptimizationApproximate}            & AE+BFS\(\Rightarrow\)BIE           & \(\times\)              & \(\cO\left(\frac{1}{A_K} + \delta\right)\)                                                                       \\
    \eqref{eq:ifo} \cite{Devolder2013FirstOrderMethods}           & AE+BFS\(\Rightarrow\)IFO           & \checkmark            & \(\cO\left(\frac{1}{A_K} + \frac{\sum_{i=0}^K A_i \delta_{\bx_i}}{A_K} \right)\)                                 \\
    \eqref{eq:ifoq}~\cite{Nabou2024ProximalGradientMethods}       & AE+SG\(\Rightarrow\)IFO-q          & {\(\times\)}            & \(\cO\left(\frac{1}{K^2} + \frac{\delta}{K^{3q/2-1}}\right)\)                                                    \\
    \eqref{eq:re} \cite{Kornilov2023IntermediateGradientMethods}  & RE+BG\(\Rightarrow\)AE             & \(\times\)              & \(\cO\left(\max \{ \frac{1}{K^p}, \frac{\sqrt{\delta}}{K^{3p/4}}, \frac{\delta}{K^{p/2}}, \delta^2  \}\right) \) \\
    \eqref{eq:ae} \cite{ArtemVasin2023AcceleratedGradientMethods} &          /                        & \(\times\)              & \(\cO \left(\frac{1}{K^2} + \tilde{R}_K \delta +K\delta^2 \right) \)                                             \\
    \eqref{eq:ae}\textbf{(this work) }                            &          /                        & \checkmark            & \(\cO\left(\frac{1}{A_K} + \sum_{k=0}^{K-1}\hat{u}_k b_k^2 \right) \)
    \\\hline
\end{tabular}
\end{table}

Besides the aforementioned works, accelerated methods with inexact gradient oracles have also been discussed in other settings. The iFGM for the strongly convex setup is presented in \cite{Devolder2013FirstOrderMethodsa}, while \cite{Cohen2018AccelerationNoiseCorrupted} explores a variation of iFGM under the absolute error condition.

To leverage the tradeoff between convergence rate and accumulated error, the ``Intermediate Gradient Method" is introduced in \cite{Devolder2013IntermediateGradientMethods}, which is closely related to iFGM. This approach was subsequently generalized for various settings, see~\cite{Dvurechensky2016StochasticIntermediateGradient, DmitryKamzolov2021UniversalIntermediateGradient, FedorStonyakin2021InexactModelFramework, Gasnikov2019FastGradientDescent, Stonyakin2019GradientMethodsProblems}.

To analyze the effect of gradient inexactness on the convergence bound, \citet{Gannot2022FrequencyDomainAnalysis} and \citet{Aybat2020RobustAcceleratedGradient} examine nonaccelerated and accelerated algorithms under relative and absolute errors, respectively. The latter specifically addresses the balance between robustness to error and convergence rate--see also~\cite{gurbuzbalaban2023robustly}.  Other notable contributions include the analysis of absolute and relative inaccuracy for proximal point methods in \cite{Rockafellar1976MonotoneOperatorsProximal}, the study of proximal gradient methods with absolute error in \cite{Schmidt2011ConvergenceRatesInexact}, the inexact proximal gradient method for weakly convex functions with absolute error in \cite{Khanh2023InexactProximalMethods}, and the minimization of accumulated (controllable) error for optimal computational resource allocation in \cite{Dessel2024OptimalInexactnessSchedules}.

\section{Main results}\label{sec:major-result}
{This section contains our main theoretical results. Section~\ref{sec:main_thms} provides nonasymptotic convergence bounds of GOGM and GFGM algorithms with inexact gradient oracles. Section~\ref{sec:rate-error-tradeoff} exploits the established convergence bounds to obtain the optimal tradeoff between the convergence rate and accumulated error and to obtain the optimal inexactness schedule along the algorithm iterations to minimize the total \(\eta\)-complexity, i.e., \(\sum_{k=1}^K \eta_k\) in \eqref{eq:inexact-oracle} when \(r=1\).}

\subsection{ Main results for iGOGM and iGFGM}~\label{sec:main_thms}
The primary assumptions of the paper are summarized below:
\begin{assumption}\label{asump:basic-assumption}
    The objective function \(f:\mR^d \rightarrow \mR \cup + \infty \) is proper, closed, convex, \(L\)-Lipschitz smooth, and the solution set \(\cX_* := \argmin_\bx f(\bx)\) is nonempty.  The gradient estimate defined as \(\tilde{\grad} f(\bx) = \grad f(\bx) + \be_\bx\) satisfies the absolute error \eqref{eq:ae} condition
    \[
    \norm{\tilde{\grad}f(\bx) - \grad f(\bx)}^2 \leq b_\bx^2,
    \]
    where \(b_\bx\) is an arbitrary non-negative constant for each \(\bx\).
\end{assumption}

The theorem below provides the nonasymptotic convergence bound for the iGOGM algorithm. The optimality measure is discussed in Remark~\ref{rem:opt_measure}.
\begin{theorem}\label{thm:convergence-rate}
	Under \cref{asump:basic-assumption}, the sequence generated by iGOGM (\cref{alg:iGOGM}) satisfies

    \begin{equation*}
	f(\bx_K) - f_* - \frac{1}{2L}\norm{\grad f(\bx_K)}^2 	\leq  \frac{L\norm{\bx_0 - \bx_*}^2}{4A_K} + \sum_{k=0}^{K-1} \hat{u}_k b_k^2,
\end{equation*}
where \(\hat{\bu}\) is defined in \eqref{eq:optimal-u} and
\[
	\sum_{k=0}^{K-1} \hat{u}_k b_k^2 =  \sum_{j=1}^{K} \frac{\left[ A_{j-1} b_{j-1} + 2 \alpha_j \sum_{i=0}^{j-1}\alpha_i b_i \right]^2}{4 L A_K (A_j-\alpha_j^2)}.
\]
In particular, when \(b_i \equiv b\), \(\hat{u}_k\) reduces to
\begin{equation*}
		\hat{u}_k = \frac{A_k(1+2\alpha_{k+1})(A_k+2\alpha_k\alpha_{k+1})}{4LA_K(A_{k+1}-\alpha_{k+1}^2)}+ \sum_{i=k+1}^{K-1} \frac{A_i(1+2\alpha_{i+1})\alpha_k\alpha_{i+1}}{2L A_K(A_{i+1}- \alpha_{i+1}^2)}.
\end{equation*}

\end{theorem}
\begin{remark}\label{rem:opt_measure}
We use the same optimality measure \(f(\bx_K) - f_* - \frac{1}{2L} \norm{\grad f(\bx_K)}^2\) used in the analysis of exact GOGM, as it is still a meaningful one for the inexact problem. For iGFGM and iGOGM, we have
\begin{align*}
	f(\by_{K+1}) = &	f\left (\bx_K-\frac{1}{L}(\grad f(\bx_K) + \be_K)\right) \\
 \leq & f(\bx_K) -\frac{1}{L}\fprod{\grad f(\bx_K),\grad f(\bx_K) + \be_K} + \frac{1}{2L}\norm{\grad f(\bx_K)+\be_K}^2 \\
		=                                                             & f(\bx_K) - \frac{1}{2L} \norm{\grad f(\bx_K)}^2 + \frac{1}{2L}\norm{\be_K}^2.
	\end{align*}
 Comparing to the measure \(f(\bx_K) - f_* - \frac{1}{2L}\norm{\grad f(\bx_K)}^2\), the extra term \(\frac{1}{2L}\norm{\be_K}^2\) can be discarded since the constraint \(\norm{\be_K}^2 \leq b_K^2\) is the only one that involves \(\be_K\) and the PEP solution sets \(\norm{\be_K}^2=b_K^2\). As the algorithm only generates \(\seq{\bx_k}\) and \(\seq{\by_k}\) up to \(k=K\), \(\by_{K+1}\) is a dummy variable used for the theoretical guarantee, and setting \(b_K=0\) does not affect the complexity of the algorithm.
\end{remark}

The proof sketch of Theorem~\ref{thm:convergence-rate} is provided below. The proof contains a couple of crucial but non-intuitive steps, motivated by the PEP analysis, provided in {\cref{sec:PEP-inexact}}. The detailed proof is provided in {\cref{sec:Algorithm-proof}.}

\paragraph{Proof sketch.} The core strategy is to combine inequalities derived from Lipschitz smoothness and convexity conditions. By rearranging terms, applying specific multipliers, and creating a {telescoping sum}, we can isolate and bound the desired optimality measure. This process is guided by the analytical feasible solution identified in \cref{sec:iGOGM}. The key steps are as follows. First, we select two fundamental inequalities and assign their respective coefficients as

 {\fontsize{11.5pt}{12.5pt}\selectfont
 \begin{align*}
		 & f(\bx_{k+1}) - f(\bx_{k}) + \fprod{\grad f(\bx_{k+1}),\bx_{k} - \bx_{k+1}} + \frac{1}{2L}\norm{\grad f(\bx_k)- \grad f(\bx_{k+1})}^2 \leq 0 & \cdots \   & v_{k,k+1} = \frac{A_k}{A_K}         \\
		 & f(\bx_k) - f_* + \fprod{\grad f(\bx_k),\bx_* - \bx_k} + \frac{1}{2L}\norm{\grad f(\bx_k)}^2 \leq 0                                          &  \cdots  \ & v_{*,k} = \frac{A_k - A_{k-1}}{A_K}.
	\end{align*}
}
Summing these inequalities weighted by their multipliers and rearranging the terms to move \(f(\bx_K) - f_*\) on the left-hand side yields a key intermediate inequality. Subtracting \(\frac{1}{2L}\norm{\grad f(\bx_K)}^2\) from both sides gives us the optimality measure we seek to bound, which has the form
\begin{align*}
    &f(\bx_K) - f_* - \frac{1}{2L}\norm{\grad f(\bx_K)}^2                                                                                                                                                                  \\
		\leq  & - \sum_{k=0}^{K-1}\frac{A_k}{A_K}\fprod{\grad f(\bx_{k+1}),\bx_{k} - \bx_{k+1}} - \sum_{k=0}^K\frac{A_k -A_{k-1}}{A_K}\fprod{\grad f(\bx_k),\bx_* - \bx_k}   \\
        &     - \frac{1}{2L}\sum_{k=0}^{K-1}\frac{A_k}{A_K}\norm{\grad f(\bx_k)- \grad f(\bx_{k+1})}^2  - \frac{1}{2L}\sum_{k=0}^{K}\frac{A_k-A_{k-1}}{A_K} \norm{\grad f(\bx_k)}^2 - \frac{1}{2L}\norm{\grad f(\bx_K)}^2.
\end{align*}
By substituting the update rule of the iGOGM algorithm for \(\bx_k\), the inequality becomes
\begin{align*}
    & f(\bx_K) - f_* - \frac{1}{2L}\norm{\grad f(\bx_K)}^2   \\
    \leq &  \sum_{k=0}^{K}\frac{\alpha_k}{A_K} \fprod{\grad f(\bx_k),\bx_0 - \bx_*} - \frac{1}{L}\sum_{k=0}^{K}\frac{A_k}{A_K} \norm{\grad f(\bx_k)}^2 + \frac{1}{L}\sum_{k=0}^{K-1}\frac{A_k}{A_K}\fprod{\grad f(\bx_k),\grad f(\bx_{k+1})} \\
    &   - \frac{2}{L}\sum_{k=0}^{K-1}\sum_{i=0}^{k} \frac{\alpha_{k+1} \alpha_i}{A_K} \fprod{\grad f(\bx_{k+1}),\grad f(\bx_{i}) + \be_i} - \frac{1}{L}\sum_{k=0}^{K-1} \frac{A_k}{A_K}\fprod{\grad f(\bx_{k+1}), \grad f(\bx_k) + \be_k}.
\end{align*}

Next, our goal is to eliminate the inner products with the negative quadratic terms. To achieve this, we express the inner products as a difference of squared norms. For instance, the first inner-product can be bounded using the equality
\begin{align*}
    & \sum_{k=0}^{K}\frac{\alpha_k}{A_K}\fprod{\grad f(\bx_k), \bx_0 - \bx_*} \\
    = & \tau \norm{\bx_0 - \bx_*}^2 + \frac{1}{4 \tau}\norm{\sum_{k=0}^{K}\frac{\alpha_k}{A_K} \grad f(\bx_k)}^2 - \frac{1}{\tau}\norm{\tau(\bx_0-\bx_*) - \frac{1}{2} \sum_{k=0}^{K}\frac{\alpha_k}{A_K}\grad f(\bx_k)}^2.
\end{align*}
After some algebraic manipulation and setting \(\tau = \frac{L}{4A_K}\), we obtain
\begin{align*}
        &  f(\bx_K) - f_* - \frac{1}{2L}\norm{\grad f(\bx_K)}^2
     \leq  \frac{L\norm{\bx_0 - \bx_*}^2}{4A_K}  - \sum_{k=1}^K \frac{A_k - \alpha_k^2}{L A_K}\norm{\grad f(\bx_k)}^2  \\ & \qquad \qquad - \frac{2}{L}\sum_{k=0}^{K-1}\sum_{i=0}^{k} \frac{\alpha_{k+1} \alpha_i}{A_K} \fprod{\grad f(\bx_{k+1}), \be_i}  - \frac{1}{L}\sum_{k=0}^{K-1} \frac{A_k}{A_K}\fprod{\grad f(\bx_{k+1}),  \be_k}.
\end{align*}
We continue this procedure for the remaining inner-product terms involving the error vectors \(\be_k\). This is accomplished by introducing the quadratic term
\[
\frac{1}{LA_K(A_{k+1} - \alpha_{k+1}^2)}\norm{(A_{k+1} - \alpha_{k+1}^2)\grad f(\bx_{k+1}) + \sum_{i=0}^k\alpha_{k+1}\alpha_i \be_i + \frac{1}{2}A_k \be_k}^2.
\]
As this step is algebraically intensive, we omit the derivation and present the final inequality as
	\begin{align*}
      f(\bx_K) - f_* - \frac{1}{2L}\norm{\grad f(\bx_K)}^2 \leq  & \frac{LR^2}{4A_K}   +  \sum_{k=0}^{K-1} \hat{u}_k b_k^2,
	\end{align*}
where \(\hat{\bu}\) is given in \eqref{eq:optimal-u}.

The theorem below provides the nonasymptotic convergence bound for the iGFGM algorithm. The proof is provided in \cref{appx:iGFGM}.
\begin{theorem}\label{thm:convergence-iFGM}
    Under \cref{asump:basic-assumption}, the sequence generated by iGFGM (\cref{alg:iGFGM}) satisfies
	\begin{equation*}
        f(\bx_K) - f_* - \frac{1}{2L}\norm{\grad f(\bx_K)}^2 \leq \frac{L\norm{\bx_0 - \bx_*}^2}{2A_K} + \sum_{k=0}^{K-1}\hat{u}_k b_k^2,
    \end{equation*}
    where \(\hat{\bu}\) is defined in \eqref{eq:optimal-u-FGM}. In particular, when \(b_i \equiv b\), \(\hat{u}_k\) reduces to
    \begin{equation*}
        \hat{u}_k = \frac{A_k^2 (1+\alpha_{k+1})}{2L A_K(2 A_{k+1} - \alpha_{k+1}^2)} + \sum_{i=k+1}^{K} \frac{\alpha_k A_{i-1}\alpha_i (1+\alpha_i)}{2LA_K(2A_i - \alpha_i^2)}.
    \end{equation*}
\end{theorem}

A special case of the iGOGM algorithm is the iOGM-a algorithm presented in \cref{sec:rate-error-tradeoff}, which shows the error accumulation of $\cO(K^3)$ in the constant error regime. Following a question by one of the reviewers about the lower bound on error accumulation, we investigate the lower bound for error accumulation of first-order methods that satisfy two properties: 1. Affine error property, 2. Curvature-dependent error amplification property. For quadratic functions (which we know are smooth), general first-order methods satisfy the affine error property--see~\cref{prop:affine-error-response}. However, curvature-dependent error amplification is generally satisfied by iGOGM-type accelerated methods--see~\cref{thm:iGOGM-curvature-dependent-amplification}-- and not, e.g., by gradient descent (GD). We then manage to establish a lower bound for iGOGM-type methods on the quadratic function. This lower bound is presented informally below in Theorem~\ref{thm:lower-bound-informal}, and in detail in Theorem~\ref{thm:curvature-dependent-quadratic-lower-bound}.
\begin{theorem}[Curvature-dependent quadratic lower bound]
\label{thm:lower-bound-informal}
Let \(L>0\), \(b\geq 0\), \(K\geq 1\), and let \(d\geq 2\). Consider the iOGM-a method applied to the class of \(L\)-smooth
convex functions over \(\mR^d\) under the inexact-gradient oracle
\(\tilde{\grad} f(\bx)=\grad f(\bx)+\be\), where \(\norm{\be}\leq b\).
Then, under the small-curvature condition of
\cref{lem:small-curvature-amplification-iOGMa}, we have
\[
    \sup_{f\in\mathcal F_L(\mR^d)}
    \sup_{\norm{\be_i}\leq b}
    \left\{
        f(\bx_K)-f^\star
    \right\}
    \geq
    \Omega\left(\frac{b^2}{L}K^2\right).
\]
\end{theorem}

    The convergence bound guarantees in \cref{thm:convergence-rate,thm:convergence-iFGM} apply to the full generalized families iGOGM and iGFGM, and therefore provide a unified framework for quantifying the tradeoff between acceleration and error accumulation. However, the bound of $\cO(K^3)$ for iOGM-a compared to the lower bound of $\Omega(K^2)$ derived above shows the suboptimality of the derived bound for error accumulation.  
    Motivated by this gap, we further derive sharper error bounds for iFGM and iOGM in \cref{thm:ifgm-K-square-error-bound,thm:iogm-K-square-error-bound}. These refined bounds scale as \(\cO(K^2)\) under constant oracle error on error accumulation, and only sacrifice $O(1)$ constant in the decaying term $\cO(K^{-2})$ and, therefore, match the lower bound up to constants.

\begin{theorem}[Tighter bound for error accumulation of iFGM]
\label{thm:ifgm-K-square-error-bound}
For iFGM \cref{alg:iGFGM} with \(\lambda_k\equiv 1\), i.e.,
\(A_k=\alpha_k^2\), under \cref{asump:basic-assumption}, the convergence rate
can be bounded as
\[
    f(\bx_K)-f_*-\frac{1}{2L}\|\nabla f(\bx_K)\|^2  \leq \frac{4LR^2}{A_K}  + \frac{2c_{\rm H}}{L} \sum_{i=0}^{K-1}\alpha_{i+1}b_i^2,
\]
where \(c_{\rm H}>0\) is a constant independent of \(K\). In particular,
if \(b_i\equiv b\), then
\[
    f(\bx_K)-f_*-\frac{1}{2L}\|\nabla f(\bx_K)\|^2 \leq \cO\left(\frac{LR^2}{K^2}+\frac{K^2b^2}{L}\right).
\]
\end{theorem}
The proof is provided in \cref{appx:iFGM-K2}

\begin{theorem}[Tighter bound for error accumulation of iOGM]
\label{thm:iogm-K-square-error-bound}
For iOGM \cref{alg:iGOGM} with \(\lambda_k\equiv1\), i.e.,
\(A_k=\alpha_k^2\), under \cref{asump:basic-assumption}, the convergence rate
can be bounded as
\[
    f(\bx_K)-f_*-\frac{1}{2L}\|\nabla f(\bx_K)\|^2  \leq \frac{4LR^2}{A_K}  + \frac{8c_{\rm H}}{L}  \sum_{i=0}^{K-1}\alpha_{i+1}b_i^2,
\]
where \(c_{\rm H}\) is the same constant in \cref{thm:ifgm-K-square-error-bound}, independent of \(K\).
In particular, if \(b_i\equiv b\), then
\[
    f(\bx_K)-f_*-\frac{1}{2L}\|\nabla f(\bx_K)\|^2  \leq \cO\left( \frac{LR^2}{K^2} +  \frac{K^2b^2}{L}  \right).
\]
\end{theorem}
The proof can be found in \cref{appx:iOGM-K2}

    The above two theorems provide sharper accumulated-error bounds for iOGM and iFGM, matching the lower bound derived for this class of algorithms. These results for the specific algorithms suggest that the larger \(\cO(K^3b^2/L)\) term obtained from the general bounds may be reducible. We propose the following conjecture for the full iGOGM and iGFGM families.

\begin{conjecture}[Convergence bounds for iGOGM and iGFGM]
\label{conj:igogm-igfgm-sharp-error-scaling}
Under \cref{asump:basic-assumption}, consider either iGFGM in \cref{alg:iGFGM} or iGOGM in \cref{alg:iGOGM} with any feasible stepsize sequence \(\seq{\alpha_i}\). The convergence bounds of both generalized families are conjectured to satisfy the rate
\[
    f(\bx_K)-f_*-\frac{1}{2L}\|\nabla f(\bx_K)\|^2  \leq   \cO\left( \frac{LR^2}{A_K}  + \frac{\sum_{i=0}^{K-1}\alpha_{i+1} b_i^2}{L} \right)
\]
\end{conjecture}

\subsection{Exploiting the rate-error tradeoff and finding the optimal inexactness schedule}\label{sec:rate-error-tradeoff}
Based on the convergence bound of the iGOGM algorithm, this section first looks into the tradeoff between the convergence rate and accumulated error by changing the stepsize \emph{given a fixed gradient inexactness along iterations}. Second, we aim to optimize the inexactness levels along the iterations, for a given stepsize, so the total cost of the oracle that controls the inexactness is minimized. Similar analysis follows for the iGFGM algorithm, but it is not included for brevity.

\subsubsection{Convergence rate and the accumulated error tradeoff}
Considering the bound \(\frac{L\norm{\bx_0 - \bx_*}^2}{4A_K} + \sum_{k=0}^{K-1}u_k b_k^2\), with \({A_k},{u_k}\) being functions of \({\alpha_k}\), one can propose to minimize the bound with respect to \({\alpha_k}\). However, we note that such a problem is a high-order polynomial optimization and its solution is intractable in general.

Even if one is satisfied with the numerical solution to find the optimized algorithm, it is preferable to solve the minimax problem \eqref{prob:minmax} numerically following the relaxation approach in \cite{Drori2014PerformanceFirstOrder} as calculated in \cref{sec:numerical-optimized-algorithm}.
The resulting optimized stepsize sequence and convergence bound are better than the solution of the polynomial optimization problem

as \(\hat{\bu}\) in \cref{lem:optimal-u-fixed-tv} is optimal only after fixing \(\bv\) to the predefined values used in the exact OGM.

We narrow the iGOGM down to iOGM-a, with the stepsize \(\alpha_i = \frac{i+a}{a}\) and \(A_k = \sum_{i=0}^k\frac{i+a}{a}= \frac{(k+2a)(k+1)}{2a}\). The condition \(A_k > \alpha_k^2\)  holds for any \(k\) when \(a> 2\). Note that when \(K \gg a\), the convergence rate is still \(\cO(K^{-2})\) {in the exact case}. Furthermore, \cite{Kim2018GeneralizingOptimizedGradient} shows that the exact algorithm OGM-a has the asymptotic worst-case bound for function value as \(\frac{a}{2}K^{-2}\) and the smallest gradient norm as \(\frac{a \sqrt{6}}{2\sqrt{a-2}}K^{-1.5}\); it achieves the best performance for both measures with \(a=4\). The simplicity of the stepsize as well as its good performances over function value and gradient norm makes OGM-a a proper choice under exact oracles, and worth further study under inexact oracles as well.

For the constant error case \(b_k \equiv \bar{b}\), replacing the stepsizes of iOGM-\(a\) in \cref{thm:convergence-rate} gives the accumulated error \(\bar{b}^2\sum_{k=0}^{K-1}\hat{u}_k\). In the following analysis, we provide a uniform upper bound for any \(a>2\). With \(\alpha_k=(k+a)/a\) and \(A_k=\frac{(k+1)(k+2a)}{2a}\), we have
\[
    A_{k+1}-\alpha_{k+1}^2  = \frac{(k+1)((a-2)k+2a^2-2a-2)}{2a^2} \geq \frac{(a-2)(k+1)(k+2a)}{2a^2}.
\]
Hence
\(
\frac{A_k}{A_{k+1}-\alpha_{k+1}^2}\leq \frac{a}{a-2}.
\)
Using this relation in \eqref{eq:theoretical-u}, together with
\[
    1+2\alpha_{k+1}\leq \frac{3(K+a)}{a}, \quad  \alpha_{k+1}\leq \frac{K+a}{a}, \quad   A_k+2\alpha_k\alpha_{k+1}\leq 3A_K,
\]
we obtain
\begin{equation}\label{eq:sum-u-iOGM-a}
    \sum_{k=0}^{K-1}\hat{u}_k  \leq \frac{4K(K+a)^2}{La(a-2)}.
\end{equation}
Therefore,
\[
    \bar{b}^2\sum_{k=0}^{K-1}\hat{u}_k  \leq \frac{4\bar{b}^2K(K+a)^2}{La(a-2)}  = \cO\left(\frac{\bar{b}^2K(K+a)^2}{La(a-2)}\right),
\]
where the hidden constant is independent of \(K\) and \(a\).

The above bound provides a general description of the dependence on both \(K\) and \(a\). For any fixed \(2<a \ll K\), the accumulated error is \(\cO(\bar{b}^2K^3/L)\). When \(K\ll a\), the bound reduces to \(\cO(\bar{b}^2K/L)\) up to the factor \(a/(a-2)\). To illustrate the actual values of \(\hat{u}_k\), we provide a numerical result of \(\hat{u}_k\) for a given \(K\) with different \(a\) in \cref{fig:uk-vs-a}. For the same \(a\) and \(k\), bigger \(K\) results in bigger \(\hat{u}_k\). For each \(a\) value, \(\hat{u}_k\) first increases and then decreases with \(k\). For a fixed \(k\) and \(K\), bigger \(a\) results in smaller \(\hat{u}_k\) values.

\begin{figure}[hbt]
	\centering
	\includegraphics[width=0.8\textwidth]{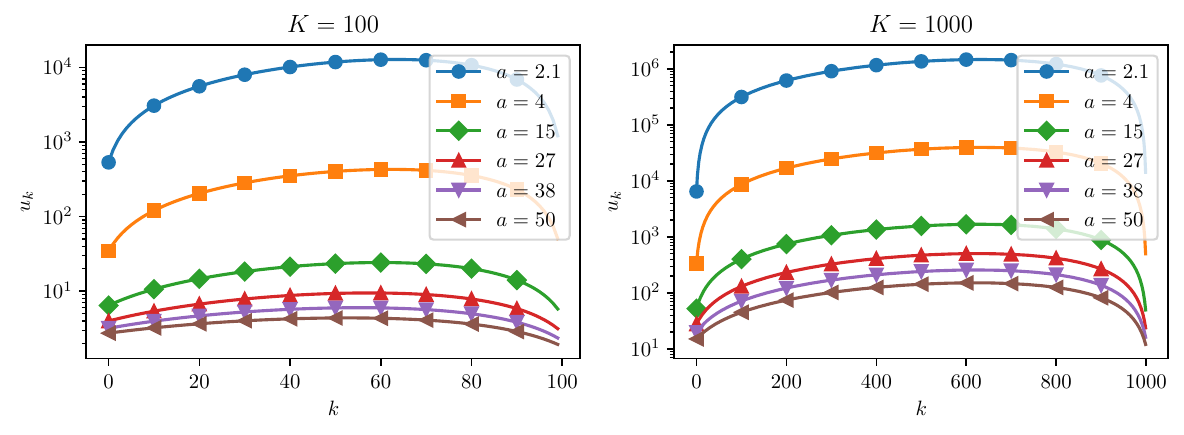}
	\caption{Values of \(\hat{u}_k\) for OGM-\(a\) with different values of \(a\) and \(K\) under constant \(b\).}\label{fig:uk-vs-a}

\end{figure}

\begin{figure}[hbt]
	\centering
	\includegraphics[width=0.8\textwidth]{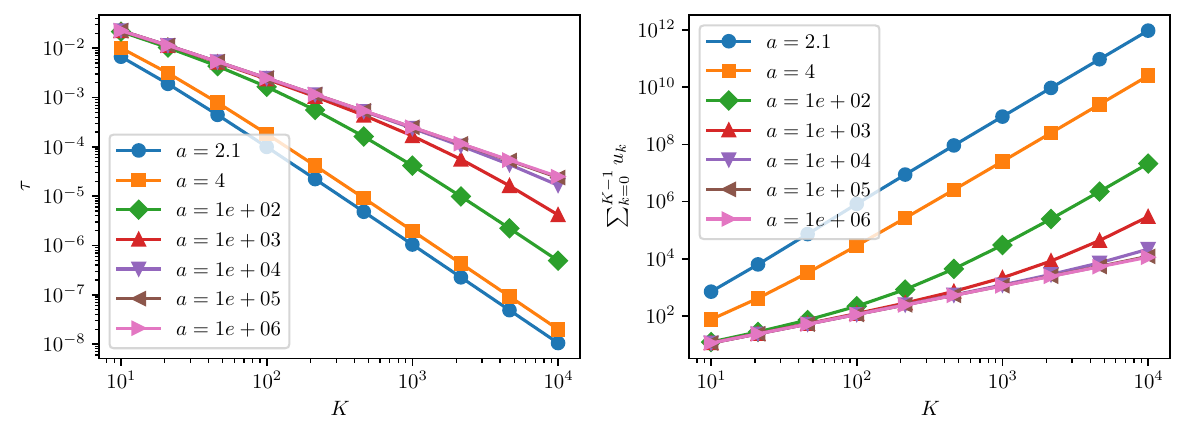}
	\caption{Convergence rate and accumulated error for iOGM-\(a\) for different \(a\) and \(K\) values.}\label{fig:tau-u-vs-a-K}
\end{figure}

Based on this result, to reduce the accumulated error, one can consider increasing \(a\). However, this will adversely slow down the convergence rate as \(A_K = \frac{(K+2a)(K+1)}{2a}\) and decreases when \(a\) is increased. We provide the comparison of convergence rate and accumulated error with different iteration numbers \(K\), and \(a\) in \cref{fig:tau-u-vs-a-K}, where \(\tau = \frac{L}{4A_K}\). The convergence rate \(\tau\) is smaller for smaller \(a\) while the accumulated error is bigger.
This result illustrates the tradeoff shown in \eqref{eq:sum-u-iOGM-a}.
When \(a = 4\), \(\sum_{k=0}^{K-1} u_k\) is between \(K^2\) and \(K^3\), better than our relaxed analysis result. Another interesting fact is that when \(K \ll a\) (\(a=10^5\) and \(a = 10^6\)), the convergence rate \(\tau = \cO(K^{-1})\) and the accumulate error is \(\cO(K^1)\). If \(K \ll a\), we have \(\alpha_k = 1\), and the proposition below provides the convergence bound of this extreme scenario.

\begin{proposition}
	With the stepsize  defined as \(\alpha_k \equiv 1\), the sequence generated by iGOGM (\cref{alg:iGOGM}) satisfies
	\begin{equation*}
		f(\bx_K) - f_* - \frac{1}{2L}\norm{\grad f(\bx_K)}^2 \leq \frac{L\norm{\bx_0 - \bx_*}^2}{4(K+1)} + \sum_{k=0}^{K-1} \frac{3(2K-k+1)}{4L(K+1)} \norm{\be_k}^2.
	\end{equation*}
	Furthermore, when \(\norm{\be_k} \leq \bar{b}, \forall k\), the above inequality becomes
	\begin{equation*}
		f(\bx_K) - f_* - \frac{1}{2L}\norm{\grad f(\bx_K)}^2 \leq \frac{L\norm{\bx_0 - \bx_*}^2}{4(K+1)} + \frac{9K}{8L} \bar{b}^2.
	\end{equation*}
\end{proposition}

\subsubsection{Optimal inexactness schedule}
Since the derived \(\hat{u}_k\) is not constant, it will be beneficial to set \(b_k\) based on the value of \(\hat{u}_k\). Recall the definition of \(\mathbb{\eta}\) from \eqref{eq:inexact-oracle}, when \(r=1\), we can minimize the total \(\eta\)-complexity, i.e., \(\sum_{k=0}^{K-1}\eta_k\), so that the accumulated error does not exceed the convergence rate as
\begin{align*} \tag{OPT-b}\label{tag:optimal-eta}
	\min_{\bb\geq \mathbf{0}} \quad  \sum_{k=0}^{K-1}\eta_k =  \sum_{k=0}^{K-1} h^{-1}(b_k)
	\quad \st \quad                        \sum_{k=0}^{K-1}\hat{u}_k b_k^2 \leq \frac{LR^2}{4A_K}.
\end{align*}
This is a convex optimization problem if \(h\) is convex and decreasing over a convex quadratic constraint, and at its optimal solution, the constraint is binding. Hence, the inequality constraint can be replaced with the equality. Let \(\lambda\) be the dual multiplier of the constraint, the Lagrangian function is
\[
	\cL(\bb,\lambda) = \sum_{k=0}^{K-1}h^{-1}(b_k) + \lambda \left( \sum_{k=0}^{K-1} \hat{u}_k b_k^2 - \frac{LR^2}{4A_K}\right).
\]
At the optimal solution \(\bb^*\) and \(\lambda^*\), from the optimality conditions, we have
\begin{align}
	\grad _{b_k} h^{-1}(b_k^*) + 2\lambda^* b_k^* u_k    & = 0 , \quad k = 0,1\cdots,K-1, \label{eq:KKT-1}\\
	\sum_{k=0}^{K-1}u_k (b_k^*)^2 - \frac{LR^2}{4A_K} & = 0. \label{eq:KKT-2}
\end{align}

Notice the coefficients \(\hat{u}_k\) in \eqref{eq:optimal-u} depend on the prescribed inexactness vector \(\bb\) through
\(
    \hat{u}_i(\bb)=\frac{(\bH\bb)_i}{b_i}.
\)
Hence, if \(\hat{\bu}\) is re-optimized while optimizing over \(\bb\), the accumulated error term becomes
\[
    \sum_{i=0}^{K-1}\hat{u}_i(\bb)b_i^2=\bb^\top\bH\bb,
\]
and the scheduling problem is no longer separable. The corresponding KKT condition involves
\[
    \grad_{b_k}h^{-1}(b_k^*)+2\lambda^*(\bH\bb^*)_k=0.
\]
Since \((\bH\bb^*)_k\) depends on the entire vector \(\bb^*\), this condition generally does not yield closed-form schedules.
To retain explicit formulas, we therefore fix \(\bu\) independently of \(\bb\), using the value obtained from the constant-error case.
Under this situation, we give closed-form expressions for \(\bb^*\) for two examples of \(h(\eta)\) \cite{Liu2023AdaptiveStochasticOptimization}, which are two common structures of the inexact gradient oracles.

\begin{lemma}[\textbf{\(h(\eta)\) with power law decay}]\label{lem:power-law}
	Let \(h(\eta) = c_1\eta^{-c_2}\) with \(c_1,c_2>0\) to be two constants. The inverse function and its gradient are
	\[
		h^{-1}(b) = 	\left(\frac{c_1}{b}\right)^{1/c_2} \quad\text{and} \quad \grad _b h^{-1}(b) = -\frac{c_1^{1/c_2} b^{-(1+c_2)/c_2}}{c_2}.
	\]
	Furthermore, the optimal primal and dual solutions of \eqref{tag:optimal-eta} are
	\begin{align*}
	    b_k^* = & \frac{\sqrt{L}R }{2\sqrt{A_K \sum_{k=0}^{K-1} u_k^{1/(1+2c_2)} }u_k^{c_2/(1+2c_2)}} , \\
        \lambda^* = & \frac{c_1^{1/c_2}}{2c_2}\left(\frac{LR^2}{4A_K}\right)^{-(1+2c_2)/(2c_2)}\left(\sum_{k=0}^{K-1} u_k^{1/(1+2c_2)}\right)^{(1+2c_2)/(2c_2)}.
	\end{align*}
\end{lemma}
\begin{proof}{Proof}
With power law decay, \eqref{eq:KKT-1} is written as
\begin{equation*}
        - \frac{c_1^{1/c_2} (b_k^*)^{-(1+c_2)/c_2}}{c_2} + 2\lambda^* b_k^* u_k = 0, \quad k=0,1,\cdots, K-1.
\end{equation*}
From this equation, \(b_k^*\) can be calculated as
\begin{equation}\label{eq:power-bk}
b_k^* = \left(\frac{2\lambda^* u_k c_2}{c_1^{1/c_2}}\right)^{- \frac{c_2}{1+2c_2}}.
\end{equation}
Inserting this equation into \eqref{eq:KKT-1}, and solve for \(\lambda^*\) we get
\begin{equation*}
    \lambda^* = \frac{c_1^{1/c_2}}{2c_2}\left(\frac{LR^2}{4A_K}\right)^{-(1+2c_2)/(2c_2)}\left(\sum_{k=0}^{K-1} u_k^{1/(1+2c_2)}\right)^{(1+2c_2)/(2c_2)}.
\end{equation*}
Replacing this solution into \eqref{eq:power-bk}, the final result can be derived.
\Halmos \end{proof}

\begin{lemma}[\textbf{\(h(\eta)\) with exponential decay}]\label{lem:exp-decay}
	Let \(h(\eta) = q_1 q_2^{-\eta}\) with \(q_1>0\) and \(q_2>1\) be two constants. Then
	\[
		h^{-1}(b) = \frac{\log q_1 - \log b}{\log q_2}, \quad \grad_b h^{-1}(b) = -\frac{1}{b \log q_2},
	\]
	and, the optimal primal and dual solutions are
	\begin{equation*}
		b_k^* =  \sqrt{\frac{LR^2}{4KA_K u_k}}, \quad \lambda^* = \frac{2KA_K}{LR^2 \log q_2}.
	\end{equation*}
\end{lemma}

\begin{proof}{Proof}
    Similar to the proof of \cref{lem:power-law}, from \eqref{eq:KKT-1} we have
    \begin{equation}
        u_k (b_k^*)^2 = \frac{1}{2\lambda^* \log q_2}.
    \end{equation}
    Inserting this into \eqref{eq:KKT-2}, we have
    \begin{equation*}
        \lambda^* = \frac{2K A_K}{LR^2 \log q_2}.
    \end{equation*}
    By the previous two equations, we can derive the final solution.
\Halmos \end{proof}

To quantify the improvement of the optimized inexactness level, we calculate the total \(\eta\)-complexity for iOGM-4 algorithm for the two structures of \(h(\eta)\) functions. We fix \(L=1\) and \(R=1\), then solve the constant \(\bar{b}\) through
\begin{equation*}
    \bar{b}^2 \sum_{k=0}^{K-1}\hat{u}_k \leq \frac{LR^2}{4A_K},
\end{equation*}
and calculate the optimal \(\bar{b}\) by the solutions we derived in \cref{lem:power-law,lem:exp-decay}.
In \cref{fig:eta-pow}, we illustrate the result for power law decay with four different values of \(c_2\). Note that the \(\eta\)-complexity decreases by increasing \(c_2\). If we focus on the improvement of \(\eta\)-complexity for the optimized \(\seq{b_k}\), the improvement is in the same order as the complexity of constant \(b\), which shows that optimizing over \(b\) can decrease the \(\eta\)-complexity significantly.
As for the exponential decay in \cref{fig:eta-exp}, the improvement is not as large as that of the power law decay. Although the improvement is not as large as that of the power law decay, considering its absolute value, it shows a considerable improvement.

\begin{figure}[hbt]
	\centering
	\includegraphics[width=0.8\textwidth]{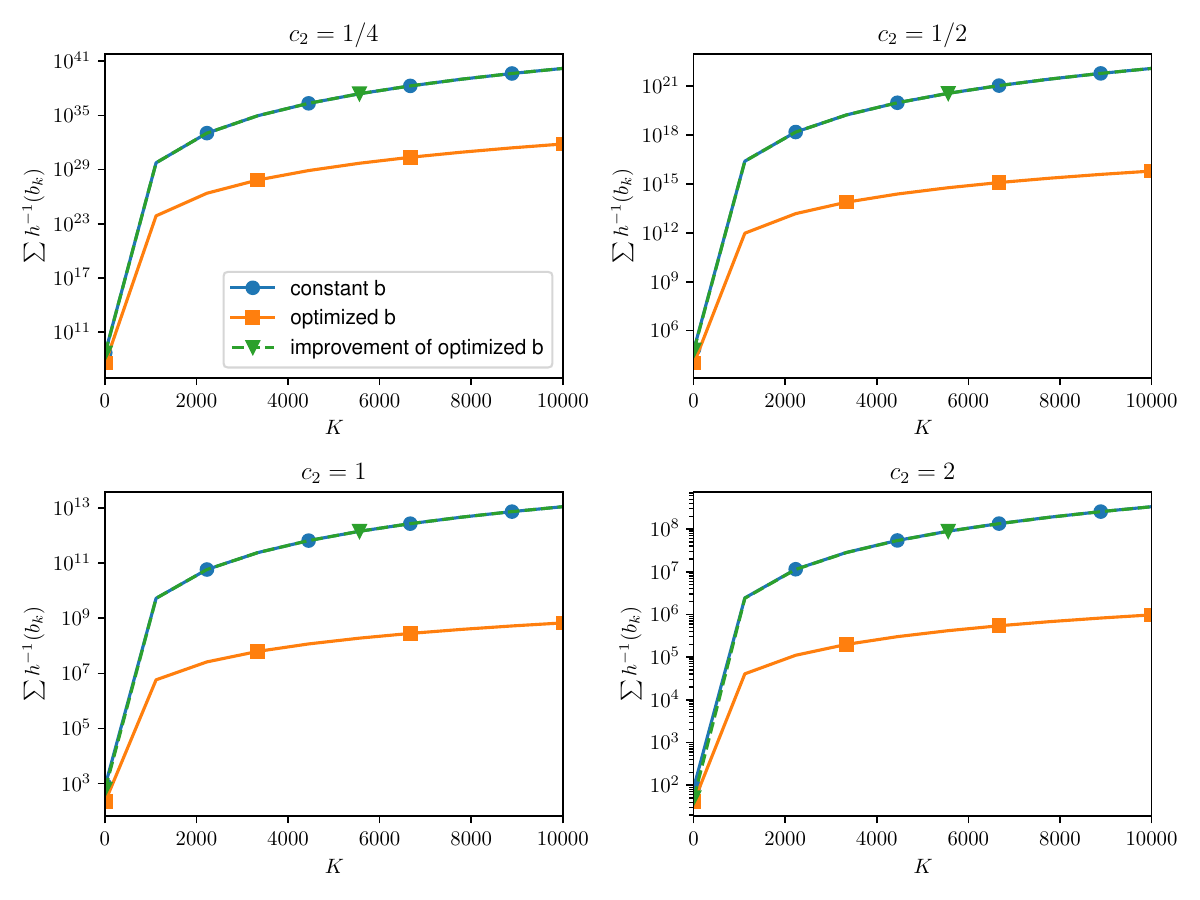}
	\caption{Total \(\eta\)-complexity of iOGM-4 for \(h(\eta)\) with power law decay with \(c_1=1\) and \(L=1, R=1 \).}\label{fig:eta-pow}
\end{figure}

\begin{figure}[hbt]
	\centering
	\includegraphics[width=0.4\textwidth]{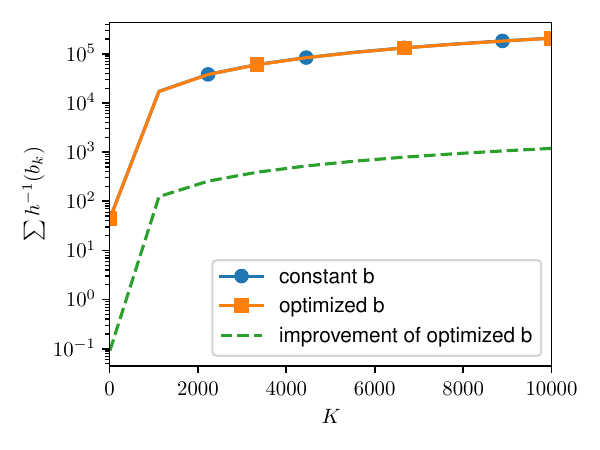}
	\caption{Total \(\eta\)-complexity of iOGM-4 for \(h(\eta)\) with exponential decay with \(q_1=1,q_2=e\) and \(L=1,R=1\).}\label{fig:eta-exp}
\end{figure}

\section{A lower-bound mechanism through affine error amplification}
\label{sec:lower-bound}
We next describe a simple lower-bound mechanism for fixed-coefficient
first-order methods under the absolute inexact-gradient oracle. The main idea
is to quantify how oracle errors propagate through the algorithm when the
method is applied to quadratic functions.

\begin{definition}[Affine error-response property]
Consider a fixed-coefficient first-order method applied to an inexact gradient
oracle of the form \eqref{eq:ae}.
Let \(\bx_K\) denote the final iterate generated by the method under the error
sequence \(\be_0,\ldots,\be_{K-1}\). Let \(\bx_K^{\rm ex}\) denote the final iterate
generated by the same method, from the same initial point and on the same
function, but under the exact-gradient oracle, i.e., with \(\be_i=\bo\) for all
\(i=0,\ldots,K-1\).

We say that the method has the affine error-response property at iteration \(K\) for a given function if there exist linear maps
\(
    \mathbf{W}_{K,0},\ldots,\mathbf{W}_{K,K-1}
\)
such that, for every admissible error sequence,
\[
    \bx_K-\bx_K^{\rm ex}
    = -\frac{1}{L}
    \sum_{i=0}^{K-1}\mathbf{W}_{K,i}\be_i .
\]
The maps \(\mathbf{W}_{K,i}\) describe how the oracle error at
iteration \(i\) propagates to the final iterate \(\bx_K\).
\end{definition}

Note that while the General First-Order (GFO) representation of first-order methods provides an affine representation of iterates in terms of oracle outputs, iterates are \emph{not} affine in the oracle's error for general smooth functions. However, as for quadratic functions, the gradient map is linear in iterates, the affine error-response property holds for all first-order methods on quadratic objectives.
\begin{proposition}[Affine error response on quadratic functions]
\label{prop:affine-error-response}
Consider a fixed-coefficient first-order method whose updates are obtained
through fixed linear combinations of previous iterates, auxiliary sequences,
and inexact gradients.
Let
\(
    f(\bx)=\frac{1}{2}\bx^\top \mathbf Q \bx,
    \bo\preceq \mathbf Q \preceq L\mathbf I.
\)
Then the method satisfies the affine error-response property on \(f\).
\end{proposition}
\begin{proof}{Proof}
For the quadratic function above, the gradient map is linear
\(
    \grad f(\bx)=\mathbf Q\bx.
\)
Under the inexact oracle,
\(
    \tilde{\grad} f(\bx)=\mathbf Q\bx+\be.
\)
Since the method coefficients are fixed and the gradient map is linear, the
entire algorithmic recursion becomes a linear dynamical system driven by the
errors \(\be_0,\ldots,\be_{K-1}\). Subtracting the exact-gradient recursion from
the inexact-gradient recursion gives a linear recursion for the perturbations
\(
    \delta \bx_k:=\bx_k-\bx_k^{\rm ex}.
\)
Unrolling this linear recursion yields linear maps
\(\mathbf{W}_{K,0}(\mathbf Q),\ldots,\mathbf{W}_{K,K-1}(\mathbf Q)\) such that
\(
    \delta \bx_K
    = -\frac{1}{L}
    \sum_{i=0}^{K-1}\mathbf{W}_{K,i}(\mathbf Q)\be_i,
\)
which is the affine error-response property. \Halmos
\end{proof}

\begin{definition}[Curvature-dependent error-amplification quantity]
Suppose that the affine error-response property holds at iteration \(K\) with
linear maps \(\mathbf{W}_{K,0},\ldots,\mathbf{W}_{K,K-1}\). The corresponding
error-amplification quantity is defined by
\[
    \mathcal E_K
    := \sup_{\norm{\mathbf d_i}\leq 1}
    \left\|
        \sum_{i=0}^{K-1}\mathbf{W}_{K,i}\mathbf d_i
    \right\|.
\]

Consider the scalar quadratic family
\(
    f_\lambda(x)=\frac{\lambda}{2}x^2, 0\leq \lambda\leq L.
\)
When the error-response maps depend on the
curvature \(\lambda\), we write them as \(\mathbf{W}_{K,i}(\lambda)\) and, hence,
\(
    \mathcal E_K(\lambda)
    = \sup_{\norm{\mathbf d_i}\leq 1}
    \left\|
        \sum_{i=0}^{K-1}\mathbf{W}_{K,i}(\lambda)\mathbf d_i
    \right\|.
\)
In the one-dimensional case, where \(\mathbf{W}_{K,i}(\lambda)\) reduce to scalar
coefficients \(w_{K,i}(\lambda)\), this quantity becomes
\(
    \mathcal E_K(\lambda)
    = \sum_{i=0}^{K-1}\abs{w_{K,i}(\lambda)}.
\)
\end{definition}

\begin{theorem}[Curvature-dependent amplification for iGOGM on scalar quadratics]
\label{thm:iGOGM-curvature-dependent-amplification}
Consider the scalar quadratic function
\(
    f_\lambda(x)=\frac{\lambda}{2}x^2,
    0\leq \lambda\leq L,
\)
and define \(\rho:=\lambda/L\). Apply iGOGM to this problem with inexact gradients
\(
    \tilde{\grad} f_\lambda(x_k)=\lambda x_k+e_k=L\rho x_k+e_k,
    |e_k|\leq b,
\)
and let
\(
    \beta_{k+1}:=\alpha_{k+1}/A_{k+1}.
\)
Then iGOGM has the scalar affine error-response representation
\(
    x_K-x_K^{\rm ex}
    = -\frac{1}{L}
    \sum_{i=0}^{K-1}w_{K,i}(\rho)e_i,
\)
where, for each fixed \(i=0,\ldots,K-1\), the coefficients
\(w_{k,i}(\rho)\) and \(q_{k,i}(\rho)\) are initialized by
\(
    w_{i+1,i}(\rho)
    = 1-\beta_{i+1}+2\alpha_i\beta_{i+1},
    \text{and } q_{i+1,i}(\rho)
    = 2\alpha_i,
\)
and then propagated, for \(k=i+1,\ldots,K-1\), according to
\[
    \begin{bmatrix}
        w_{k+1,i}(\rho)\\[1mm]
        q_{k+1,i}(\rho)
    \end{bmatrix}
    = \mathbf M_k(\rho)
    \begin{bmatrix}
        w_{k,i}(\rho)\\[1mm]
        q_{k,i}(\rho)
    \end{bmatrix},
\ \text{where,} \ \ 
    \mathbf M_k(\rho)
    = \begin{bmatrix}
        (1-\beta_{k+1})(1-\rho)-2\alpha_k\beta_{k+1}\rho
        &
        \ \beta_{k+1}
        \\[1mm]
        -2\alpha_k\rho
        &
        \ 1
    \end{bmatrix}.
\]
Consequently, the curvature-dependent scalar error-amplification quantity is
\(
    \mathcal E_K(\rho)
    = \sum_{i=0}^{K-1}\abs{w_{K,i}(\rho)},
\)
where
\[
    w_{K,i}(\rho)
    = \begin{bmatrix}1 & 0\end{bmatrix}
    \mathbf M_{K-1}(\rho)\cdots \mathbf M_{i+1}(\rho)
    \begin{bmatrix}
        1-\beta_{i+1}+2\alpha_i\beta_{i+1}\\[1mm]
        2\alpha_i
    \end{bmatrix}.
\]
\end{theorem}
The proof can be found in \cref{appx:proof-curvature-dependent}

Below, we show that iGOGM (and its special form iOGM-a) amplifies the error quadratically \(K^2\), in zero and small curvature regimes. We note that this error amplification does not happen, e.g., by gradient descent (GD) method.
\begin{lemma}[Zero-curvature amplification for iGOGM]
\label{lem:zero-curvature-amplification-iGOGM}
For iGOGM on the scalar quadratic family in
\cref{thm:iGOGM-curvature-dependent-amplification}, when \(\rho=0\), the
coefficients satisfy
\(
    w_{K,i}(0)
    = 2\alpha_i
    + (1-2\alpha_i)\frac{A_i}{A_K},
    i=0,\ldots,K-1.
\)
Consequently,
\(
    \mathcal E_K(0)
    = \sum_{i=0}^{K-1}
    \left|
        2\alpha_i
        + (1-2\alpha_i)\frac{A_i}{A_K}
    \right|.
\)
\end{lemma}
The proof is in \cref{appx:proof-zero-iGOGM}.

\begin{lemma}[Zero-curvature amplification for iOGM-\(a\)]
\label{lem:zero-curvature-amplification-iOGMa}
Consider iOGM-\(a\), where
\(
    \alpha_k=(k+a)/a,\
    A_k=\sum_{i=0}^k\alpha_i=(k+1)(k+2a)/2a.
\)
Then
\(
    \mathcal E_{K,a}(0)=\Theta(K^2),
\)
i.e,, there exist constants \(c_a,C_a>0\), depending only on \(a\), such
that for all sufficiently large \(K\),
\(
    c_aK^2
    \leq \mathcal E_{K,a}(0)
    \leq C_aK^2.
\)
\end{lemma}

For the proof, see \cref{appx:zero-iOGMa}. The small-curvature assumption below says that the large amplification seen at \(\lambda=0\) persists for a small positive curvature \(\lambda=\mu L/K^2\).


\begin{lemma}[Small-curvature amplification for iOGM-\(a\)]
\label{lem:small-curvature-amplification-iOGMa}
Fix \(a>0\). There exist constants \(\mu_a>0\), \(c_a>0\), and
\(K_a\ge 1\), depending only on \(a\), such that for all \(K\ge K_a\)
and all \(0<\mu\le \mu_a\),
\[
    \mathcal E_{K,a}\left(\frac{\mu}{K^2}\right)
    =
    \sum_{i=0}^{K-1}
    \left|
        w_{K,i}\left(\frac{\mu}{K^2}\right)
    \right|
    \ge
    c_aK^2 .
\]
Equivalently, for \(\lambda=\mu L/K^2\), we have
\(\mathcal E_{K,a}(\lambda)\ge c_aK^2\).
\end{lemma}
Proof is in \cref{appx:proof-amplification-iOGMa}

\begin{theorem}[Curvature-dependent quadratic lower bound]
\label{thm:curvature-dependent-quadratic-lower-bound}
Let \(L>0\), \(b\geq 0\), \(K\geq 1\), and let \(d\geq 2\). Consider a
fixed-coefficient iGOGM-type method applied to the class of \(L\)-smooth
convex functions over \(\mR^d\) under the inexact-gradient oracle
\(\tilde{\grad} f(\bx)=\grad f(\bx)+\be\), where \(\norm{\be}\leq b\).
Let \(\mathcal E_K(\rho)\) denote the scalar curvature-dependent
error-amplification quantity from
\cref{thm:iGOGM-curvature-dependent-amplification}, with
\(\rho=\lambda/L\). Then
\[
    \sup_{f\in\mathcal F_L(\mR^d)}
    \sup_{\norm{\be_i}\leq b}
    \left\{
        f(\bx_K)-f^\star
    \right\}
    \geq
    \sup_{0\leq \lambda\leq L}
    \frac{\lambda b^2}{2L^2}\mathcal E_K(\lambda/L)^2 .
\]
In particular, for iOGM-\(a\), under the small-curvature condition of
\cref{lem:small-curvature-amplification-iOGMa}, we have
\[
    \sup_{f\in\mathcal F_L(\mR^d)}
    \sup_{\norm{\be_i}\leq b}
    \left\{
        f(\bx_K)-f^\star
    \right\}
    \geq
    \Omega\left(\frac{b^2}{L}K^2\right).
\]
\end{theorem}

\begin{proof}{Proof}
Fix \(\lambda\in[0,L]\) and consider the rank-one quadratic function
\[
    f_\lambda(\bx)
    =
    \frac{\lambda}{2}\langle \bx,\boldsymbol{\delta}_1\rangle^2,
    \qquad \bx\in\mR^d,
\]
where \(\boldsymbol{\delta}_1\) denotes the first coordinate vector. Then
\(\nabla^2 f_\lambda(\bx)=\lambda \boldsymbol{\delta}_1\boldsymbol{\delta}_1^\top\),
and hence \(f_\lambda\) is convex and \(L\)-smooth for every
\(0\leq\lambda\leq L\). Moreover, since \(d\geq 2\), \(f_\lambda\) is not
strongly convex whenever \(\lambda>0\). Also, \(f_\lambda^\star=0\) and
\(\argmin f_\lambda=\{\bx\in\mR^d:\langle \bx,\boldsymbol{\delta}_1\rangle=0\}\).

Initialize the method at \(\bx_0=\bo\). Under the exact-gradient oracle, the
trajectory remains at zero, and hence \(\bx_K^{\rm ex}=\bo\). Now restrict the
oracle errors to the first coordinate direction, namely
\(\be_i=e_i\boldsymbol{\delta}_1\) with \(|e_i|\leq b\). Since the initialization,
the gradient, and the oracle errors all lie in
\(\operatorname{span}\{\boldsymbol{\delta}_1\}\), all iterates remain in this
one-dimensional subspace. Therefore, the induced dynamics along
\(\boldsymbol{\delta}_1\) are exactly the scalar quadratic dynamics for
\(x\mapsto \lambda x^2/2\).

By the scalar affine error-response representation, the first coordinate of
\(\bx_K\) satisfies
\[
    \langle \bx_K,\boldsymbol{\delta}_1\rangle
    =
    -\frac{1}{L}
    \sum_{i=0}^{K-1}w_{K,i}(\lambda/L)e_i .
\]
Choose the adversarial scalar errors as
\(e_i=b\,\operatorname{sign}(w_{K,i}(\lambda/L))\), with an arbitrary sign
when \(w_{K,i}(\lambda/L)=0\). Then \(\norm{\be_i}=|e_i|\leq b\), and
\[
    \left|\langle \bx_K,\boldsymbol{\delta}_1\rangle\right|
    =
    \frac{b}{L}
    \sum_{i=0}^{K-1}\abs{w_{K,i}(\lambda/L)}
    =
    \frac{b}{L}\mathcal E_K(\lambda/L).
\]
Consequently,
\[
    f_\lambda(\bx_K)-f_\lambda^\star
    =
    \frac{\lambda}{2}\langle \bx_K,\boldsymbol{\delta}_1\rangle^2
    =
    \frac{\lambda b^2}{2L^2}\mathcal E_K(\lambda/L)^2.
\]
Since \(f_\lambda\in\mathcal F_L(\mR^d)\) and the above construction is valid
for every \(\lambda\in[0,L]\), taking the supremum over \(\lambda\) gives
\[
    \sup_{f\in\mathcal F_L(\mR^d)}
    \sup_{\norm{\be_i}\leq b}
    \left\{
        f(\bx_K)-f^\star
    \right\}
    \geq
    \sup_{0\leq \lambda\leq L}
    \frac{\lambda b^2}{2L^2}\mathcal E_K(\lambda/L)^2.
\]

Now specialize to iOGM-\(a\) and choose \(\lambda=\mu L/K^2\), where
\(0<\mu\leq \mu_a\). By
\cref{lem:small-curvature-amplification-iOGMa},
\[
    \mathcal E_{K,a}\left(\frac{\mu}{K^2}\right)\geq c_aK^2.
\]
Substituting this bound and \(\lambda=\mu L/K^2\) into the previous display
yields
\[
    f_\lambda(\bx_K)-f_\lambda^\star
    \geq
    \frac{\mu L}{K^2}\frac{b^2}{2L^2}c_a^2K^4
    =
    \frac{\mu c_a^2}{2}\frac{b^2}{L}K^2.
\]
Therefore,
\[
    \sup_{f\in\mathcal F_L(\mR^d)}
    \sup_{\norm{\be_i}\leq b}
    \left\{
        f(\bx_K)-f^\star
    \right\}
    \geq
    \Omega\left(\frac{b^2}{L}K^2\right).
\]
This completes the proof. \Halmos
\end{proof}

\section{Optimized Gradient Method with inexact gradient oracle}\label{sec:PEP-iGOGM}
In this section, we introduce the analytical tool used to derive our quantifiable convergence bounds, i.e., Performance Estimation Problem (PEP). Proposed by \citet{Drori2014PerformanceFirstOrder}, PEP is a technique that formulates an optimization problem to find the worst-case performance of an algorithm on a given class of functions. Inspired by the derivation of the Optimized Gradient Method (OGM), which relies on this tool, our analysis will first focus on iGOGM and then be extended to iGFGM.

\subsection{Performance Estimation Problems (PEP) overview}

To solve the convex and Lipschitz smooth problems with the general first-order method using an exact gradient oracle, the PEP is defined as
\begin{align*}
	\max_{\substack{d, f                                                                                                                                       \\
	\bx_*, \bx_0, \cdots, \bx_k \in \mR^d
	}} \quad & f(\bx_K) - f(\bx_*)                                                                            \\
	\st  \qquad    \quad                                      & f \in \cF_{0,L},  \quad \bx_* \in \argmin_\bx f(\bx)                                           \\
	                                                          & \norm{\bx_0 - \bx_*}^2 \leq R^2                                                                \\
	                                                          & \bx_k =\bx_0 - \sum_{i=0}^{k-1}\frac{1}{L}\theta_{k,i}\grad f(\bx_i)  \quad k  = 1,\cdots,K,
\end{align*}
where \(\cF_{0,L}\) denotes the class of merely convex functions with Lipschitz continuous gradients. Note that \(\bx_k =\bx_0 - \sum_{i=0}^{k-1}\frac{1}{L}\theta_{k,i}\grad f(\bx_i),\ k  = 1,\cdots,K\), is the general first-order algorithm (GFO) performs which covers a wide range of first-order methods, including GD, FGM, Polyak's Heavy Ball Method~\cite{polyak1964some}, and Optimized Gradient Method~\cite{Drori2014PerformanceFirstOrder}.

The above problem is an infinite-dimensional optimization as the variable \(f\) is a function. \citet{Drori2014PerformanceFirstOrder} proposed a notion of \(\cF_{\mu,L}\)-interpolation that only considers function values and gradients at discrete points \(\seq{\bx_k}\), which are the terms that play a role in the optimization problem.
\begin{definition}[\(\cF_{\mu,L}\)-interpolation (Definition 2 in \cite{Taylor2016SmoothStronglyConvex})]
	Let \(\cI\) be an index set, and consider the set of triples \(\cS = \{ (\bx_i,\bg_i,f_i) \}_{i\in \cI}\) where \(\bx_i, \bg_i \in \mR^d\) for all \(i\in \cI\). Set \(\cS\) is \(\cF_{\mu,L}\)-interpolable if and only if there exists a function \(f \in \cF_{\mu,L}(\mR^d)\) such that we have both \(\bg_i = \grad f(\bx_i)\) and \(f(\bx_i) = f_i\) for all \(i \in \cI\).
\end{definition}

\begin{theorem}[\(\cF_{\mu,L}\)-interpolable (Theorem 4 in \cite{Taylor2016SmoothStronglyConvex}) ]
	Set \(\{ (\bx_i, \bg_i, f_i) \}_{i \in \cI}\) is \(\cF_{\mu,L}\)-interpolable (\(0 \leq \mu \leq L \leq  \infty\)) if and only if the following set of conditions holds for every pair of indices \(i \in \cI\) and \(j \in \cI\):
	\begin{equation}
		f_i - f_j - \bg_j^\top (\bx_i - \bx_j) \geq \frac{1}{2(1 - \mu/L)}\left( \frac{1}{L} \| \bg_i - \bg_j \|^2 + \mu \| \bx_i - \bx_j \|^2 - \frac{2\mu}{L} (\bg_j - \bg_i)^\top (\bx_j - \bx_i) \right).
	\end{equation}
\end{theorem}

The corresponding discrete PEP model is
\begin{align*}
	\max_{\substack{d,                                                                                                                   \\ \cS \subset \mR^d\times \mR^d\times \mR}} & f_K - f_* \tag{PEP-Exact} \label{prob:PEP-exact}\\
	\st \quad & \cS = \seq{(\bx_i,\bg_i,f_i)}_{i \in \seq{*,0,1,\cdots,K}}                                                               \\
	          & \bg_* = 0, \quad \norm{\bx_0 - \bx_*}^2\leq R^2                                                                            \\
	          & f_i - f_j - \bg_j^\top(\bx_i - \bx_j) \geq \frac{1}{2L}\norm{\bg_i - \bg_j}^2 \quad \forall i,j \in \seq{*,0,1,\cdots,K} \\
	          & \bx_k =\bx_0 - \sum_{i=0}^{k-1}\frac{1}{L}\theta_{k,i}\bg_i  \quad k = 1,\cdots, K.
\end{align*}
The theoretical equivalence of the discrete PEP model and the worst-case performance of GFO over the given class of functions is proven in \cite{Taylor2017ExactWorstCase, Taylor2016SmoothStronglyConvex} with \textit{convex interpolation} definition that guarantees PEP to generate the \textit{tight} worst-case performance.

The PEP model can also be used to find algorithms that minimize the worst-case function. In \citet{Drori2014PerformanceFirstOrder}, numerical results indicate existence of a better algorithm than the well-known FGM. \citet{Kim2016OptimizedFirstOrder} follow this idea and provide an explicit form of the generated algorithm called the Optimized Gradient Method (OGM), which is proven to match the lower bound with a smaller \(\cO(1)\) constant~\cite{Drori2017ExactInformationBased}. \citet{Drori2019EfficientFirstOrder} further reveal the equivalence between OGM and the conjugate gradient method--see also~\cite{Kim2017ConvergenceAnalysisOptimized, Kim2018GeneralizingOptimizedGradient, Kim2020OptimizingEfficiencyFirst, Park2023Factor$2$Acceleration}.

PEP technique has been successfully applied to various optimization methods, including gradient descent \cite{Grimmer2023ProvablyFasterGradient}, algorithms with line search \cite{Klerk2016WorstCaseComplexity, Drori2019EfficientFirstOrder}, proximal gradient methods \cite{Kim2018AnotherLookFast}, proximal point algorithms \cite{Kim2021AcceleratedProximalPoint}, algorithms with inexact oracles \cite{DeKlerk2020WorstCaseConvergence, Barre2022PrincipledAnalysesDesign}, and many others \cite{Barre2020ComplexityGuaranteesPolyak, Ryu2019FindingForwardDouglasrachford, Drori2016OptimalVariantKelleys}. Besides the deterministic smooth regimes, PEP has also been incorporated for stochastic problems~\cite{Taylor2019StochasticFirstOrder} and for problems satisfying \emph{relative} smoothness \cite{VanScoy2018FastestKnownGlobally}.

Given an algorithm, if \(d\) is large enough, PEP can equivalently be written as a Semidefinite Programming (SDP) and solved efficiently, primarily through its dual. However, optimizing over both the PEP dual and the algorithmic variables generally yields a nonconvex problem. Hence, to find an algorithm, a \emph{relaxed} PEP formulation (through dropping some constraints and changes of variable) is generally solved, which may not provide an optimal algorithm. However, \citet{DasGupta2023BranchBoundPerformance} propose a spatial branch-and-bound method to solve the original nonconvex problem over the dual PEP and algorithmic variable and provide a numerical guide to design algorithms.

\begin{remark}
A closely related topic is the analysis based on Integral Quadratic Constraints (IQC), originally studied in control theory \cite{Megretski1997SystemAnalysisVia}. Similar to PEP, IQC aims to find worst-case functions through optimization, but it is not an exact formulation of the function class; hence, it can only provide an upper bound. This technique has been used to analyze first-order methods in \cite{Lessard2016AnalysisDesignOptimization, Hu2017DissipativityTheoryNesterovs, VanScoy2018FastestKnownGlobally} and gradient method with inexact oracles in \cite{Hu2021AnalysisBiasedStochastic, Cyrus2018RobustAcceleratedOptimization}.
\end{remark}

\subsection{PEP for optimization with inexact oracle}\label{sec:PEP-inexact}

For inexact gradient oracles, the GFO is generalized to inexact GFO (iGFO) as
\begin{equation}\label{eq:GFO}
	\bx_{k} = \bx_0 -\frac{1}{L}\sum_{i=0}^{k-1}\theta_{k,i}(\bg_i + \be_i),
\end{equation}
and the discrete PEP model is adjusted as
\begin{align*}
	\max_{\substack{d,                                                                                                                   \\ \cS \subset \mR^d\times \mR^d\times \mR, \\ \be_0, \cdots, \be_K \in \mR^d}} & f_K - f_* \tag{PEP}\label{prob:PEP}\\
	\st \quad & \cS = \seq{(\bx_i,\bg_i,f_i)}_{i \in \seq{*,0,1,\cdots,K}}                                                               \\
	          & \bg_* = 0, \quad \norm{\bx_0 - \bx_*}^2\leq R^2                                                                            \\
	          & f_i - f_j - \bg_j^\top(\bx_i - \bx_j) \geq \frac{1}{2L}\norm{\bg_i - \bg_j}^2 \quad \forall i,j \in \seq{*,0,1,\cdots,K} \\
	          & \bx_k =\bx_0 - \sum_{i=0}^{k-1}\frac{1}{L}\theta_{k,i}(\bg_i + \be_i) \quad k = 1,\cdots, K                              \\
	          & \norm{\be_k}^2 \leq b_k^2\quad k=0,\cdots K-1,
\end{align*}
which is the generalization of the exact model. If \(b_k \equiv 0\) then \( \be_k \equiv \bo\), iGFO reduces to GFO, and \eqref{prob:PEP} reduces to \eqref{prob:PEP-exact}.

This formulation can be further simplified by introducing a Gram matrix. Define
\begin{align}
\label{eq:bFbX}
	\bF & := \bmat{f_0 - f_* & f_1 - f_* & \cdots & f_K - f_*} \in \mR^{1\times (K+1)},                                                       \\
	\bX & := \bmat{ \bx_0    & \be_0     & \cdots & \be_{K-1}                          & \bg_0 & \cdots & \bg_K } \in \mR^{d \times (2K+2)},
\end{align}
and the Gram matrix \(\bG :=  \bX^\top \bX \in \mS^{2K+2}\) which is a symmetric positive semidefinite (PSD) matrix. Also, we define sparse vectors \(\vec{\bff}_k \in \mR^{K+1}\) and \(\vec{\bx}_k, \vec{\bg}_k, \vec{\be}_k \in \mR^{2K+2}\) to select corresponding columns in \(\bF\) or \(\bX\) to recover the needed element, which satisfy
\begin{align*}
	 & f_k - f_* = \bF\vec{\bff}_k, \quad \vec{\bg}_* = \mathbf{0}, \quad \bg_k = \bX \vec{\bg}_k, \quad \be_k = \bX\vec{\be}_k,                           \\
	 & \bx_0 = \bX \vec{\bx}_0, \quad \bx_k = \bX\vec{\bx}_k =  \bX \left(\vec{\bx}_{0} - \frac{1}{L}\sum_{i=0}^{k-1}\theta_{k,i}(\vec{\bg}_i + \vec{\be}_i)\right).
\end{align*}
\textbf{Note.} We note that the subscript \(k\) on these 0-1 vectors does \emph{not} mean their \(k\)-th elements are equal to 1. For instance, based on the definitions of \(\bF\) and \(\bX\) in \eqref{eq:bFbX}, \(\vec{\bff}_k\), \(\vec{\be}_k\), and \(\vec{\bg}_k\) are standard basis vectors with \(k+1\)-th, \(k+2\)-th, \(K+k+2\)-th elements equal to 1, respectively, and \(\vec{\bx}_0\) is a standard basis vector with its first element equal to 1. With a slight abuse of the notation, \(\vec{\bx}_k\) (which is not a basis vector) is the vector used to represent \(\bx_k\) by a linear combination of the columns of \(\bX\), and it is a function of \(\theta_{k,i}\).

Without loss of generality, we assume \(f_*=0\) and \(\bx_*=0\), hence
\(\vec{\bff}_* = \mathbf{0}, \vec{\bx}_* = \mathbf{0}\). With this new notation, the PEP problem is written as
\begin{align*}
	\max_{d,\bG\in \mS_+^{2K+2},\bF \in \mR^{1\times (K+1)}}  \quad & \bF\vec{\bff}_K  \tag{PEP-Gram} \label{prob:PEP-Gram}                                                                          \\
	\st \quad                                                       & \bF(\vec{\bff}_j - \vec{\bff}_i) + \Tr(\bG \bA^{ij}) \leq 0, \quad \forall i,j \in  \seq{*,0,1,\cdots,K} \\
	                                                                & \Tr(\bG\vec{\bx}_0 \vec{\bx}_0^\top) - R^2 \leq 0                                                        \\
	                                                                & \Tr(\bG \vec{\be}_i \vec{\be}_i^\top) - b_i^2 \leq 0, \quad i = 0, \cdots, K -1                          \\
	                                                                & \Rank(\bG)\leq d,
\end{align*}
where
\begin{equation*}
	  \bA^{ij} := \frac{1}{2}((\vec{\bx}_i - \vec{\bx}_j) \vec{\bg}_j^\top + \vec{\bg}_j(\vec{\bx}_i - \vec{\bx}_j)^\top ) + \frac{1}{2L}(\vec{\bg}_i - \vec{\bg}_j)(\vec{\bg}_i - \vec{\bg}_j)^\top.
\end{equation*}

When dealing with large-scale problems, i.e., when \(2K+2 \leq  d\), we can drop the rank constraint \(\Rank{(\bG)}\leq d\) without changing the optimal value (see Theorem~5 in \cite{Taylor2017ExactWorstCase}), and the equivalent problem will be a Semidefinite Programming problem (SDP) written as
\begin{align*}
	\max_{\bG\in \mS^{2K+2}_+,\bF \in \mR^{1\times (K+1)}}  \quad & \bF\vec{\bff}_K  \tag{SDP-PEP}  \label{prob:SDP-PEP}                                                                        \\
	\st \quad                                                     & \bF(\vec{\bff}_j - \vec{\bff}_i) + \Tr(\bG\bA^{ij}) \leq 0, \quad \forall i,j \in  \seq{*,0,1,\cdots,K} \\
	                                                              & \Tr(\bG\vec{\bx}_0 \vec{\bx}_0^\top) - R^2 \leq 0                                                       \\
	                                                              & \Tr(\bG \vec{\be}_i \vec{\be}_i^\top) - b_i^2 \leq 0, \quad i = 0, \cdots, K -1.
\end{align*}
The optimal solution of \eqref{prob:SDP-PEP} will provide an exact worst-case performance of a given algorithm for convex problems.

It is possible to find the optimal algorithm with parameter \(\th\) by solving the minimax problem
\begin{equation}\label{prob:minmax}
	\min_{\theta_{k,i}} \max_{\bG,\bF \in \cS(\text{SDP-PEP})} \bF \vec{\bff}_K,
\end{equation}
where \(\bG,\bF \in\cS(\text{SDP-PEP})\) refers to \(\bG\) and \(\bF\) satisfy the constraints of \eqref{prob:SDP-PEP}. Since \(\bA^{i,j}\) is indeed a function of \(\{ \theta_{k,i}\}\), this minimax problem has bilinear terms in the constraints of the inner problem. Writing the inner maximization by its dual minimization form results in a (generally nonconvex)  Quadratically Constrained Quadratic Program (QCQP). By wisely selecting the constraints, OGM's~\cite{Drori2014PerformanceFirstOrder} idea is to transform the hard QCQP into a linear SDP that can be solved efficiently. Solving this problem numerically inspires a theoretical solution and results in the first version of OGM \emph{when gradients are exact}.

Recall that for FGM, there are two sequences, the primary sequence \(\{ \by_k \}\) and the secondary one \(\{ \bx_k \}\). The OGM is optimized for minimizing \(f(\bx_k)-f_*\) while the commonly used measure is \(f(\by_K) - f_*\). Due to this fact, \citet{Kim2017ConvergenceAnalysisOptimized} propose a new measure \(f(\bx_K) -f_*- \frac{1}{2L}\norm{\grad f(\bx_K)}^2\) since for smooth function and with the OGM update rule, \(f(\by_{K+1}) = f(\bx_K - \frac{1}{L} \grad f(\bx_K)) \leq f(\bx_K) - \frac{1}{2L}\norm{\grad f(\bx_K)}^2\).

Based on the constraint selection and objective modification discussed above, the relaxed optimization problem is
\begin{align*}
	\max_{\bG\in \mS_+^{2K+2}, \bF \in \mR^{(1\times K+1)}}  \quad & \bF\vec{\bff}_K - \frac{1}{2L} \Tr( \bG\vec{\bg}_K\vec{\bg}_K^\top)  \tag{P}  \label{prob:P}                                      \\
	\st \quad                                                      & \bF(\vec{\bff}_{i+1} - \vec{\bff}_i) + \Tr(\bG\bA^{i,i+1}) \leq 0, \quad i= 0,\cdots K-1, & \qquad  (v_{i,i+1}) \\
	                                                               & \bF(\vec{\bff}_i-\vec{\bff}_*) + \Tr(\bG \bA^{*,i}) \leq 0, \quad i= 0,\cdots K, \qquad   & (v_{*,i})           \\
	                                                               & \Tr(\bG\vec{\bx}_0 \vec{\bx}_0^\top) - R^2 \leq 0,                                        & \qquad  (\tau)      \\
	                                                               & \Tr(\bG \vec{\be}_i \vec{\be}_i^\top) - b_i^2 \leq 0, \quad i = 0, \cdots, K -1.          & (u_i)
\end{align*}
One can solve this problem by solving its dual problem. To write the dual of (P), we define the dual variables as \(\bv, \bv_*, \tau \), and \(\bu\) as indicated in the parenthesis following each constraint. The Lagrangian function is
\begin{align*}
	  & \cL(\bG,\bF,\bv,\bv_*,\bu,\tau)                                                                                                                                                                                                                        \\
	= & \bF\vec{\bff}_K - \frac{1}{2L} \Tr(\bG \vec{\bg}_K\vec{\bg}_K^\top) - \sum_{i=0}^{K-1}v_{i,i+1}\left(\bF(\vec{\bff}_{i+1}- \vec{\bff}_i)+\Tr (\bG\bA^{i,i+1})\right) - \sum_{i=0}^{K}v_{*,i}\left(\bF(\vec{\bff}_i - \vec{\bff}_*)+\Tr(\bG\bA^{*,i})\right) \\
	  & - \tau \left(\Tr(\bG\vec{\bx}_0\vec{\bx}_0^\top ) - R^2 \right)  -  \sum_{i=0}^{K-1} u_i \left(\Tr(\bG\vec{\be}_i\vec{\be}_i^\top ) - b_i^2 \right)                                                                                                    \\
	= & \tau R^2 + \sum_{i=0}^{K-1} u_i b_i^2 + \bF \left(\vec{\bff}_K + \sum_{i=0}^{K-1}v_{i,i+1}(\vec{\bff}_i - \vec{\bff}_{i+1}) + \sum_{i=0}^K v_{*,i}(\vec{\bff}_* - \vec{\bff}_i)\right)                                                                          \\
	  & - \Tr \left(\bG \left(\sum_{i=0}^{K-1} v_{i,i+1}\bA^{i,i+1} + \sum_{i=0}^K v_{*,i}\bA^{*,i} + \tau \vec{\bx}_0\vec{\bx}_0^\top + \sum_{i=0}^{K-1} u_i \vec{\be}_i\vec{\be}_i^\top + \frac{1}{2L}\vec{\bg}_K\vec{\bg}_K^\top \right)\right).
\end{align*}
For the problem \(\max_{\bG\in \mS_+,\bF} \cL(\bG,\bF,\bv,\bv_*,\bu,\tau)\) to have a bounded solution, we can write the corresponding dual problem as
\begin{align*}
	\min_{\tau,\bv,\bv_*,\bu \geq 0} \quad & \tau R^2 + \sum_{i=0}^{K-1}u_i b_i^2 \tag{D}   \label{prob:D}                                                                                                                                                             \\
	\st \quad                              & \vec{\bff}_K + \sum_{i=0}^{K-1}v_{i,i+1}(\vec{\bff}_i - \vec{\bff}_{i+1}) + \sum_{i=0}^K v_{*,i} (\vec{\bff}_* - \vec{\bff}_i) = 0                                                                          \\
	                                       & \sum_{i=0}^{K-1}v_{i,i+1} \bA^{i,i+1} + \sum_{i=0}^K v_{*,i}\bA^{*,i} + \tau \vec{\bx}_0 \vec{\bx}_0^\top + \sum_{i=0}^{K-1} u_i \vec{\be}_i \vec{\be}_i^\top  + \frac{1}{2L}\vec{\bg}_K\vec{\bg}_K^\top  \succeq 0. \numberthis \label{eq:PSD}
\end{align*}

Following the same scaling argument in the exact-gradient PEP framework of \citet{Taylor2016SmoothStronglyConvex}, one can normalize the smoothness constant and the initial-distance bound by rescaling the function and the domain. In the inexact-gradient scenario, the same scaling also normalizes the absolute error vector by the scale \(LR\).
Thus, the relevant error quantity is the dimensionless vector \(\bb/(LR)\), while the PEP value scales by \(LR^2\). The following lemma formalizes this scaling relation.

\begin{lemma}[Scaling of the PEP]
\label{lem:scaling-inexact-pep}
For any fixed \(K\), \(L>0\), \(R>0\), the algorithm parameters
\(\theta_{k,i}\), and absolute error vector
\(\bb=(b_0,\cdots,b_{K-1})^\top\in\mR_+^K\), define the normalized error vector
\[
    \boldsymbol{\beta}
    = \frac{\bb}{LR}.
\]
In this lemma, \(Q_K(L,R,\bb)\) denotes the optimal value of a model chosen from \eqref{prob:SDP-PEP}, \eqref{prob:P}, and \eqref{prob:D}.
Then
\begin{equation}\label{eq:pep-scaling-compact}
    Q_K(L,R,\bb)
    = LR^2
    Q_K(1,1,\boldsymbol{\beta}).
\end{equation}

Moreover, for the dual model, let
\(Q_K^{\mathrm{fix}}(L,R,\bb;\bar\tau)\) denote the optimal value
of \eqref{prob:D} after fixing \(\tau=L\bar\tau\). Then
\begin{equation}\label{eq:fixed-tau-scaling-compact}
    Q_K^{\mathrm{fix}}(L,R,\bb;\bar\tau)
    = LR^2
    Q_K^{\mathrm{fix}}(1,1,\boldsymbol{\beta};\bar\tau).
\end{equation}
\end{lemma}

\begin{proof}{Proof}
Fix one model from \eqref{prob:SDP-PEP}, \eqref{prob:P}, and
\eqref{prob:D}. Consider any feasible instance with parameters
\((L,R,\bb)\). Define
\[
    \bar{\bx}_i :=\frac{\bx_i-\bx_*}{R},
    \qquad \bar{\bg}_i :=\frac{\bg_i}{LR},
    \qquad \bar{\be}_i :=\frac{\be_i}{LR},
    \quad \bar f(\bar{\bx}) := \frac{f(\bx_*+R\bar{\bx})-f_*}{LR^2}.
\]
Then \(f\in\cF_{0,L}\) is equivalent to \(\bar f\in\cF_{0,1}\). The initial
condition \(\norm{\bx_0-\bx_*}\leq R\) is equivalent to
\(\norm{\bar{\bx}_0}\leq 1\). Similarly, the error condition
\(\norm{\be_i}\leq b_i\) is equivalent to
\(\norm{\bar{\be}_i}\leq \beta_i\), where
\(\beta_i=b_i/(LR)\).
The method recursion is also preserved under the same change of variables. The
equality
\[
    \bx_k
    = \bx_0-\frac1L\sum_{i=0}^{k-1}\theta_{k,i}(\bg_i+\be_i)
\]
holds exactly when
\[
    \bar{\bx}_k
    = \bar{\bx}_0-\sum_{i=0}^{k-1}\theta_{k,i}
    (\bar{\bg}_i+\bar{\be}_i)
\]
holds. Hence the feasible set with parameters \((L,R,\bb)\) is in one-to-one
correspondence with the feasible set with parameters
\((1,1,\boldsymbol{\beta})\).
The corresponding performance measure is multiplied by \(LR^2\). For
\eqref{prob:SDP-PEP}, we have
\[
    f_K-f_* = LR^2(\bar f_K-\bar f_*).
\]
For the modified objective in \eqref{prob:P}, we have
\[
    f_K-f_*-\frac{1}{2L}\norm{\bg_K}^2
    = LR^2
    \left(
    \bar f_K-\bar f_*-\frac12\norm{\bar{\bg}_K}^2
    \right).
\]
Therefore, taking the optimal values gives
\[
    Q_K(L,R,\bb)
    = LR^2Q_K(1,1,\boldsymbol{\beta})
\]
for \eqref{prob:SDP-PEP} and \eqref{prob:P}.

It remains to consider \eqref{prob:D}. For any feasible solution
\((\tau,\bu,\bv,\bv_*)\) with parameters \((L,R,\bb)\), define
\[
    \bar\tau=\frac{\tau}{L},
    \qquad \bar u_i=Lu_i,\quad i=0,\cdots,K-1,
\]
and keep \(v_{i,i+1}\) and \(v_{*,i}\) unchanged. This change of variables is
invertible. Under the same normalization as above, the equality and PSD
constraints in \eqref{prob:D} become exactly the corresponding constraints of
the normalized dual problem with parameters \((1,1,\boldsymbol{\beta})\).
Moreover,
\[
    \tau R^2+\sum_{i=0}^{K-1}u_i b_i^2
    = LR^2
    \left(
    \bar\tau+\sum_{i=0}^{K-1}\bar u_i\beta_i^2
    \right).
\]
Taking the optimal values gives the same scaling relation for \eqref{prob:D}.
If \(\tau=L\bar\tau\) is fixed in the original dual model, then the normalized
model fixes the corresponding multiplier at \(\bar\tau\), and the same
argument gives \eqref{eq:fixed-tau-scaling-compact}. \Halmos
\end{proof}

\citet{Kim2016OptimizedFirstOrder} analytically solved \eqref{prob:D} nested inside the minimization over the algorithm parameters \(\th\), \emph{in the absence of gradient inexactness, i.e., when \(\sum_{i=0}^{K-1}u_i b_i^2=0\) and \(\sum_{i=0}^{K-1} u_i \vec{\be}_i \vec{\be}_i^\top=0\),} and derived OGM with the following solution,
\begin{equation}\label{eq:OGM-solution}
	\tau = \frac{L}{4A_K}, \quad v_{i,i+1}=\frac{A_i}{A_K},  \quad  v_{*,0}=v_{0,1}, \quad v_{*,i} =v_{i,i+1}- v_{i-1,i}, \quad v_{*,K} = 1 - v_{K-1,K},
\end{equation}
and
\begin{equation}\label{eq:OGM-algorithm}
	\theta_{k,i}= \begin{cases}
		              \frac{2\alpha_i\alpha_k + A_{k-1}}{A_k}  	\quad                           & i=k-1            \\
		              \frac{2\alpha_k \alpha_i}{A_k} + \frac{A_{k-1}}{A_k}\theta_{k-1,i}  \quad & 0 \leq i\leq k-2\end{cases},
\end{equation}
with \(A_k = \sum_{i=0}^k \alpha_i\), \(A_0 = 1, \alpha_{-1}=A_{-1} = 0\).
By \cref{lem:scaling-inexact-pep}, this choice of \(\tau\) is
scale-consistent, since it corresponds to fixing
\(\bar\tau=1/(4A_K)\) in the normalized problem and the fixed-\(\tau\) dual optimal value scales by \(LR^2\).
Further, they showed that this solution is also a feasible solution to GOGM \cite{Kim2018GeneralizingOptimizedGradient} with  \(\alpha_k^2 \leq A_k\). We will follow this process, use the feasible solution above, and derive the value of \(\bu\) for GOGM with the \emph{inexact gradients}.

\begin{remark}\label{rmk:no-recursive-algorithm}
Based on the numerical experiments to solve over both the algorithm variables and the variables of the relaxed dual PEP formulations \eqref{prob:D} presented in \cref{sec:numerical-optimized-algorithm}, it is not straightforward to obtain a recursive algorithm in the inexact case. We leave designing such a recursive algorithm to future research. 
\end{remark}

It is easy to verify that the equality constraint in \eqref{prob:D} holds with this solution. We still need to find a feasible solution for \(\bu\) to satisfy the PSD constraint in \eqref{eq:PSD}. After substitution by \eqref{eq:OGM-solution} and \eqref{eq:OGM-algorithm}, the left-hand-side matrix of the PSD constraint (denoted as \(\bM\)) in \eqref{eq:PSD} is
\begin{equation}\label{eq:matrix-M}
	\bM = \begin{pmatrix*}[l]
		      \tau              & [\mathbf{0}]_{1\times K}    & [\bp]_{1 \times (K+1)}     \\
		      [\mathbf{0}]_{K \times 1}   & [\bU]_{K \times K}          & [\bB]_{K \times (K+1)}     \\
		      [\bp^\top]_{(K+1) \times 1} & [\bB^\top]_{(K+1) \times K} & [\bC]_{(K+1) \times (K+1)}
	      \end{pmatrix*},
\end{equation}
where \(\bU =\diag( u_0 \cdots u_{K-1} )\), \(\bp_{i}= - \frac{\alpha_{i-1}}{2A_K}\),
\begin{equation}\label{eq:B-definition}
	\bB_{i,j}=\begin{cases}
		          \frac{2\alpha_{i-1}\alpha_i +A_{i-1}}{2LA_K}\quad & j=i+1               \\
		          \frac{\alpha_{i-1}\alpha_{j-1}}{LA_K} \quad       & i+2\leq j \leq K +1
	          \end{cases},
	\quad \text{and} \quad
	\bC_{i,j}= \begin{cases}
		           \frac{A_{i-1}}{LA_K} \quad                  & i=j      \\
		           \frac{\alpha_{i-1}\alpha_{j-1}}{LA_K} \quad & i \neq j \\
	           \end{cases}.
\end{equation}

For \(K=3\), one example of \(\bM\) is listed below
\begin{equation*}
	\bM = \left(\begin{array}{c|ccc|cccc}
			            \frac{L}{4 A_3}              & 0                                  & 0                                                  & 0                                                  & - \frac{1}{2 A_3}          & - \frac{{\alpha}_{1}}{2 A_3}            & - \frac{{\alpha}_{2}}{2 A_3}                       & - \frac{{\alpha}_{3}}{2 A_3}            \\[2pt]
			            \hline
			            0                            & {u}_{0}                            & 0                                                  & 0                                                  & 0                          & \frac{1 + 2 {\alpha}_{1}}{2 L A_3}      & \frac{{\alpha}_{2}}{L A_3}                         & \frac{{\alpha}_{3}}{L A_3}              \\
			            0                            & 0                                  & {u}_{1}                                            & 0                                                  & 0                          & 0                                       & \frac{ 2 {\alpha}_{1} {\alpha}_{2} + A_1}{2 L A_3} & \frac{{\alpha}_{1} {\alpha}_{3}}{L A_3} \\0 & 0 & 0 & {u}_{2} & 0 & 0 & 0 & \frac{2 {\alpha}_{2} {\alpha}_{3} + A_2}{2 L A_3}\\[2pt]
			            \hline
			            - \frac{1}{2 A_3}            & 0                                  & 0                                                  & 0                                                  & \frac{1}{L A_3}            & \frac{{\alpha}_{1}}{L A_3}              & \frac{{\alpha}_{2}}{L A_3}                         & \frac{{\alpha}_{3}}{L A_3}              \\
			            - \frac{{\alpha}_{1}}{2 A_3} & \frac{1 + 2 {\alpha}_{1}}{2 L A_3} & 0                                                  & 0                                                  & \frac{{\alpha}_{1}}{L A_3} & \frac{A_1}{L A_3}                       & \frac{{\alpha}_{1} {\alpha}_{2}}{L A_3}            & \frac{{\alpha}_{1} {\alpha}_{3}}{L A_3} \\
			            - \frac{{\alpha}_{2}}{2 A_3} & \frac{{\alpha}_{2}}{L A_3}         & \frac{ 2 {\alpha}_{1} {\alpha}_{2} + A_1}{2 L A_3} & 0                                                  & \frac{{\alpha}_{2}}{L A_3} & \frac{{\alpha}_{1} {\alpha}_{2}}{L A_3} & \frac{A_2}{L A_3}                                  & \frac{{\alpha}_{2} {\alpha}_{3}}{L A_3} \\
			            - \frac{{\alpha}_{3}}{2 A_3} & \frac{{\alpha}_{3}}{L A_3}         & \frac{{\alpha}_{1} {\alpha}_{3}}{L A_3}            & \frac{ 2 {\alpha}_{2} {\alpha}_{3} + A_2}{2 L A_3} & \frac{{\alpha}_{3}}{L A_3} & \frac{{\alpha}_{1} {\alpha}_{3}}{L A_3} & \frac{{\alpha}_{2} {\alpha}_{3}}{L A_3}            & \frac{1}{L}
		            \end{array} \right)
\end{equation*}

From the structure of \(\bM\), we can observe that all the rows of \(\bB\) have at least one positive element, hence to make \(\bM \succeq 0\) we must have \(\bu > \bo\). Following the decomposition of \cite{Kim2017ConvergenceAnalysisOptimized}, i.e., \(\bM = \frac{1}{\tau} \bw_\tau \bw_\tau^\top + \bR\) with \(\bw_\tau\) being the first column of \(\bM\), with the condition \(A_k = \sum_{i=0}^{k}\alpha_i\), the residual matrix \(\bR\) has the form
\begin{equation*}
	\bR = \begin{pmatrix*}[l]
		      0                           & [\mathbf{0}]_{1\times K}  & [\mathbf{0}]_{1 \times K+1}                         \\
		      [\mathbf{0}]_{K \times 1}   & [\bU]_{K \times K}        & [\bB]_{K \times K+1}                                \\
		      [\mathbf{0}]_{K+1 \times 1} & [\bB^\top]_{K+1 \times K} & [\bC - \frac{1}{\tau}\bp^\top \bp ]_{K+1\times K+1}
	      \end{pmatrix*},
\end{equation*}
where \(\bC - \frac{1}{\tau}\bp^\top \bp = \diag(0,\seq{(A_{i-1}-\alpha_{i-1}^2)/(LA_K)}_{i=2}^{K+1})\), and equals to \(\bo\) with \(\seq{\alpha_k}\) defined in OGM.
\(\bR\) in the example of \(K=3\) is equal to
\begin{equation*}
	\bR = \left(\begin{array}{c|ccc|cccc}
			            0 & 0                                  & 0                                                  & 0                                                  & 0 & 0                                  & 0                                                  & 0                                       \\[2pt]
			            \hline
			            0 & {u}_{0}                            & 0                                                  & 0                                                  & 0 & \frac{1 + 2 {\alpha}_{1}}{2 L A_3} & \frac{{\alpha}_{2}}{L A_3}                         & \frac{{\alpha}_{3}}{L A_3}              \\
			            0 & 0                                  & {u}_{1}                                            & 0                                                  & 0 & 0                                  & \frac{ 2 {\alpha}_{1} {\alpha}_{2} + A_1}{2 L A_3} & \frac{{\alpha}_{1} {\alpha}_{3}}{L A_3} \\0 & 0 & 0 & {u}_{2} & 0 & 0 & 0 & \frac{2 {\alpha}_{2} {\alpha}_{3} + A_2}{2 L A_3}\\[2pt]
			            \hline
			            0 & 0                                  & 0                                                  & 0                                                  & 0 & 0                                  & 0                                                  & 0                                       \\
			            0 & \frac{1 + 2 {\alpha}_{1}}{2 L A_3} & 0                                                  & 0                                                  & 0 & \frac{A_1 - \alpha_1^2}{L A_3}     & 0                                                  & 0                                       \\
			            0 & \frac{{\alpha}_{2}}{L A_3}         & \frac{ 2 {\alpha}_{1} {\alpha}_{2} + A_1}{2 L A_3} & 0                                                  & 0 & 0                                  & \frac{A_2-\alpha_2^2}{L A_3}                       & 0                                       \\
			            0 & \frac{{\alpha}_{3}}{L A_3}         & \frac{{\alpha}_{1} {\alpha}_{3}}{L A_3}            & \frac{ 2 {\alpha}_{2} {\alpha}_{3} + A_2}{2 L A_3} & 0 & 0                                  & 0                                                  & \frac{A_3 - \alpha_3^2}{LA_3}
		            \end{array} \right).
\end{equation*}
Recalling that \(\bR^{\backslash[1]}\) denotes the matrix resulting from dropping the first column and row of \(\bR\), this submatrix is indeed the Schur complement of the block \([\tau]\) of the matrix \(\bM\). Since \(\tau >0\) (as \(\tau = \frac{L}{4A_K}\)), we know that \(\bM \succeq 0\) if and only if \(\bR^{\backslash[1]} \succeq 0\). In the following lemma, we will show that \(\bu = \infty\) is the unique feasible solution when \(A_i = \alpha_i^2, \forall i = 1,\cdots,K\).
\begin{lemma}\label{lem:u-infty}
	With the given solution in \eqref{eq:OGM-solution} and \eqref{eq:OGM-algorithm}, if \(A_i - \alpha_i^2 = 0, \forall i = 1,\cdots,K\), then \(\bu = \infty\) is the unique solution that makes \(\bR^{\backslash[1]}\succeq 0\). When \(A_i - \alpha_i^2>0\), \(\bu = \infty\) is still a feasible solution but not unique.
\end{lemma}
\begin{proof}{Proof}
	First, we know that all \(u_i\) are positive numbers. Define \(\vec{\bv}_k\) be a unit vector with a 1 at the \(k\)-th element. If \(\bR^{\backslash[1]}\succeq 0\), then for any \(c\in \mR\) and defining \(\bv := \vec{\bv}_i+c\vec{\bv}_j\), the condition
	\begin{equation*}
		0 \leq  \bv^\top \bR^{\backslash[1]} \bv =  c^2 \bR^{\backslash[1]}_{j,j} +  \bR^{\backslash[1]}_{i,i} + 2c\bR^{\backslash[1]}_{i,j}
	\end{equation*}
	holds. To show that, there are four different scenarios:
        \begin{equation*}
        \begin{array}{l@{\qquad \qquad}l}
            1.\ \bR^{\backslash[1]}_{i,i} = 0, \bR^{\backslash[1]}_{j,j} = 0, & 2.\ \bR^{\backslash[1]}_{i,i} = u_k, \bR^{\backslash[1]}_{j,j} = u_l, \\
            3.\ \bR^{\backslash[1]}_{i,i} = u_k, \bR^{\backslash[1]}_{j,j} = 0, & 4.\ \bR^{\backslash[1]}_{i,i} = 0, \bR^{\backslash[1]}_{j,j} = u_l.
        \end{array}
        \end{equation*}
For the first and second scenarios, it is obvious that \(\bR^{\backslash[1]}_{i,j}=0\), with diagonal elements being non-negative, the condition holds.
For the third scenario, \(\bR^{\backslash[1]}_{i,j} \geq 0\). Consider the case \(\bR^{\backslash[1]}_{i,j} > 0\), since \(c\) can be any negative number, the condition  	\(\bv^\top \bR^{\backslash[1]} \bv = u_k + 2c \bR^{\backslash[1]}_{i,j}\geq 0\)  holds only when \(u_k = \infty\).
	For the fourth scenario, the situation is similar. \(\bv^\top \bR^{\backslash[1]} \bv =c^2 u_l + 2c \bR^{\backslash[1]}_{i,j}\geq 0\) holds only with \(u_l = \infty\). \Halmos \end{proof}

The \textit{exact} OGM with feasible solution \eqref{eq:OGM-solution} and \(\bu = \infty\) provides the convergence rate of
\begin{equation*}
	f(\by_{K+1}) - f_* \leq \frac{LR^2}{4A_K} + \sum_{i=0}^{K-1} u_i b_i^2 = \frac{LR^2}{4A_K}.
\end{equation*}

The unique feasible solution \(\bu = \infty\) has no effect on the final convergence rate since \(\
	b_i \equiv 0 \). However, with the inexact gradient oracle, such a solution will ruin the convergence rate as the accumulated error approaches \(\infty\). Hence, we need to change either the feasible solution given in \eqref{eq:OGM-solution} or the algorithm's structure defined in \eqref{eq:OGM-algorithm}. The details will be discussed in the next section.

\section{Generalized OGM with inexact gradient oracle}\label{sec:iGOGM}

The goal of this section is to construct an analytical feasible solution of \eqref{prob:D} for iGOGM. By weak duality, any feasible solution of \eqref{prob:D} provides an upper bound on the worst-case convergence rate. Rather than solving the full dual problem over all multipliers, we restrict attention to a tractable family motivated by the exact-oracle OGM solution.

Following the exact-oracle OGM construction, we fix the multipliers as
\[
    v_{i,i+1}=\frac{A_i}{A_K},\quad i=0,\ldots,K-1.
\]
The equality constraint in \eqref{prob:D} then uniquely determines the corresponding multipliers \(\seq{v_{*,i}}\). Hence, after fixing \(\bv\), the remaining variables in \eqref{prob:D} are \(\tau\) and \(\bu\). A natural first step is to also set \(\tau=L/(4A_K)\), as in the exact-oracle OGM solution. With this choice, the PSD constraint can be reduced, via a Schur complement argument, to a condition involving only \(\bu\). This reduction leads to an explicit optimal choice of \(\bu\) for the accumulated-error term. Interestingly, the same pair \((\tau,\bu)\) is optimal for the restricted problem of \eqref{prob:D} in which only \(\bv\) is fixed and \(\tau\) and \(\bu\) are optimized jointly, as shown in \cref{lem:optimal-u-fixed-tv}.

As we claimed in \cref{lem:u-infty}, OGM (\(A_k = \alpha_k^2\)) with the slected multiplier in \eqref{eq:OGM-solution} leads to the degenerate feasible choice \(\bu = \infty\). To make \(\bu\) to be bounded, we have to guarantee that the stepsize condition holds with strict inequality, i.e., \(A_k - \alpha_k^2 >0\).

Under this strict condition, when \(\tau=L/(4A_K)\), the block
\(\bC-\frac1\tau\bp^\top\bp\) has a zero first row and column, while its remaining principal block is positive definite. Specifically, after deleting this zero row and column, the remaining block can be defined as
\(\bD := \diag\left(\frac{A_1-\alpha_1^2}{LA_K},\ldots,\frac{A_K-\alpha_K^2}{LA_K}\right)\succ0.\)
Since the corresponding first column of \(\bB\) defined in \eqref{eq:B-definition} is also zero, the PSD constraint can be reduced to the Schur complement condition
\( \bU-\bB^{\backslash[1]}\bD^{-1}\bB^{\backslash[1]\top}\succeq 0.\)
Define 
\begin{equation}\label{eq:H-definition}
\bH := \bB^{\backslash[1]}\bD^{-1}\bB^{\backslash[1]\top}. 
\end{equation}
 With this reduction, the problem reduces to finding \(\bu\) through solving the problem  \(\min_\bu\sum_{i=0}^{K-1}u_i b_i^2\) subject to \(\bU-\bH\succeq0\). The next lemma identifies the optimal \(\bu\) and further shows that the same \(\tau=L/(4A_K)\) is optimal in the restricted dual problem where only \(\bv\) is fixed.

\begin{lemma}
\label{lem:optimal-u-fixed-tv}
Assume \(A_k-\alpha_k^2>0\), \(k=1,\cdots,K\), and \(b_i>0\), \(i=0,\cdots,K-1\). Consider the restricted-form of the problem \eqref{prob:D} after fixing \( v_{i,i+1}=\frac{A_i}{A_K},\ i=0,\ldots,K-1, \) while optimizing over \(\tau\geq 0\) and \(\bu\geq \bo\). The equality constraint in \eqref{prob:D} then uniquely determines
\[ 
    v_{*,0}=v_{0,1},\qquad
    v_{*,i}=v_{i,i+1}-v_{i-1,i},\quad i=1,\ldots,K-1,
    \qquad  v_{*,K}=1-v_{K-1,K}.
\]
Then, the optimal solution is
\begin{equation}
    \tau^*=\frac{L}{4A_K}, \qquad
    u_i^* = \hat{u}_i := \frac{(\bH\bb)_i}{b_i}, \quad i=0,\cdots,K-1,
    \label{eq:optimal-u}
\end{equation}
where \(\bb=(b_0,\cdots,b_{K-1})^\top\), which is also a feasible solution of \eqref{prob:D}. Moreover, the optimal objective value of this restricted problem is
\[
\frac{LR^2}{4A_K}+\bb^\top\bH\bb .
\]
\end{lemma}

\begin{proof}{Proof}
After fixing \(v_{i,i+1}=A_i/A_K\), the equality constraint in \eqref{prob:D} fixes the corresponding values of \(\bv_*\) as stated above. Therefore, the remaining variables are \(\tau\) and \(\bu\). With these values of \(\bv\) and \(\bv_*\), the left-hand-side matrix of the PSD constraint \eqref{eq:PSD} has the block form
\[
	\bM(\tau,\bu) = \begin{pmatrix*}[l]
		      \tau              & [\mathbf{0}]_{1\times K}    & [\bp]_{1 \times (K+1)}     \\
		      [\mathbf{0}]_{K \times 1}   & [\bU]_{K \times K}          & [\bB]_{K \times (K+1)}     \\
		      [\bp^\top]_{(K+1) \times 1} & [\bB^\top]_{(K+1) \times K} & [\bC]_{(K+1) \times (K+1)}
	      \end{pmatrix*},
\]
where \(\bU=\diag(u_0,\ldots,u_{K-1})\), and \(\bp\), \(\bB\), and \(\bC\) are as defined above.

First, note that any feasible solution must have \(\tau>0\). Indeed, if \(\tau=0\), then the first diagonal entry of the PSD matrix is zero, which would force the whole first row and column to be zero. This is impossible since the first entry of \(\bp\) is \(-1/(2A_K)\). Hence, we may take the Schur complement with respect to \(\tau\). Thus \(\bM(\tau,\bu)\succeq 0\) is equivalent to
\[
    \begin{pmatrix}
        \bU      & \bB                           \\
        \bB^\top & \bC-\frac{1}{\tau}\bp^\top\bp
    \end{pmatrix}
    \succeq 0.
\]
Let \( \boldsymbol{\alpha}:= (1,\alpha_1,\ldots,\alpha_K)^\top, \   \bar{\boldsymbol{\alpha}}:= (\alpha_1,\ldots,\alpha_K)^\top. \)
With \(A_0=\alpha_0=1\), we can write
\[
    \bC-\frac{1}{\tau}\bp^\top\bp
    = \begin{pmatrix}
        0   & \bo^\top \\
        \bo & \bD
    \end{pmatrix}
    + \eta(\tau)\boldsymbol{\alpha}\boldsymbol{\alpha}^\top,
\]
where
\[
    \eta(\tau) := \frac{1}{LA_K}-\frac{1}{4A_K^2\tau} = \frac{4A_K\tau-L}{4LA_K^2\tau}.
\]
Since the matrix above is a principal submatrix of the Schur complement, feasibility requires \( \eta(\tau)\geq 0, \)
and therefore \( \tau\geq \frac{L}{4A_K}. \)
It remains to characterize the feasible set for \(\bu\) for any \(\tau\geq L/(4A_K)\). Since the first column of \(\bB\) is zero, write
\( \bB=\begin{pmatrix} \bo & \bB^{\backslash[1]} \end{pmatrix}. \)
If \(\tau>L/(4A_K)\), then \(\eta(\tau)>0\). Reordering the blocks in the Schur complement gives the equivalent condition
\[
    \begin{pmatrix}
        \bU                     & \bo                                 & \bB^{\backslash[1]}                      \\
        \bo^\top                & \eta(\tau)                          & \eta(\tau)\bar{\boldsymbol{\alpha}}^\top \\
        \bB^{\backslash[1]\top} & \eta(\tau)\bar{\boldsymbol{\alpha}} &
        \bD+\eta(\tau)\bar{\boldsymbol{\alpha}}\bar{\boldsymbol{\alpha}}^\top
    \end{pmatrix}
    \succeq 0.
\]
Taking the Schur complement with respect to the scalar block \(\eta(\tau)\), this condition is equivalent to
\[
    \begin{pmatrix}
        \bU                     & \bB^{\backslash[1]} \\
        \bB^{\backslash[1]\top} & \bD
    \end{pmatrix}
    \succeq 0.
\]
When \(\tau=L/(4A_K)\), we have \(\eta(\tau)=0\), and the same condition follows by deleting the zero row and column. Since \(\bD\succ 0\), the last condition is equivalent to
\( \bU-\bB^{\backslash[1]}\bD^{-1}\bB^{\backslash[1]\top}\succeq 0, \)
or, equivalently, \( \bU-\bH\succeq 0. \)
The restricted problem is thus equivalent to
\[
    \min_{\tau\geq L/(4A_K),\,\bu\geq \bo}
    \quad \tau R^2+\sum_{i=0}^{K-1}u_i b_i^2
    \qquad \st
    \qquad \bU-\bH\succeq 0.
\]
Since \(R>0\), the objective is strictly increasing in \(\tau\). Hence, \( \tau^*=\frac{L}{4A_K}. \) It remains only to solve the \(\bu\)-subproblem
\[
    \min_{\bu\geq \bo}
    \quad \sum_{i=0}^{K-1}u_i b_i^2
    \qquad \st
    \qquad \bU-\bH\succeq 0.
\]
The Lagrangian of this problem is \(\cL(\bu,\bX)=\sum_{i=0}^{K-1}b_i^2u_i-\Tr\big(\bX(\bU-\bH)\big), \ \bX \succeq 0.\)
The KKT conditions are
\[
    \bU - \bH \succeq \bo,\qquad
    \bX\succeq \bo,\qquad
    \grad_{\bu}\mathcal{L}(\bu,\bX)=\bo,\qquad
    \Tr\big(\bX(\bU-\bH)\big)=0.
\]
We verify these conditions by choosing \(\bX^*=\bb\bb^\top\) and \(\hat{\bu}\) as in \eqref{eq:optimal-u}. Clearly, \(\bX^*\succeq 0\). Since \( \Tr(\bX\bU)=\sum_{i=0}^{K-1}X_{ii}u_i, \)
the Lagrangian can be written as \( \mathcal{L}(\bu,\bX) = \sum_{i=0}^{K-1}(b_i^2-X_{ii})u_i + \Tr(\bX\bH). \)
Thus stationarity is equivalent to \(X_{ii}=b_i^2\) for all \(i=0,\ldots,K-1\), which holds for \(\bX^*=\bb\bb^\top\).

It remains to show primal feasibility and complementary slackness. Let \(\hat{\bU}=\diag(\hat{\bu})\). Since \(\bB^{\backslash[1]}\) and \(\bD^{-1}\) are elementwise nonnegative, \(\bH\) is elementwise nonnegative. For any \(\bz\in\mR^K\), using the symmetry of \(\bH\), we have
\[
    \bz^\top(\hat{\bU}-\bH)\bz
    = \sum_{i,j=0}^{K-1}H_{ij}\frac{b_j}{b_i}z_i^2
    - \sum_{i,j=0}^{K-1}H_{ij}z_iz_j
    = \frac{1}{2}\sum_{i,j=0}^{K-1}H_{ij}b_ib_j
    \left(\frac{z_i}{b_i}-\frac{z_j}{b_j}\right)^2
    \geq 0.
\]
Hence, \(\hat{\bU}-\bH\succeq 0\). Furthermore, \eqref{eq:optimal-u} gives \( (\hat{\bU}-\bH)\bb=\bo, \)
and therefore \(\Tr\big(\bX^*(\hat{\bU}-\bH)\big)  = \Tr\big(\bb\bb^\top(\hat{\bU}-\bH)\big)  =0. \)
The Slater condition holds by taking \(u_i=m\) for all \(i=0,\ldots,K-1\), where \(m>\lambda_{\max}(\bH)\). Hence, the KKT conditions are sufficient, \(\hat{\bu}\) is optimal, and the objective value is
\[
    \sum_{i=0}^{K-1}b_i^2 \hat{u}_i = \sum_{i=0}^{K-1}b_i(\bH\bb)_i  = \bb^\top\bH\bb. \Halmos
\]
\end{proof}

Note that when \(b_i \equiv b\), \eqref{eq:optimal-u} reduces to the row-sum solution \(\hat{u}_i=\sum_{j=0}^{K-1}H_{ij}\). Inserting the entries of \(\bH\) gives
\begin{equation}\label{eq:theoretical-u}
    \hat{u}_i =  \frac{A_i(1+2\alpha_{i+1})(A_i+2\alpha_i\alpha_{i+1})}{4LA_K(A_{i+1}-\alpha_{i+1}^2)}+ \sum_{k=i+1}^{K-1} \frac{A_k(1+2\alpha_{k+1})\alpha_i\alpha_{k+1}}{2L A_K(A_{k+1}- \alpha_{k+1}^2)}, \quad i=0,\cdots,K-1.
\end{equation}
The expression above also shows that when \(A_k - \alpha_k^2 = 0\), \(\hat{u}_i\) becomes unbounded, which recovers the unique feasible solution in \cref{lem:u-infty}.

\subsection{PEP-inspired proof of the convergence analysis of iGOGM} \label{sec:Algorithm-proof}

In this section, we provide the detailed proof of \cref{thm:convergence-rate}. As explained in \cite{Goujaud2023FundamentalProofStructures}, a feasible solution to the PEP's dual problem \eqref{prob:SDP-PEP}, or its relaxed version \eqref{prob:D}, provides a direct proof for the algorithm's convergence bound. We note that even though we use the PEP technique to find an upper bound on the convergence bound,
the proof can be understood without any prior knowledge of the PEP or SDP's duality theory.

Following this idea, we will establish our proof for iGOGM based on the feasible solution we derive in \cref{lem:optimal-u-fixed-tv}. Recall the properties of Lagrangian duality, any such feasible solution of \eqref{prob:D} yields the following upper bound on the optimization error
\begin{align*}
    f(\bx_K) - f_* \leq & \cL(\bG, \bF, \bv, \bv_*,\bu,\tau) \\
    =&  \tau R^2 + \sum_{i=0}^{K-1} u_i b_i^2 + \bF \left(\vec{\bff}_K + \sum_{i=0}^{K-1}v_{i,i+1}(\vec{\bff}_i - \vec{\bff}_{i+1}) + \sum_{i=0}^K v_{*,i}(\vec{\bff}_* - \vec{\bff}_i)\right)                                                                          \\
	  & - \Tr \left(\bG \left(\sum_{i=0}^{K-1} v_{i,i+1}\bA^{i,i+1} + \sum_{i=0}^K v_{*,i}\bA^{*,i} + \tau \vec{\bx}_0\vec{\bx}_0^\top + \sum_{i=0}^{K-1} u_i \vec{\be}_i\vec{\be}_i^\top + \frac{1}{2L}\vec{\bg}_K\vec{\bg}_K^\top \right)\right) \\
      \leq & \tau R^2 +\sum_{i=0}^{K-1} u_i b_i^2.
\end{align*}
Based on the above inequality, our proof is constructed by forming a linear combination of the inequalities derived from the function's convexity and \(L\)-smoothness properties. The dual variables \(\bv\) and \(\bu\) act as the coefficients for this combination. The first step is to multiply the following inequalities by their corresponding dual variables

{
    \fontsize{11.5pt}{12.5pt}\selectfont
 \begin{align*}
		 & f(\bx_{k+1}) - f(\bx_{k}) + \fprod{\grad f(\bx_{k+1}),\bx_{k} - \bx_{k+1}} + \frac{1}{2L}\norm{\grad f(\bx_k)- \grad f(\bx_{k+1})}^2 \leq 0 &  \cdots \   & v_{k,k+1} = \frac{A_k}{A_K}         \\
		 & f(\bx_k) - f_* + \fprod{\grad f(\bx_k),\bx_* - \bx_k} + \frac{1}{2L}\norm{\grad f(\bx_k)}^2 \leq 0                                          &  \cdots  \ & v_{*,k} = \frac{A_k - A_{k-1}}{A_K} \\
         & \norm{\bx_0 - \bx_*}^2 \leq R^2 & \cdots \ & \tau = \frac{L}{4A_K} \\
         & \norm{\be_k }^2 \leq b_k^2 & \cdots \ & u_k
	\end{align*}
    }
Summing over \(k\) and subtracting \(\frac{1}{2L}\norm{\grad f(\bx_K)}^2\) on both sides and rearranging terms we get
\begin{align*}
    &f(\bx_K) - f_* - \frac{1}{2L}\norm{\grad f(\bx_K)}^2                                                                                                                                                                  \\
		\leq              & \tau(R^2 - \norm{\bx_0 - \bx_*}^2) + \sum_{k=0}^{K-1} u_k(b_k^2 - \norm{\be_k}^2) - \sum_{k=0}^{K-1}\frac{A_k}{A_K}\fprod{\grad f(\bx_{k+1}),\bx_{k} - \bx_{k+1}}  \\
        &   - \sum_{k=0}^K\frac{A_k -A_{k-1}}{A_K}\fprod{\grad f(\bx_k),\bx_* - \bx_k}    - \frac{1}{2L}\sum_{k=0}^{K-1}\frac{A_k}{A_K}\norm{\grad f(\bx_k)- \grad f(\bx_{k+1})}^2                                                                                   \\
		                  &   - \frac{1}{2L}\sum_{k=0}^{K}\frac{A_k-A_{k-1}}{A_K} \norm{\grad f(\bx_k)}^2 - \frac{1}{2L}\norm{\grad f(\bx_K)}^2.
\end{align*}
Replacing \(\bx_k\) with the update rule of iGOGM defined by \eqref{eq:GFO} and \eqref{eq:OGM-algorithm}, the above inequality becomes
\begin{align*}
    & f(\bx_K) - f_* - \frac{1}{2L}\norm{\grad f(\bx_K)}^2   \\
    \leq & \tau(R^2 - \norm{\bx_0 - \bx_*}^2) + \sum_{k=0}^{K-1}u_k(b_k^2 - \norm{\be_k}^2) + \sum_{k=0}^{K}\frac{\alpha_k}{A_K} \fprod{\grad f(\bx_k),\bx_0 - \bx_*} \\
    & + \frac{1}{L}\sum_{k=0}^{K-1}\frac{A_k}{A_K}\fprod{\grad f(\bx_k),\grad f(\bx_{k+1})}  - \frac{2}{L}\sum_{k=0}^{K-1}\sum_{i=0}^{k} \frac{\alpha_{k+1} \alpha_i}{A_K} \fprod{\grad f(\bx_{k+1}),\grad f(\bx_{i}) + \be_i}                              \\
		                  &  - \frac{1}{L}\sum_{k=0}^{K-1} \frac{A_k}{A_K}\fprod{\grad f(\bx_{k+1}), \grad f(\bx_k) + \be_k} - \frac{1}{L}\sum_{k=0}^{K}\frac{A_k}{A_K} \norm{\grad f(\bx_k)}^2 .
\end{align*}

To complete the proof, we must find \(\tau\) and \(\seq{u_k}\) that ensures the right-hand-side of the above inequality is bounded by \(\tau R^2 + \sum_{k=0}^{K-1}u_k b_k^2\). This is equivalent to finding a feasible solution for the dual problem \eqref{prob:D}, since the bound is guaranteed if the positive semidefinite constraint \eqref{eq:PSD} holds. Our procedure for finding \(\bu\) is exactly the step of canceling the remaining terms in the right-hand-side. We will provide a detailed explanation below.

First, we do the Schur complement step corresponding to \(\bM = \frac{1}{\tau}\bw_\tau \bw_\tau^\top +\bR\), which is also the step to bound \(\fprod{\grad f(\bx_k), \bx_0 - \bx_*}\).
\begin{align*}
         & f(\bx_K) - f_* - \frac{1}{2L}\norm{\grad f(\bx_K)}^2   \\
     \leq &  \tau(R^2 - \norm{\bx_0 - \bx_*}^2) +  \sum_{k=0}^{K-1}u_k(b_k^2 - \norm{\be_k}^2) - \underbrace{\frac{1}{\tau}\norm{\tau(\bx_0-\bx_*) - \frac{1}{2} \sum_{k=0}^{K}\frac{\alpha_k}{A_K}\grad f(\bx_k)}^2}_{\frac{1}{\tau} \bw_\tau \bw_\tau^\top} \\
     & + \tau \norm{\bx_0 - \bx_*}^2 + \sum_{k=0}^K\left( \frac{\alpha_k^2}{4\tau A_K^2} - \frac{A_k}{L A_K} \right)\norm{\grad f(\bx_k)}^2  + \frac{1}{2\tau}\sum_{k=0}^K\sum_{i = 0}^{k-1}\frac{\alpha_k\alpha_i}{A_K^2}\fprod{\grad f(\bx_k), \grad f(\bx_i)}  \\
		                  & + \frac{1}{L}\sum_{k=0}^{K-1}\frac{A_k}{A_K}\fprod{\grad f(\bx_k),\grad f(\bx_{k+1})}  - \frac{2}{L}\sum_{k=0}^{K-1}\sum_{i=0}^{k} \frac{\alpha_{k+1} \alpha_i}{A_K} \fprod{\grad f(\bx_{k+1}),\grad f(\bx_{i}) + \be_i} \\
                          & - \frac{1}{L}\sum_{k=0}^{K-1} \frac{A_k}{A_K}\fprod{\grad f(\bx_{k+1}), \grad f(\bx_k) +  \be_k}   \\
    = & \frac{LR^2}{4A_K} + \sum_{k=0}^{K-1}u_k(b_k^2 - \norm{\be_k}^2) -\frac{4A_K}{L}\norm{\frac{L}{4A_K}(\bx_0-\bx_*) - \frac{1}{2} \sum_{k=0}^{K}\frac{\alpha_k}{A_K}\grad f(\bx_k)}^2  \\
    &  - \sum_{k=1}^K \frac{A_k - \alpha_k^2}{L A_K}\norm{\grad f(\bx_k)}^2    - \frac{2}{L}\sum_{k=0}^{K-1}\sum_{i=0}^{k} \frac{\alpha_{k+1} \alpha_i}{A_K} \fprod{\grad f(\bx_{k+1}), \be_i} - \frac{1}{L}\sum_{k=0}^{K-1} \frac{A_k}{A_K}\fprod{\grad f(\bx_{k+1}),  \be_k},
\end{align*}
where in the last equality we replace \(\tau = \frac{L}{4A_K}\) and \(A_0 = \alpha_0 = 1\). In the exact case, \(b_k \equiv 0 \) and \(\be_k \equiv \bo\), the inner-product terms in the above inequality are all zero, and all norm terms have negative coefficients. The convergence result of the exact GOGM can be recovered as
	\begin{equation*}
		f(\bx_K) - f_* - \frac{1}{2L}\norm{\grad f(\bx_K)}^2\leq  \frac{L\norm{\bx_0 - \bx^*}^2}{4A_K}.
	\end{equation*}
In the inexact case, we need to continue to find a value of \(\seq{u_k}\) that cancels the non-zero inner-product terms. To achieve this goal, we will follow the step of the Schur complement reduction characterized in \cref{lem:optimal-u-fixed-tv}, i.e., \(\bU-\bB^{\backslash[1]}\bD^{-1}\bB^{\backslash[1]\top}\succeq 0 \)
\begin{align*}
    & f(\bx_K) - f_* - \frac{1}{2L}\norm{\grad f(\bx_K)}^2  \\
    \leq & \frac{LR^2}{4A_K} +  \sum_{k=0}^{K-1}u_k(b_k^2 - \norm{\be_k}^2) - \sum_{k=1}^K \frac{A_k - \alpha_k^2}{L A_K}\norm{\grad f(\bx_k)}^2    \\
    &  - \frac{2}{L}\sum_{k=0}^{K-1}\sum_{i=0}^{k} \frac{\alpha_{k+1} \alpha_i}{A_K} \fprod{\grad f(\bx_{k+1}), \be_i}  - \frac{1}{L}\sum_{k=0}^{K-1} \frac{A_k}{A_K}\fprod{\grad f(\bx_{k+1}),  \be_k} \\
    = & \frac{LR^2}{4A_K} + \sum_{k=0}^{K-1}u_k(b_k^2 - \norm{\be_k}^2) + \sum_{k=0}^{K-1} \frac{1}{LA_K(A_{k+1} - \alpha_{k+1}^2)} \sum_{i=0}^{k-1} \alpha_{k+1}^2\alpha_i^2 \norm{\be_i}^2 \\
    & - \underbrace{\sum_{k=0}^{K-1} \frac{1}{LA_K(A_{k+1} - \alpha_{k+1}^2)}\norm{(A_{k+1} - \alpha_{k+1}^2)\grad f(\bx_{k+1}) + \sum_{i=0}^k\alpha_{k+1}\alpha_i \be_i + \frac{1}{2}A_k \be_k}^2}_{\bB^{\backslash[1]}\bD^{-1}\bB^{\backslash[1]^\top }}                                                                \\
                           & + \sum_{k=0}^{K-1} \frac{1}{LA_K(A_{k+1} - \alpha_{k+1}^2)}\sum_{i=0}^{k}\sum_{j=0}^{i-1}2 \alpha_{k+1}^2\alpha_i\alpha_j \fprod{\be_i,\be_j} \numberthis \label{eq:slack_motivation} \\
		                   & +\sum_{k=0}^{K-1} \frac{1}{LA_K(A_{k+1} - \alpha_{k+1}^2)} \sum_{i=0}^{k-1} \alpha_{k+1} \alpha_i A_k \fprod{\be_i,\be_k} + \sum_{k=0}^{K-1} \frac{(2\alpha_{k+1} \alpha_k + A_k)^2}{4LA_K(A_{k+1} - \alpha_{k+1}^2)} \norm{\be_k}^2 \\
                        \leq & \frac{LR^2}{4A_K} +  \sum_{k=0}^{K-1}u_k(b_k^2 - \norm{\be_k}^2) + \sum_{k=0}^{K-1}  \frac{(2\alpha_{k+1} \alpha_k + A_k)^2}{4LA_K (A_{k+1} - \alpha_{k+1}^2)}\norm{\be_k}^2  \\
                        & + \sum_{k=0}^{K-1}\sum_{j=k+1}^{K-1} \frac{\alpha_k^2 \alpha_{j+1}^2}{LA_K(A_{j+1}-\alpha_{j+1}^2)}\norm{\be_k}^2            + \sum_{k=0}^{K-1}\sum_{i=0}^{k-1} \frac{\alpha_{k+1}\alpha_i (A_k+ 2\alpha_{k+1} \alpha_k)}{LA_K(A_{k+1} - \alpha_{k+1}^2)}\fprod{\be_k,\be_i}                                                              \\
		     &  + \sum_{k=0}^{K-1}\sum_{i=0}^{k-1}\sum_{j=k+1}^{K-1} \frac{2\alpha_{j+1}^2 \alpha_k \alpha_i}{LA_K(A_{j+1} - \alpha_{j+1}^2)} \fprod{\be_k,\be_i}.
\end{align*}
For notation simplicity, let \(0\leq i <k \leq K-1\), we use \(P_{k,i}\) to replace the coefficients of \(\fprod{\be_k,\be_i}\) as
	\begin{align*}
		P_{k,k} := & \frac{(2\alpha_{k+1} \alpha_k + A_k)^2}{4LA_K(A_{k+1} - \alpha_{k+1}^2)} +\sum_{j=k+1}^{K-1} \frac{\alpha_k^2 \alpha_{j+1}^2}{LA_K(A_{j+1}-\alpha_{j+1}^2)},                                \\
		P_{k,i} := & \frac{\alpha_{k+1}\alpha_i (A_k + 2 \alpha_{k+1}\alpha_k )}{LA_K(A_{k+1} - \alpha_{k+1}^2)} + \sum_{j=k+1}^{K-1} \frac{2\alpha_{j+1}^2 \alpha_k \alpha_i}{LA_K(A_{j+1} - \alpha_{j+1}^2)}.
	\end{align*}
	Hence, the above inequality can be rewritten as
	\begin{align*}
     & f(\bx_K) - f_* - \frac{1}{2L}\norm{\grad f(\bx_K)}^2  \\
     \leq & \frac{LR^2}{4A_K} +  \sum_{k=0}^{K-1}u_k(b_k^2 - \norm{\be_k}^2) +  \sum_{k=0}^{K-1} P_{k,k} \norm{\be_k}^2  + \sum_{k=0}^{K-1}\sum_{i=0}^{k-1} P_{k,i}\fprod{\be_k,\be_i}     \\
		= & \frac{LR^2}{4A_K} + \sum_{k=0}^{K-1}u_k b_k^2 - \left[ \sum_{k=0}^{K-1}(u_k-P_{k,k})\norm{\be_k}^2 - \sum_{k=0}^{K-1}\sum_{i=0}^{k-1}P_{k,i}\fprod{\be_k,\be_i}\right].
	\end{align*}

It remains to choose \(u_k\) so that the square-bracketed term above is nonnegative for arbitrary error vectors. Notice that for each cross term
\( -P_{k,i}\fprod{\be_k,\be_i}, \  0\leq i<k\leq K-1,
\)
we can construct a quadratic form with \(\zeta_{k,i}>0\) as
\[
    \frac{P_{k,i}}{2}\zeta_{k,i}
    \norm{
        \be_k-\zeta_{k,i}^{-1}\be_i
    }^2
    = \frac{P_{k,i}}{2}\zeta_{k,i}\norm{\be_k}^2
    + \frac{P_{k,i}}{2}\zeta_{k,i}^{-1}\norm{\be_i}^2
    - P_{k,i}\fprod{\be_k,\be_i},
\]
which shows that the cross term can be matched for any positive \(\zeta_{k,i}\). With this construction, we can rewrite the square-bracketed term as
\begin{align*}
     & \sum_{k=0}^{K-1}(u_k-P_{k,k})\norm{\be_k}^2
     - \sum_{k=0}^{K-1}\sum_{i=0}^{k-1}P_{k,i}\fprod{\be_k,\be_i} \\
     ={}&
     \sum_{k=0}^{K-1}
    \left(
        u_k-P_{k,k}
        - \frac{1}{2}\sum_{i=0}^{k-1}P_{k,i}\zeta_{k,i}
        - \frac{1}{2}\sum_{j=k+1}^{K-1}P_{j,k}\zeta_{j,k}^{-1}
    \right) \norm{\be_k}^2  + \frac{1}{2}
    \sum_{k=0}^{K-1}\sum_{i=0}^{k-1}
    P_{k,i}\zeta_{k,i}
    \norm{
        \be_k-\zeta_{k,i}^{-1}\be_i
    }^2 .
\end{align*}
Choosing
\[
    u_k
    = P_{k,k}
    + \frac{1}{2}\sum_{i=0}^{k-1}P_{k,i}\zeta_{k,i}
    + \frac{1}{2}\sum_{j=k+1}^{K-1}P_{j,k}\zeta_{j,k}^{-1}
\]
makes the coefficient of each \(\norm{\be_k}^2\) term in the first summation zero, and hence makes the whole square-bracketed term nonnegative. Moreover, recall that we also want to make the remaining error term
\(
    \sum_{k=0}^{K-1}u_k b_k^2
\)
as small as possible. With the above choice of \(u_k\), we have
\begin{align*}
    \sum_{k=0}^{K-1}u_k b_k^2
    = \sum_{k=0}^{K-1}P_{k,k}b_k^2
    + \frac{1}{2}
    \sum_{k=0}^{K-1}\sum_{i=0}^{k-1}
    P_{k,i}
    \left(
        \zeta_{k,i}b_k^2
        + \zeta_{k,i}^{-1}b_i^2
    \right).
\end{align*}
Since the above expression is separable in each \(\zeta_{k,i}\), with \(\bb > \mathbf{0}\), minimizing it with respect to \(\zeta_{k,i}>0\) gives
\[
    \zeta_{k,i}=\frac{b_i}{b_k}.
\]
The corresponding \(u_k\) becomes
\[
    \hat{u}_k
    = P_{k,k}
    + \frac{1}{2b_k}\sum_{i=0}^{k-1}P_{k,i}b_i
    + \frac{1}{2b_k}\sum_{j=k+1}^{K-1}P_{j,k}b_j,
    \qquad k=0,\ldots,K-1.
\]
With this choice, the square-bracketed term becomes
\begin{align*}
    &\sum_{k=0}^{K-1}(\hat{u}_k-P_{k,k})\norm{\be_k}^2
    - \sum_{k=0}^{K-1}\sum_{i=0}^{k-1}
    P_{k,i}\fprod{\be_k,\be_i} \\
    ={}&
    \frac{1}{2}
    \sum_{k=0}^{K-1}\sum_{i=0}^{k-1}
    P_{k,i}b_kb_i
    \norm{
        \frac{\be_k}{b_k}
        - \frac{\be_i}{b_i}
    }^2
    \geq 0.
\end{align*}
Therefore, the square-bracketed term is nonnegative and can be dropped. We finally obtain
\[
    f(\bx_K)-f_*-\frac{1}{2L}\norm{\grad f(\bx_K)}^2
    \leq \frac{LR^2}{4A_K}
    + \sum_{k=0}^{K-1}\hat{u}_k b_k^2. \Halmos
\]

Inspired by the approach for the analysis of iGOGM, we also find a convergence rate of inexact Generalized FGM (iGFGM) with the algorithm shown in \cref{alg:iGFGM} in \cref{appx:iGFGM}. Similar to the result of iGOGM, the bound of iGFGM is a summation of the exact convergence rate and accumulated error.

\begin{remark}
    Note that \cref{thm:convergence-rate} is valid only for \(0 < \alpha_k^2 < A_k\), i.e., \(\alpha_k^2\) cannot be equal to \(A_k\), which reflects a limitation of the analytically feasible solution we derived. However, as demonstrated in \cref{thm:convergence-iFGM}, for iFGM, we can set \(\alpha_k^2 = A_k\). This particular choice allows iFGM to attain its fastest convergence rate if the oracle is exact.
\end{remark}

\section{Conclusions}
In this paper, we analyze the Generalized Optimized Gradient Method (GOGM) with the inexact gradient oracle, i.e., iGOGM, under the absolute error assumption. By utilizing the Performance Estimation Problem (PEP) analysis tool, a new upper bound for the convergence rate of the iGOGM algorithm is derived. This bound comprises two components: the convergence rate obtained from the exact gradient oracle and the accumulated error resulting from the gradient oracle's inexactness. Such a bound demonstrates the effect of the inexactness and does not require boundedness of the feasible region. Furthermore, the accumulated error component of the bound is independent of the initial condition.

Furthermore, from this convergence bound, the optimal schedule to set the oracle inexactness along iterations is proposed. Such a study aims to minimize the sampling or computational effort for gradient estimation while maintaining the convergence rate.

\newpage
\bibliographystyle{plainnat}
\IfFileExists{refs.bib}
  {\bibliography{refs}}
  {\bibliography{../refs}}
\newpage
\begin{appendices}
\section{Supplementary lemmas and proofs }

\subsection{Details on motivational examples}\label{sec:detail_example}
\paragraph{Bilevel optimization.} Consider the bilevel optimization problem
	\begin{align*}
		\min_\bx & \quad f(\bx;\by^*(\bx))                    \\
		\st      & \quad \by^*(\bx) = \argmin_\by g(\bx,\by),
	\end{align*}
	with \(f\) being continuously differentiable and \(g\) being two times continuously differentiable and strongly convex functions. To solve the problem with a gradient-based method, we first obtain the gradient with respect to \(\bx\) as
	\begin{align*}
		\grad f(\bx;\by^*(\bx)) & = \grad_\bx f(\bx;\by^*(\bx)) - \grad^2_{\bx\by} g(\bx,\by^*(\bx))[\grad^2_{\by\by} g(\bx,\by^*(\bx))]^{-1} \grad_\by f(\bx;\by^*(\bx)) \\
		                        & = \grad_\bx f(\bx;\by^*(\bx)) - \grad \by^*(\bx)^\top \grad_\by f(\bx;\by^*(\bx)).
	\end{align*}
	This requires solving the lower-level problem to optimality, which is unattainable in many scenarios. When the lower-level problem is solved to a suboptimal point \(\tilde{\by}(\bx)\) and the gradient is evaluated at this point instead of \(\by^*(\bx)\), the resulting gradient is inexact:
	\begin{equation*}
		\grad f(\bx;\by^*(\bx)) \neq \tilde{\grad}f(\bx,\tilde{\by}(\bx)) := \grad_\bx f(\bx,\tilde{\by}(\bx)) - \grad^2_{\bx\by}g(\bx,\tilde{\by}(\bx))[\grad^2_{\by\by} g(\bx,\tilde{\by}(\bx))]^{-1} \grad_\by f (\bx;\tilde{\by}(\bx)).
	\end{equation*}
	Under some continuity assumptions, the gradient error is shown to be bounded as
	\begin{equation*}
		\norm{\tilde{\grad}f(\bx;\tilde{\by}(\bx)) - \grad f(\bx;\by^*(\bx))} \leq C\norm{ \tilde{\by}(\bx) - \by^*(\bx)},
	\end{equation*}
	where \(C := L_{f_\bx} + \frac{L_{f_\by}C_{g_{\bx\by}}}{\mu_g} + C_{f_\by} \left(\frac{L_{g_\bx\by}}{\mu_g} + \frac{L_{g_{\by\by}} C_{g_{\bx\by}}}{\mu_g^2}\right)\), with \(L_{f_\bx}, L_{f_\by}\) are the Lipschitz continuity constants of \(\grad _\bx f(\bx;\by)\) and \(\grad_\by f(\bx;\by)\), respectively; \(L_{g_{\bx\by}}, L_{g_{\by\by}}\) are the Lipschitz constants of \(\grad^2_{\bx\by}g(\bx,\by)\) and \(\grad^2_{\by\by}g(\bx,\by)\) with respect to (w.r.t.) \(\by\), respectively; \(C_{f_\by}\), \(C_{g_{\bx\by}}\) are the upper bounds on \(\norm{\grad_\by f(\bx;\by)}\) and \(\norm{\grad^2_{\bx\by}g(\bx,\by)}\), respectively; and \(\mu_g\) is the strong convexity constant of \(g(\bx,\by)\) w.r.t. \(\by\)-- see~\cite{Ghadimi2018ApproximationMethodsBilevel} for details.

	Now let \(\tilde{\by}_k(\bx)\) be the solution of the lower-level problem by the gradient descent method with stepsize \(\frac{2}{\mu_g + L_g}\) after \(k\) iterations. From the {iteration complexity} of GD for the class \(\cF_{\mu,L}\), the gradient inexactness can be bounded as
	\begin{equation*}
		\norm{\tilde{\grad}f(\bx;\tilde{\by}_k(\bx)) - \grad f(\bx;\by^*(\bx))} \leq C \left(\frac{Q_g - 1}{Q_g + 1}\right)^k \norm{\by_0 - \by^*(\bx)},
	\end{equation*}
	where \(Q_g := L_g/\mu_g\). The above bound follows an exponential decay and represents the error bound in \eqref{eq:inexact-oracle} with \(\eta\) being the iteration number \(k\).

\paragraph{Composition optimization.}
	Consider the nested composition optimization
	\begin{equation*}
		\min_{\bx} \ f(\bx) := h(\bg(\bx)) \quad \text{with} \quad \bg(\bx) := \frac{1}{N} \sum_{i=1}^N \bg_i(\bx),
	\end{equation*}
    where \(h:\mR^k\rightarrow\mR\) and \(\bg:\mR^d\rightarrow\mR^k\) are continuously differentiable, and \(N\) is a large positive integer. From the chain rule, the exact gradient is \(\grad f(\bx) = \grad \bg(\bx) \grad h(\bg(\bx))\). However, assume one can only approximate \(\bg\) and \(\grad \bg\) through sampling as
	\begin{equation*}
		\tilde{\bg}(\bx) := \frac{1}{\abs{\cS_\bg}}\sum_{i=1}^{\abs{\cS_\bg}} \bg_{\cS_\bg[i]}(\bx) \quad \tilde{\grad}\bg(\bx) := \frac{1}{\abs{\cS_\grad}}\sum_{i=1}^{\abs{\cS_\grad}} \grad \bg_{\cS_\grad[i]} (\bx),
	\end{equation*}
	where \(\cS_{\bg}\), \(\cS_{\grad}\) are two uniformly sampled subsets of \(\seq{1,2,\cdots,N}\).

	\begin{lemma}\label{lem:composition-error-bound}
		Assume for any \(\bv\) and \(\bx\), \(\norm{\grad h(\bv)} \leq C_h\) and \(\norm{\grad \bg(\bx)} \leq C_g\). Also, \(h(\bv)\) is Lipschitz smooth with constant \(L_h\), \(\mE_i[\norm{\grad \bg_i(\bx) - \grad \bg(\bx)}^2] \leq \sigma^2_\grad\) and \(\mE[\norm{\bg_i(\bx) - \bg(\bx)}^2] \leq \sigma^2_\bg\). Assume \(\abs{\cS_\grad}\) and \(\abs{\cS_\bg}\) are properly chosen so \(\frac{2C_h^2}{b^2\abs{\cS_\grad}} \sigma_\grad^2 + \frac{2C_g^2 L_h^2}{b^2\abs{\cS_\bg}} \sigma_\bg^2 \leq 1\), then the error of the estimated gradient is bounded as \[
			\norm{\tilde{\grad} \bg(\bx) \grad h(\tilde{\bg}(\bx)) - \grad \bg(\bx) \grad h(\bg(\bx))}^2 \leq b^2,
		\]
		with probability at least
		\(1 - \frac{2C_h^2}{b^2\abs{\cS_\grad}} \sigma_\grad^2 - \frac{2C_g^2 L_h^2}{b^2\abs{\cS_\bg}} \sigma_\bg^2\).
	\end{lemma}
	\begin{proof}{Proof}
		By Young's inequality,
		\begin{align*}
			     & \norm{\tilde{\grad} \bg(\bx) \grad h(\tilde{\bg}(\bx)) - \grad \bg(\bx) \grad h(\bg(\bx))}^2                                                                                                  \\
			=    & \norm{\tilde{\grad} \bg(\bx) \grad h(\tilde{\bg}(\bx)) - \grad \bg(\bx) \grad h(\tilde{\bg}(\bx)) + \grad \bg(\bx) \grad h(\tilde{\bg}(\bx)) - \grad \bg(\bx) \grad h(\bg(\bx))}^2            \\
			\leq & 2\norm{\tilde{\grad} \bg(\bx) \grad h(\tilde{\bg}(\bx)) - \grad \bg(\bx) \grad h(\tilde{\bg}(\bx))}^2 + 2\norm{\grad \bg(\bx) \grad h(\tilde{\bg}(\bx)) - \grad \bg(\bx) \grad h(\bg(\bx))}^2 \\
			\leq & 2 \norm{\grad h(\tilde{\bg}(\bx))}^2 \norm{\tilde{\grad}\bg(\bx) - \grad \bg(\bx)}^2 + 2 \norm{\grad \bg(\bx)}^2 \norm{\grad h(\tilde{\bg}(\bx)) - \grad h (\bg(\bx))}^2                      \\
			\leq & 2 C_h^2 \norm{\tilde{\grad}\bg(\bx) - \grad \bg(\bx)}^2 + 2 C_g^2 L_h^2 \norm{\tilde{\bg}(\bx) - \bg(\bx)}^2.
		\end{align*}
		Taking the expectation on both sides and with the bounded variance assumptions, we have
		\begin{align*}
			      \mE[\norm{\tilde{\grad} \bg(\bx) \grad h(\tilde{\bg}(\bx)) - \grad \bg(\bx) \grad h(\bg(\bx))}^2]
			\leq  \frac{2C_h^2}{\abs{\cS_\grad}} \sigma_\grad^2 + \frac{2C_g^2 L_h^2}{\abs{\cS_\bg}} \sigma_\bg^2.
		\end{align*}
		By Markov's inequality, we have
		\begin{align*}
			     & \Pr\left(\norm{\tilde{\grad} \bg(\bx) \grad h(\tilde{\bg}(\bx)) - \grad \bg(\bx) \grad h(\bg(\bx))}^2 \geq \epsilon \left(\frac{2C_h^2}{\abs{\cS_\grad}} \sigma_\grad^2 + \frac{2C_g^2 L_h^2}{\abs{\cS_\bg}} \sigma_\bg^2 \right)\right)        \\
			\leq & \frac{1}{\epsilon} \mE\left[\norm{\tilde{\grad} \bg(\bx) \grad h(\tilde{\bg}(\bx)) - \grad \bg(\bx) \grad h(\bg(\bx))}^2\right] / \left(\frac{2C_h^2}{\abs{\cS_\grad}} \sigma_\grad^2 + \frac{2C_g^2 L_h^2}{\abs{\cS_\bg}} \sigma_\bg^2 \right)\leq  \frac{1}{\epsilon}.
		\end{align*}
		Assuming \(\epsilon > 1\), i.e., the bound holds non-triviality, we can conclude  that
		\[
			\norm{\tilde{\grad} \bg(\bx) \grad h(\tilde{\bg}(\bx)) - \grad \bg(\bx) \grad h(\bg(\bx))}^2 \leq \epsilon \left(\frac{2C_h^2}{\abs{\cS_\grad}} \sigma_\grad^2 + \frac{2C_g^2 L_h^2}{\abs{\cS_\bg}} \sigma_\bg^2 \right)
		\]
		holds with probability at least \(1 - \frac{1}{\epsilon}\).	\Halmos \end{proof}

\subsection{Convergence bound for iGFGM}\label{appx:iGFGM}

\begin{lemma}
    The sequence \(\seq{\bx_k}\) generated by \cref{alg:iGFGM} is the same sequence as the one generated by iGFO (\(\hat{\bx}_k = \bx_0 - \frac{1}{L}\sum_{i=0}^{k-1}\theta_{k,i} \tilde{\grad} f(\hat{\bx}_i)\)) with
    \begin{equation*}
        \theta_{k,i} = \frac{\alpha_i (A_k - A_i)+A_i}{A_k}.
    \end{equation*}
\end{lemma}
\begin{proof}{Proof}
    The equivalence of these two sequences is shown through induction. For \(k=0\) and \(k=1\), the equivalence holds trivially.
    Assume \(\bx_i = \hat{\bx}_{i}\) for all \(i\leq k\) with \(k>1\), then
    \begin{align*}
        \bx_{k+1} = & \left(1 - \frac{\alpha_{k+1}}{A_{k+1}} \right) \left(\bx_k - \frac{1}{L}
        \tilde{\grad}f(\bx_k) \right) + \frac{\alpha_{k+1}}{A_{k+1}} \left(\bx_0 - \frac{1}{L} \sum_{i=0}^{k} \alpha_i \tilde{\grad}f(\bx_i)\right) \\
        = &\left(1 - \frac{\alpha_{k+1}}{A_{k+1}} \right) \left(\bx_0 - \frac{1}{L}\sum_{i=0}^{k-1}\theta_{k,i}\tilde{\grad}f(\bx_i) - \frac{1}{L}\tilde{\grad}f(\bx_k) \right) + \frac{\alpha_{k+1}}{A_{k+1}} \left(\bx_0 - \frac{1}{L} \sum_{i=0}^{k} \alpha_i \tilde{\grad}f(\bx_i)\right) \\
        = & \bx_0 - \frac{1}{L} \frac{A_{k+1} - \alpha_{k+1}}{A_{k+1}} \sum_{i=0}^{k-1}\frac{\alpha_i(A_k - A_i) +A_i}{A_k}\tilde{\grad}f(\bx_i) - \frac{1}{L}\frac{\alpha_{k+1}}{A_{k+1}}\sum_{i=0}^{k-1}\alpha_i \tilde{\grad}f(\bx_i) \\
        & - \frac{1}{L}\left(1 - \frac{\alpha_{k+1}}{A_{k+1}}  + \frac{\alpha_{k+1}\alpha_k}{A_{k+1}}\right)\tilde{\grad}f(\bx_k) \\
        = & \bx_0 - \frac{1}{L}\sum_{i=0}^{k-1} \frac{\alpha_i(A_k - A_i) + A_i + (A_{k+1} - A_k) \alpha_i}{A_{k+1}}\tilde{\grad}f(\bx_i) - \frac{1}{L} \frac{\alpha_k \alpha_{k+1} + A_k}{A_{k+1}} \tilde{\grad}f(\bx_k) \\
        = & \bx_0 - \frac{1}{L}\sum_{i=0}^{k-1}\frac{\alpha_i(A_{k+1} - A_{i}) + A_i}{A_{k+1}} - \frac{1}{L} \frac{\alpha_k (A_{k+1} - A_k) + A_k}{A_{k+1}} \tilde{\grad}f(\bx_k)\\
        = & \bx_0 - \frac{1}{L}\sum_{i=0}^{k}\frac{\alpha_i(A_{k+1} - A_i) + A_i}{A_{k+1}}\tilde{\grad}f(\bx_i) \\
        =& \bx_0 - \frac{1}{L}\sum_{i=0}^k \theta_{k+1,i}\tilde{\grad}f(\bx_i)\\
        =& \hat{\bx}_{k+1}.
    \end{align*}
    Hence the sequences \(\seq{\bx_k}\) and \(\seq{\hat{\bx}_k}\) are equivalent to each other.
\Halmos \end{proof}

For iGFGM, after fixing \(\tau=L/(2A_K)\) and
\(v_{i,i+1}=A_i/A_K\), the same Schur complement reduction as in
\cref{lem:optimal-u-fixed-tv} yields a constraint of the form \( \bU-\bH\succeq0 \),
where \(\bH\in\mS^K\) denotes the corresponding Schur-complement matrix as defined \eqref{eq:H-definition}. This reduced condition allows us to identify the optimal choice of \(\bu\) for the accumulated-error term, as stated next.

\begin{lemma}[Optimal \(\bu\) for fixed \(\tau\) and \(\bv\) in iGFGM]\label{lem:optimal-u-fixed-tv-FGM}
Assume  \(2A_k-\alpha_k^2>0\), \(k=1,\cdots,K\), and \(b_i>0\), \(i=0,\cdots,K-1\). For the fixed analytical multipliers \(\tau=L/(2A_K)\), \(v_{i,i+1}=A_i/A_K\), and the corresponding \(\bv_*\), the remaining minimization over \(\bu\) is
\begin{equation*}
    \min_{\bu\geq\bo}\sum_{i=0}^{K-1}u_i b_i^2: \diag(\bu)-\bH\succeq 0.
\end{equation*}
Its optimal solution is
\begin{equation}
    \hat{u}_i=\frac{(\bH\bb)_i}{b_i}, \quad i=0,\cdots,K-1,
    \label{eq:optimal-u-FGM}
\end{equation}
where \(\bb=(b_0,\cdots,b_{K-1})^\top\). Moreover, if \(b_i\equiv b\), then
\begin{equation}
    \hat{u}_k = \frac{A_k^2 (1+\alpha_{k+1})}{2L A_K(2 A_{k+1} - \alpha_{k+1}^2)} + \sum_{i=k+1}^{K} \frac{\alpha_k A_{i-1}\alpha_i (1+\alpha_i)}{2LA_K(2A_i - \alpha_i^2)}, \quad k=0,\cdots,K-1.
    \label{eq:theoretical-u-FGM}
\end{equation}
\end{lemma}
The proof follows the same Schur complement reduction and KKT verification as in \cref{lem:optimal-u-fixed-tv}, and is omitted.

\section{Proofs of convergence upper bounds for iFGM and iOGM}
\subsection{Proof of \cref{thm:ifgm-K-square-error-bound}}\label{appx:iFGM-K2}
We will prove the result by constructing a feasible solution to the dual problem \eqref{prob:D}.
Recall that the dual PSD constraint has the block form
\begin{equation*}
	\bM = \begin{pmatrix*}[l]
		      \tau              & [\mathbf{0}]_{1\times K}    & [\bp]_{1 \times (K+1)}     \\
		      [\mathbf{0}]_{K \times 1}   & [\bU]_{K \times K}          & [\bB]_{K \times (K+1)}     \\
		      [\bp^\top]_{(K+1) \times 1} & [\bB^\top]_{(K+1) \times K} & [\bC]_{(K+1) \times (K+1)}
	      \end{pmatrix*},
\end{equation*}
where \(\bU=\diag(u_0,\ldots,u_{K-1})\). For any \(\tau>0\), define
\(
    \bC_\tau:=\bC-\frac1\tau\bp\bp^\top
\)
and
\[
    \bN_\tau:=
    \begin{pmatrix}
        \bU & \bB\\
        \bB^\top & \bC_\tau
    \end{pmatrix}.
\]
Then the Schur complement with respect to the leading scalar block (\(\tau\)-block) with \(\tau>0\) gives
\[
    \bM\succeq0
    \quad\Longleftrightarrow\quad
    \bN_\tau\succeq0,
\]
If, in addition, \(\bC_\tau\succ0\), then a second Schur complement gives
\[
    \bN_\tau\succeq0
    \quad\Longleftrightarrow\quad
    \bU-\bB\bC_\tau^{-1}\bB^\top\succeq0.
\]
Thus it suffices to choose \(\tau>0\) and \(\bu\geq \bo\) such that
\(\bC_\tau\succ0\) and
\(\bU-\bB\bC_\tau^{-1}\bB^\top\succeq0\). We construct such a pair in two
steps, first choose \(\tau\) to make \(\bC_\tau\succ0\), and then choose the
diagonal matrix \(\bU\) to dominate \(\bB\bC_\tau^{-1}\bB^\top\).

In the iFGM case, \(\lambda_k\equiv1\) and \(A_k=\alpha_k^2\). The
fixed-multiplier choice \(v_{i,i+1}=A_i/A_K\) in
\cref{lem:optimal-u-fixed-tv-FGM} gives an \(\cO(K^3b^2/L)\) accumulated-error
term when \(b_i\equiv b\). We obtain a tighter bound by choosing
\[
    v_{i,i+1} = 1 - \left(1 - \frac{A_i}{A_K}\right)^2, \quad i=0,\cdots,K-1.
\]
The equality constraint in \eqref{prob:D} then determines
\[
    v_{*,0}=v_{0,1},\qquad
    v_{*,i}=v_{i,i+1}-v_{i-1,i},\quad i=1,\ldots,K-1,
    \qquad v_{*,K}=1-v_{K-1,K}.
\]
Since \(v_{i,i+1}\) is increasing in \(i\) and \(v_{K-1,K}\leq 1\), all these
multipliers are nonnegative. With these new choices of multipliers, the Schur complement obtained after eliminating the \(\tau\)-block is not a diagonal matrix. We therefore identify a value of \(\tau\) so that \(\bC_\tau\succ 0\).

To begin with, we will show \(\bC \succ 0\). Define \(\bD_\alpha:=\diag(\alpha_0,\ldots,\alpha_K)\), with the convention \(A_{-1}=0\), a direct substitution of the above choice of \(\bv\) gives
\[
    \bC
    = \frac{1}{L A_K^2}\bD_\alpha \bW \bD_\alpha,
\]
where
\[
W_{ij}
= \left\{
\begin{array}{ll}
2A_K-A_{i-1},
& i=j,\\[1mm]
\frac{1}{2}\left(2A_K-A_{i-1}-A_{i-2}+A_{j-2}\right),
& i>j,\\[1mm]
W_{ji},
& j>i,
\end{array}
\right.
\qquad i,j=1,\ldots,K+1.
\]
As \(\bD_\alpha\) is positive definite, we only need to show \(\bW\succ 0\). Notice \(\bW\) can be decomposed into three parts as
\(
    \bW=\bP+\frac12\bQ+\bT,
\)
where 
\[
    P_{ij}=A_K-A_{\max\{i,j\}-1},
    \qquad Q_{ij}=A_{\min\{i,j\}-2},
    \qquad T_{ij}
    = \begin{cases}
        A_K-\dfrac12 A_{i-2},
        & i=j,\\[1mm]
        \dfrac12\alpha_{\max\{i,j\}-1},
        & i\ne j,
    \end{cases}
    \qquad i,j=1,\ldots,K+1.
\]
Since \(A_k=\sum_{i=0}^k\alpha_i\), for any \(\bz\in\mathbb R^{K+1}\), it holds
\[
    \bz^\top \bP \bz
    = \sum_{i=1}^K \alpha_i
    \left(\sum_{j=1}^{i}z_j\right)^2
    \geq 0,
    \qquad \bz^\top \bQ \bz
= \sum_{i=0}^{K-1}\alpha_i
\left(\sum_{j=i+2}^{K+1}z_j\right)^2
\geq 0.
\]
Therefore \(\bP\succeq0\) and \(\bQ\succeq0\).
It remains to show \(\bT\succ0\). For any \(\bz\in\mR^{K+1}\), by the definition
of \(\bT\), for any nozero \(\bz\), we have
\begin{align*}
\bz^\top\bT\bz
&= \sum_{i=1}^{K+1}\left(A_K-\frac12A_{i-2}\right)z_i^2
+ \sum_{1\leq j<i\leq K+1}\alpha_{i-1}z_j z_i  \\
&= \sum_{i=1}^{K+1}\left(A_K-\frac12A_{i-2}\right)z_i^2
+ \frac12\sum_{i=2}^{K+1}\alpha_{i-1}
\left[
\left(\sum_{j=1}^{i}z_j\right)^2
- \left(\sum_{j=1}^{i-1}z_j\right)^2
- z_i^2
\right] \\
&= \sum_{i=1}^{K+1}
\left(A_K-\frac12A_{i-2}-\frac12\alpha_{i-1}\right)z_i^2
+\frac12\alpha_K\left(\sum_{j=1}^{K+1}z_j\right)^2
-\frac12\sum_{i=1}^{K}
(\alpha_i-\alpha_{i-1})
\left(\sum_{j=1}^i z_j\right)^2 \\
&\geq \sum_{i=1}^{K+1}
\left(A_K-\frac12A_{i-2}-\frac12\alpha_{i-1}\right)z_i^2
+\frac12\alpha_K\left(\sum_{j=1}^{K+1}z_j\right)^2
-\frac12\sum_{i=1}^{K}
i(\alpha_i-\alpha_{i-1})
\sum_{j=1}^{i}z_j^2 \\
&\geq \sum_{i=1}^{K+1}
\left(A_K-\frac12A_{i-2}-\frac12\alpha_{i-1}\right)z_i^2
+\frac12\alpha_K\left(\sum_{j=1}^{K+1}z_j\right)^2
-\frac12A_K\sum_{j=1}^{K}z_j^2 \\
&= \frac12\sum_{i=1}^{K}
\left(A_K-A_{i-1}\right)z_i^2
+\frac12A_K z_{K+1}^2
+\frac12\alpha_K\left(\sum_{j=1}^{K+1}z_j\right)^2
\geq 0.
\end{align*}
The second equality uses
\(2z_i\sum_{j=1}^{i-1}z_j
= (\sum_{j=1}^{i}z_j)^2-(\sum_{j=1}^{i-1}z_j)^2-z_i^2\), and the first inequality follows from
\((\sum_{j=1}^{i}z_j)^2\leq i\sum_{j=1}^{i}z_j^2\), together with
\(\alpha_i-\alpha_{i-1}\geq0\). The second inequality uses
\(\sum_{i=1}^{K}i(\alpha_i-\alpha_{i-1})\leq A_K\), which follows from
\(A_K=\alpha_K^2\) and \(\alpha_K\geq (K+1)/2\). The last equality follows from
\(A_{i-1}=A_{i-2}+\alpha_{i-1}\) and \(A_K=A_{K-1}+\alpha_K\). The last inequality holds because \(A_K-A_{i-1}>0\) for \(i=1,\ldots,K\) and \(A_K>0\). Since the last inequality is strict unless \(\bz=0\), we have \(\bW\succ0\) and hence \(\bC\succ 0\).

Next, we will select \(\tau\) so that \(\bC_\tau = \bC-\frac1\tau \bp\bp^\top\succ 0\).
Since \(\bC\succ0\), the rank-one Schur complement condition tells that
\[
    \bC-\frac1\tau\bp\bp^\top \succ  0
    \quad\Longleftrightarrow\quad
    \tau > \bp^\top\bC^{-1}\bp .
\]
Therefore, it suffices to upper bound
\(
   \bp^\top\bC^{-1}\bp .
\)
Applying the weighted norm notation \(\|\by\|_{\bC}:=(\by^\top\bC\by)^{1/2}\), whose dual norm is
\(\|\bp\|_{\bC^{-1}}=(\bp^\top\bC^{-1}\bp)^{1/2}\), we can deduce
\[
    \bp^\top\bC^{-1}\bp
    = \sup_{\by\ne0}
    \frac{(\bp^\top\by)^2}{\by^\top\bC\by} = \frac{L}{4A_K^2} \sup_{\bz \neq 0}\frac{
    \left(\sum_{i=1}^{K+1}
    (2A_K-A_{i-1}-A_{i-2})z_i\right)^2
    }{\bz^\top\bW\bz},
\]
where in the second equality we replace \(\by\) with \(\bz=\bD_\alpha\by\), and use the definition of \(\bC\) and \(\bp\). Next, we bound the numerator and denominator separately.
For the denominator, since \(\bW=\bP+\frac12\bQ+\bT \succeq \bP + \bT\),
we have
\[
    \bz^\top\bW\bz
    \geq \sum_{i=1}^{K}\alpha_i\left(\sum_{j=1}^{i}z_j\right)^2
    + \frac12\alpha_K\left(\sum_{j=1}^{K+1}z_j\right)^2.
\]
Therefore, using \(\alpha_{i-1}\leq \alpha_i\), and
\(
    \sum_{i=1}^{K+1}(2A_K-A_{i-1}-A_{i-2})z_i
    = \sum_{i=1}^{K}(\alpha_{i-1}+\alpha_i)\sum_{j=1}^{i}z_j
    + \alpha_K\sum_{j=1}^{K+1}z_j
\), we can bound the numerator as
\begin{align*}
& \left(
    \sum_{i=1}^{K+1}(2A_K-A_{i-1}-A_{i-2})z_i
\right)^2
\leq 2\left(
    \sum_{i=1}^{K}(\alpha_{i-1}+\alpha_i)\sum_{j=1}^{i}z_j
\right)^2
+ 2\alpha_K^2\left(\sum_{j=1}^{K+1}z_j\right)^2 \\
& \leq 8A_K\sum_{i=1}^{K}\alpha_i\left(\sum_{j=1}^{i}z_j\right)^2
+ 2\alpha_K^2\left(\sum_{j=1}^{K+1}z_j\right)^2 \leq
8A_K\,\bz^\top\bW\bz .
\end{align*}
The last inequality holds because with \(\alpha_K^2=A_K=\sum_{i=0}^{K}\alpha_i\geq \alpha_K \geq 1\).
With the above bounds, we can conclude
\[
    \bp^\top\bC^{-1}\bp
    \leq \frac{L}{4A_K^2} 8A_K
    = \frac{2L}{A_K}.
\]
Choosing \(\tau=4L/A_K\), the preceding bound and the rank-one inequality
\(\bp\bp^\top\preceq(\bp^\top\bC^{-1}\bp)\bC\) give
\(\bC_\tau=\bC-\tau^{-1}\bp\bp^\top\succeq \frac12\bC\succ0\) and
\(\bC_\tau^{-1}\preceq2\bC^{-1}\).

Having completed the first Schur complement step, it remains to dominate
\(\bB\bC_\tau^{-1}\bB^\top\) by the diagonal matrix \(\bU\).
For any \(\bw\in\mR^K\),
\[
    \bw^\top\bB\bC_\tau^{-1}\bB^\top\bw
    \leq 2\bw^\top\bB\bC^{-1}\bB^\top\bw =
    2\sup_{\by\ne0}
    \frac{(\bw^\top\bB\by)^2}{\by^\top\bC\by}.
\]
Set \(\bz=\bD_\alpha\by\). By the definition of \(\bC\) and the decomposition
\(\bW=\bP+\frac12\bQ+\bT \succ \frac{1}{2}\bQ\), it holds that
\begin{equation}\label{eq:yCy-lowerbound}
    \by^\top\bC\by
    = \frac{1}{LA_K^2}\bz^\top\bW\bz
    \geq \frac{1}{2LA_K^2}
    \sum_{i=0}^{K-1}\alpha_i
    \left(\sum_{j=i+2}^{K+1}z_j\right)^2 .
\end{equation}
For the numerator, since \(y_j=z_j/\alpha_{j-1}\), for \(i=1,\ldots,K\), a direct substitution gives
\begin{align*}
    & \abs{(\bB\by)_i}
    = \frac{\alpha_{i-1}}{2LA_K^2}
    \abs{
        \frac{(\alpha_{i-1}+\alpha_i)(2A_K-A_i)}{\alpha_i}z_{i+1}
        + \sum_{j=i+2}^{K+1}
        \left(
            2A_K+A_{i-2}-A_{j-2}-A_{j-1}
        \right)z_j
    } \\
    & \leq \frac{\alpha_{i-1}}{2LA_K^2}
    \abs{
        \frac{(\alpha_{i-1}+\alpha_i)(2A_K-A_i)}{\alpha_i}
        \sum_{j=i+1}^{K+1}z_j
    } + \frac{\alpha_{i-1}}{2LA_K^2}
    \abs{
        \sum_{j=i+3}^{K+1}
        (\alpha_{j-2}+\alpha_{j-1})
        \sum_{\ell=j}^{K+1}z_\ell
    } \\
    &\quad+
    \frac{\alpha_{i-1}}{2LA_K^2}
    \abs{
        \left[
            2A_K+A_{i-2}-A_i-A_{i+1}
            - \frac{(\alpha_{i-1}+\alpha_i)(2A_K-A_i)}{\alpha_i}
        \right]
        \sum_{j=i+2}^{K+1}z_j
    } \\
    &\leq \frac{2\alpha_{i-1}}{LA_K}
    \abs{\sum_{j=i+1}^{K+1}z_j}
    + \frac{\alpha_{i-1}}{2LA_K^2}
    \sum_{j=i+3}^{K+1}
    (\alpha_{j-2}+\alpha_{j-1})
    \abs{\sum_{\ell=j}^{K+1}z_\ell}+
    \frac{2\alpha_{i-1}}{LA_K}
    \abs{\sum_{j=i+2}^{K+1}z_j}  .
\end{align*}
Then, we can continue to derive
\begin{align*}
    \sum_{i=1}^{K}\frac{(\bB\by)_i^2}{\alpha_i}
    &\leq \frac{12}{L^2A_K^2}
    \sum_{i=1}^{K}\alpha_{i-1}
    \left(\sum_{j=i+1}^{K+1}z_j\right)^2
    + \frac{12}{L^2A_K^2}
    \sum_{i=1}^{K}\alpha_i
    \left(\sum_{j=i+2}^{K+1}z_j\right)^2 \\
    &\quad+
    \frac{3}{2L^2A_K^3}
    \sum_{i=1}^{K}\alpha_{i-1}
    \sum_{j=i+3}^{K+1}
    (\alpha_{j-2}+\alpha_{j-1})
    \left(\sum_{\ell=j}^{K+1}z_\ell\right)^2 \\
    &\leq \frac{27}{L^2A_K^2}
    \sum_{i=1}^{K}\alpha_{i-1}
    \left(\sum_{j=i+1}^{K+1}z_j\right)^2 .
\end{align*}
Therefore, by Cauchy--Schwarz inequality,
\begin{align*}
    (\bw^\top\bB\by)^2
    &\leq \left(\sum_{i=1}^{K}\alpha_iw_i^2\right)
    \left(\sum_{i=1}^{K}\frac{(\bB\by)_i^2}{\alpha_i}\right) \leq
    \frac{27}{L^2A_K^2}
    \left(\sum_{i=1}^{K}\alpha_iw_i^2\right)
    \left[
    \sum_{i=0}^{K-1}\alpha_i
    \left(\sum_{j=i+2}^{K+1}z_j\right)^2
    \right].
\end{align*}
Combining this with \eqref{eq:yCy-lowerbound}, we get
\[
   \bw^\top\bB\bC_\tau^{-1}\bB^\top\bw
   \leq 2 \sup_{\by\ne0}
    \frac{(\bw^\top\bB\by)^2}{\by^\top\bC\by}
    \leq \frac{108}{L}
    \sum_{i=1}^{K}\alpha_iw_i^2 .
\]
Thus, taking \(c_{\rm H}=54\) and choosing
\[
    u_i:=\frac{2c_{\rm H}}{L}\alpha_{i+1},
    \qquad  i=0,\ldots,K-1,
\]
we have \(\bU-\bB\bC_\tau^{-1}\bB^\top\succeq0\).
And the final convergence bound can be obtained by substituting the above choices of
\(\tau\) and \(\bu\) into the dual objective function. \Halmos

\subsection{Proof of \cref{thm:iogm-K-square-error-bound}}\label{appx:iOGM-K2}
We use the same multipliers as in \cref{thm:ifgm-K-square-error-bound}.
Compared with iFGM, the only difference is the algorithm parameters \(\theta_{k,i}\) that determine the blocks \(\bB\) and \(\bC\) of \(\bM\) in \eqref{eq:matrix-M}.
Let \(\bM^{\rm O}\), \(\bC^{\rm O}\) and \(\bB^{\rm O}\) denote the corresponding iOGM dual matrix and blocks, we choose
\[
    \tau=\frac{4L}{A_K},
    \qquad u_i=\frac{8c_{\rm H}}{L}\alpha_{i+1},
    \quad i=0,\ldots,K-1,
\]
where \(c_{\rm H}=54\) is the constant used in
\cref{thm:ifgm-K-square-error-bound}.  It remains to prove that
\(\bM^{\rm O}\succeq0\). Using the same \(\bD_\alpha\) as in the iFGM case, we can express \(\bC^{\rm O}\) as
\(
    \bC^{\rm O}
    = \frac{1}{LA_K^2}\bD_\alpha \bW^{\rm O}\bD_\alpha,
\)
where
\[
W^{\rm O}_{ij}
= \left\{
\begin{array}{ll}
2A_K-A_{i-1},
& i=j,\\[1mm]
2A_K-A_{i-1}-A_{i-2}+A_{j-2}+\dfrac12\alpha_{j-1},
& i>j,\\[1mm]
W^{\rm O}_{ji},
& j>i,
\end{array}
\right.
\]
We first show that \(\bC^{\rm O}\succ0\) and then prove \(\bC_\tau^{\rm O}:=\bC^{\rm O}-\tau^{-1}\bp\bp^\top\succ0\). Here
\(\bW^{\rm O}\) is decomposed into two parts as
\(
    \bW^{\rm O}=\bR+\bH,
\)
where
\[
R_{ij}
= \begin{cases}
2A_K-A_{i-1}-A_{i-2}-\dfrac12\alpha_{i-1},
& i=j,\\[1mm]
2A_K-A_{\max\{i,j\}-1}-A_{\max\{i,j\}-2},
& i\ne j,
\end{cases}
\]
and
\[
H_{ij}
= A_{\min\{i,j\}-2}
+ \frac12\alpha_{\min\{i,j\}-1},
\qquad i,j=1,\ldots,K+1.
\]
For any \(\bz\in\mathbb R^{K+1}\), a direct calculation proves \(\bR\succeq0\) as
\[
\begin{aligned}
    \bz^\top\bR\bz
    &= \frac12 z_1^2
    + \frac12\sum_{i=1}^{K}
    \alpha_i
    \left(
        \sum_{j=1}^{i}z_j
        + \sum_{j=1}^{i+1}z_j
    \right)^2  \geq 0.
\end{aligned}
\]
Similarly, we can also show bound \(\bH\) through
\[
    \bz^\top\bH\bz
    = \frac{\alpha_0}{2}
    \left(\sum_{j=1}^{K+1}z_j\right)^2
    + \frac12\sum_{i=2}^{K+1}
    (\alpha_{i-2}+\alpha_{i-1})
    \left(\sum_{j=i}^{K+1}z_j\right)^2  \geq 0.
\]
Every coefficient on the right-hand side is positive.  If the right-hand side
is zero, then
\[
    \sum_{j=K+1}^{K+1}z_j=0,\quad
    \sum_{j=K}^{K+1}z_j=0,\quad \ldots,\quad
    \sum_{j=1}^{K+1}z_j=0,
\]
which implies \(z_{K+1}=z_K=\cdots=z_1=0\),  hence \(\bH\succ0\) and \(\bC^{\rm O}\succ0\). Next, we will bound \(\bp^\top(\bC^{\rm O})^{-1}\bp\). The calculation is similar to the iFGM case, and we can derive
\[
    \bp^\top(\bC^{\rm O})^{-1}\bp
    = \frac{L}{4A_K^2}
    \sup_{\bz\ne0}
    \frac{
    \left(
        \sum_{i=1}^{K+1}
        (2A_K-A_{i-1}-A_{i-2})z_i
    \right)^2
    }{
        \bz^\top\bW^{\rm O}\bz
    } \leq \frac{L}{2A_K}\sup_{\bz\ne0} \frac{ \bz^\top \bR \bz}{\bz^\top\bW^{\rm O}\bz} \leq \frac{L}{2A_K}.
\]
With \(\tau=4L/A_K\), we can conclude
\[
    \bC_\tau^{\rm O}
    = \bC^{\rm O}-\frac1\tau\bp\bp^\top
    \succeq
    \frac78\bC^{\rm O}
    \succ0 .
\]

It remains to dominate
\(\bB^{\rm O}(\bC_\tau^{\rm O})^{-1}(\bB^{\rm O})^\top\) by \(\bU\).
From the previous derivation, we know
\(
    (\bC_\tau^{\rm O})^{-1}
    \preceq
    \frac87(\bC^{\rm O})^{-1}.
\)
Therefore, for any \(\bw\in\mathbb R^K\),
\[
    \bw^\top \bB^{\rm O}(\bC_\tau^{\rm O})^{-1}
    (\bB^{\rm O})^\top\bw
    \leq \frac87
    \sup_{\by\ne0}
    \frac{(\bw^\top\bB^{\rm O}\by)^2}
         {\by^\top\bC^{\rm O}\by}.
\]
With \(\bz=\bD_\alpha\by\), we can write
\[
    \by^\top\bC^{\rm O}\by
    = \frac{1}{LA_K^2}\bz^\top\bW^{\rm O}\bz \geq \frac{1}{LA_K^2} \bz^\top \bH \bz
    \geq \frac{1}{2LA_K^2}
    \sum_{i=0}^{K-1}
    \alpha_i
    \left(
        \sum_{j=i+2}^{K+1}z_j
    \right)^2 .
\]
Moreover, following the same technique used in \cref{thm:ifgm-K-square-error-bound}, yields the uniform bound
\[
    \sum_{i=1}^{K}
    \frac{((\bB^{\rm O}\by)_i)^2}{\alpha_i}
    \leq \frac{189}{L^2A_K^2}
    \sum_{i=0}^{K-1}
    \alpha_i
    \left(
        \sum_{j=i+2}^{K+1}z_j
    \right)^2 .
\]
Continue the rest calculation, we will have
\[
    \sup_{\by\ne0}
    \frac{(\bw^\top\bB^{\rm O}\by)^2}
         {\by^\top\bC^{\rm O}\by}
    \leq \frac{378}{L}
    \sum_{i=1}^{K}\alpha_i w_i^2 .
\]
Consequently,
\[
\begin{aligned}
    \bw^\top \bB^{\rm O}(\bC_\tau^{\rm O})^{-1}
    (\bB^{\rm O})^\top\bw
    &\leq \frac87\cdot \frac{378}{L}
    \sum_{i=1}^{K}\alpha_i w_i^2     =
    \frac{432}{L}
    \sum_{i=1}^{K}\alpha_i w_i^2,
\end{aligned}
\]
As \(c_{\rm H}=54\), choosing
\(
    u_i=\frac{8c_{\rm H}}{L}\alpha_{i+1},
    \  i=0,\ldots,K-1,
\)
gives
\(
    \bU-
    \bB^{\rm O}(\bC_\tau^{\rm O})^{-1}
    (\bB^{\rm O})^\top
    \succeq 0
\) and the final result can be proved. \Halmos

\section{Proofs of the lower bound}
\subsection{Proof of \cref{thm:curvature-dependent-quadratic-lower-bound}} \label{appx:proof-curvature-dependent}
\begin{proof}{Proof}
For the scalar quadratic \(f_\lambda(x)=\frac{\lambda}{2}x^2\), we have
\(\grad f_\lambda(x)=\lambda x=L\rho x\). Hence, the iGOGM updates become
\begin{align*}
    y_{k+1}&=(1-\rho)x_k-\frac{1}{L}e_k, \\
    z_{k+1}&=z_k-2\alpha_k\rho x_k-\frac{2\alpha_k}{L}e_k, \\
    x_{k+1}&=(1-\beta_{k+1})y_{k+1}+\beta_{k+1}z_{k+1}.
\end{align*}
Let
\(
    \delta x_k:=x_k-x_k^{\rm ex},
    \text{and } \delta z_k:=z_k-z_k^{\rm ex}.
\)
Subtracting the exact-gradient recursion from the inexact-gradient recursion
gives
\begin{align*}
    \delta y_{k+1}&=(1-\rho)\delta x_k-\frac{1}{L}e_k, \\
    \delta z_{k+1}&=\delta z_k-2\alpha_k\rho\,\delta x_k-\frac{2\alpha_k}{L}e_k, \\
    \delta x_{k+1}&=(1-\beta_{k+1})\delta y_{k+1}+
    \beta_{k+1}\delta z_{k+1}.
\end{align*}
For a fixed error index \(i\), suppose only \(e_i\) is nonzero and write
\[
    \delta x_k=-\frac{1}{L}w_{k,i}(\rho)e_i,
    \qquad \delta z_k=-\frac{1}{L}q_{k,i}(\rho)e_i .
\]
At time \(i+1\), the error \(e_i\) enters the recursion and gives
\[
    w_{i+1,i}(\rho)
    = 1-\beta_{i+1}+2\alpha_i\beta_{i+1},
    \qquad q_{i+1,i}(\rho)
    = 2\alpha_i.
\]
For later times \(k=i+1,\ldots,K-1\), substituting the coefficient
representations into the perturbation recursion gives
\[
    \begin{bmatrix}
        w_{k+1,i}(\rho)\\
        q_{k+1,i}(\rho)
    \end{bmatrix}
    = \mathbf M_k(\rho)
    \begin{bmatrix}
        w_{k,i}(\rho)\\
        q_{k,i}(\rho)
    \end{bmatrix},
\]
with \(\mathbf M_k(\rho)\) as stated. By linearity, the contribution of all
errors is the sum of the individual contributions, and hence
\(
    x_K-x_K^{\rm ex}
    = -\frac{1}{L}
    \sum_{i=0}^{K-1}w_{K,i}(\rho)e_i.
\)
Since the problem is one-dimensional, the adversary can choose the signs of
\(e_i\) to align with the signs of \(w_{K,i}(\rho)\). Therefore,
\(
    \mathcal E_K(\rho)
    = \sum_{i=0}^{K-1}\abs{w_{K,i}(\rho)},
\)
which completes the proof. \Halmos
\end{proof}

\subsection{Proof of \cref{lem:zero-curvature-amplification-iGOGM}} \label{appx:proof-zero-iGOGM}
\begin{proof}{Proof}
When \(\rho=0\),
\[
    \mathbf M_k(0)
    = \begin{bmatrix}
        1-\beta_{k+1} & \beta_{k+1}\\
        0 & 1
    \end{bmatrix}.
\]
Since
\(
    1-\beta_{k+1}
    = 1-\frac{\alpha_{k+1}}{A_{k+1}}
    = \frac{A_k}{A_{k+1}},
\)
we have
\(
    \prod_{j=i+1}^{K-1}(1-\beta_{j+1})
    = \prod_{j=i+1}^{K-1}\frac{A_j}{A_{j+1}}
    = \frac{A_{i+1}}{A_K}.
\)
Equivalently, using the initialization at time \(i+1\) and telescoping the
recursion for \(w_{k,i}(0)\) yields
\(
    w_{K,i}(0)
    = 2\alpha_i
    + \left(w_{i+1,i}(0)-2\alpha_i\right)\frac{A_{i+1}}{A_K}.
\)
Since
\(
    w_{i+1,i}(0)
    = 1-\beta_{i+1}+2\alpha_i\beta_{i+1},
\)
and \(\beta_{i+1}=\alpha_{i+1}/A_{i+1}\), this expression can be written,
after simplification using \(A_{i+1}=A_i+\alpha_{i+1}\), as
\(
    w_{K,i}(0)
    = 2\alpha_i
    + (1-2\alpha_i)\frac{A_i}{A_K}.
\)
Therefore,
\(
    \mathcal E_K(0)
    = \sum_{i=0}^{K-1}\abs{w_{K,i}(0)}
    = \sum_{i=0}^{K-1}
    \left|
        2\alpha_i
        + (1-2\alpha_i)\frac{A_i}{A_K}
    \right|.
\)\Halmos
\end{proof}

\subsection{Proof of \cref{lem:zero-curvature-amplification-iOGMa}} \label{appx:zero-iOGMa}
\begin{proof}{Proof}
By \cref{lem:zero-curvature-amplification-iGOGM},
\[
    w_{K,i}(0)
    = \frac{2(i+a)}{a}
    + \left(1-\frac{2(i+a)}{a}\right)
    \frac{(i+1)(i+2a)}{(K+1)(K+2a)}.
\]
For indices \(i\leq K/2\), the ratio \(A_i/A_K\) is bounded above by a constant
strictly smaller than one. Moreover, \(\alpha_i=(i+a)/a=\Theta(i+1)\). Hence,
for all \(i\leq K/2\) and all sufficiently large \(K\),
\(
    \abs{w_{K,i}(0)}
    \geq c_a'(i+1)
\)
for some constant \(c_a'>0\) depending only on \(a\). Therefore,
\[
    \mathcal E_{K,a}(0)
    = \sum_{i=0}^{K-1}\abs{w_{K,i}(0)}
    \geq \sum_{i=0}^{\lfloor K/2\rfloor}c_a'(i+1)
    \geq c_aK^2.
\]
On the other hand, since \(0\leq A_i/A_K\leq 1\) and \(\alpha_i=O(K)\) for
\(i=0,\ldots,K-1\), there exists a constant \(C_a'>0\) such that
\(
    \abs{w_{K,i}(0)}\leq C_a'K
    \
    \text{for all } i=0,\ldots,K-1.
\)
Thus,
\[
    \mathcal E_{K,a}(0)
    \leq \sum_{i=0}^{K-1}C_a'K
    = C_aK^2.
\]
Combining the two bounds proves the claim. \Halmos
\end{proof}

\subsection{Proof of \cref{lem:small-curvature-amplification-iOGMa}} \label{appx:proof-amplification-iOGMa}
\begin{proof}{Proof}
At zero curvature, we have
\[
w_{K,i}(0)=2\alpha_i+(1-2\alpha_i)A_i/A_K
=2\alpha_i(1-A_i/A_K)+A_i/A_K.
\] 
Consider the middle index set
\(\mathcal I_K:=\{i:\ K/4\le i\le K/2\}\). For every
\(i\in\mathcal I_K\), we have \(\alpha_i=\Theta(K)\). Moreover,
\(A_i/A_K=((i+1)(i+2a))/((K+1)(K+2a))\le \theta_a<1\) for all
sufficiently large \(K\), where \(\theta_a\in(0,1)\) depends only on \(a\).
Hence there exists \(c_a^0>0\), depending only on \(a\), such that
\(|w_{K,i}(0)|=w_{K,i}(0)\ge c_a^0K\) for all \(i\in\mathcal I_K\).

It remains to show that this lower bound is stable under the perturbation
\(\rho=\mu/K^2\). Let \(s_{k,i}(\rho):=(w_{k,i}(\rho),q_{k,i}(\rho))^\top\).
The coefficient recursion can be written as
\(s_{k+1,i}(\rho)=M_k(\rho)s_{k,i}(\rho)\). Moreover,
\[
    M_k(\rho)-M_k(0)
    =
    \begin{bmatrix}
        -\rho\{(1-\beta_{k+1})+2\alpha_k\beta_{k+1}\} & 0\\
        -2\alpha_k\rho & 0
    \end{bmatrix}.
\]
For iOGM-\(a\), uniformly over \(k\le K\), we have
\(\alpha_k=O(K)\) and \(\alpha_k\beta_{k+1}=O(1)\). Thus, for
\(\rho=\mu/K^2\), \(\|M_k(\rho)-M_k(0)\|\le C_a\mu/K\). Also,
\(\|M_k(\rho)\|\le 1+C_a\mu/K\), and the initial coefficient state satisfies
\(\|s_{i+1,i}(\rho)\|\le C_aK\). Therefore, by a discrete Gronwall argument,
\(\|s_{k,i}(\rho)\|\le C_aK\) for all \(i+1\le k\le K\), provided
\(0<\mu\le \mu_a\), after possibly decreasing \(\mu_a\).

Define \(d_{k,i}:=s_{k,i}(\rho)-s_{k,i}(0)\). Since the initial state is
independent of \(\rho\), we have \(d_{i+1,i}=0\). Also,
\(d_{k+1,i}=M_k(0)d_{k,i}+(M_k(\rho)-M_k(0))s_{k,i}(\rho)\). Using the
preceding bounds gives
\[
    \|d_{k+1,i}\|
    \le
    \|d_{k,i}\|
    +
    C_a\frac{\mu}{K}\cdot C_aK
    \le
    \|d_{k,i}\|+C_a\mu .
\]
Iterating over at most \(K\) steps yields \(\|d_{K,i}\|\le C_a\mu K\), and
therefore
\[
    \left|
        w_{K,i}\left(\frac{\mu}{K^2}\right)-w_{K,i}(0)
    \right|
    \le C_a\mu K .
\]
Choose \(\mu_a>0\) sufficiently small so that \(C_a\mu_a\le c_a^0/2\). Then,
for all \(0<\mu\le\mu_a\) and all \(i\in\mathcal I_K\),
\[
    \left|
        w_{K,i}\left(\frac{\mu}{K^2}\right)
    \right|
    \ge
    |w_{K,i}(0)|
    -
    \left|
        w_{K,i}\left(\frac{\mu}{K^2}\right)-w_{K,i}(0)
    \right|
    \ge
    \frac{c_a^0}{2}K .
\]
Since \(|\mathcal I_K|=\Theta(K)\), it follows that
\[
    \mathcal E_{K,a}\left(\frac{\mu}{K^2}\right)
    \ge
    \sum_{i\in\mathcal I_K}
    \left|
        w_{K,i}\left(\frac{\mu}{K^2}\right)
    \right|
    \ge
    c_aK^2 .
\]
This completes the proof. \Halmos
\end{proof}

\section{Numerical solutions for optimized algorithm with inexact oracle}{\label{sec:numerical-optimized-algorithm}}
	We provide the optimized stepsize for different \(K\) with \(L = 1\) and \(R=1\). We report three results, the first one is for \(\bb = \bo\), the second one is for \(b_k^2 \equiv \bar{b}^2 = 0.01\) and the third one takes \(\seq{b_k}\) to be a decreasing sequence. Note that \(\bar{b}^2\) in the second scenario is set equal to the average value of \(b_k^2\) in the third scenario.

\begin{table}[hbt]
    \centering
    \caption{Numerical solution of optimized first-order algorithm with \(b_i \equiv 0\)}
    \label{tab:theta-b0}
    \begin{tabular}{@{}l@{\qquad}l@{}}
    \hline
    \noalign{\vskip 3pt}
        \(K\) & \multicolumn{1}{c}{\(\theta^{b=0}\)}        \\
        \hline
         1 & \(\begin{pmatrix}
			               1.6180
		               \end{pmatrix}\)\\
        2 & \(\begin{pmatrix}
			               1.6180 & 0 \\ 1.7921 & 2.0193
		               \end{pmatrix}\)\\
        3 & \(\begin{pmatrix}
			                  1.6180 & 0 & 0 \\ 1.7921 & 2.0193 & 0\\ 1.8677 & 2.4618 & 2.2316
		                  \end{pmatrix} \)\\
        4 & \(\begin{pmatrix}
			                  1.6180 & 0 & 0 & 0 \\ 1.7921 & 2.0193 & 0 & 0\\ 1.8677 & 2.4617 & 2.2316 & 0\\ 1.9078 & 2.6966 & 2.8856 & 2.3654
		                  \end{pmatrix}  \)\\
                    5 & \(\begin{pmatrix}
			               1.6180 & 0 & 0 & 0 & 0 \\ 1.7921 & 2.0193 & 0 & 0 & 0\\ 1.8676 & 2.4617 & 2.2315 & 0 & 0\\ 1.9078 & 2.6966 & 2.8855 & 2.3653 & 0\\ 1.9318 & 2.8373 & 3.2771 & 3.1828 & 2.4580
		               \end{pmatrix}\) \\ \hline
    \end{tabular}
\end{table}

\begin{table}[hbt]
    \centering
    \renewcommand{\arraystretch}{1.2}
    \caption{Numerical solution of optimized first-order algorithm with \(b_i\equiv \bar{b}\)}
    \label{tab:theta-barb}
    \begin{tabular}{@{}l@{\qquad}l@{}}
    \hline
    \noalign{\vskip 3pt}
        \(K\) & \multicolumn{1}{c}{\(\theta^{b=\bar{b}}\)}        \\
        \hline
         1 & \(\begin{pmatrix}
			                     1.5509
		                     \end{pmatrix}\)\\
        2 & \( \begin{pmatrix}
			                                                        1.5537 & 0 \\ 1.7028 & 1.7642
		                                                        \end{pmatrix}\)\\
        3 & \(\begin{pmatrix}
			          1.5552 & 0 & 0 \\ 1.7075 & 1.7828 & 0\\ 1.7668 & 2.0875 & 1.7590 \end{pmatrix} \)\\
        4 & \(\begin{pmatrix}
			                        1.5561 & 0 & 0 & 0 \\ 1.7101 & 1.7926 & 0 & 0\\ 1.7719 & 2.1112 & 1.7962 & 0\\ 1.7984 & 2.2476 & 2.1371 & 1.6578
		                        \end{pmatrix}   \)\\
                    5 & \(\begin{pmatrix}
			1.3714 & 0 & 0 & 0 & 0 \\ 1.4958 & 1.4457 & 0 & 0 & 0\\ 1.5444 & 1.6859 & 1.3960 & 0 & 0\\ 1.5675 & 1.8115 & 1.6837 & 1.3303 & 0\\ 1.5644 & 1.8444 & 1.7931 & 1.5876 & 1.2111
			\end{pmatrix}\) \\ \hline
    \end{tabular}
\end{table}

\begin{table}[hbt]
    \centering
    \renewcommand{\arraystretch}{1.2}
    \caption{Numerical solution of the optimized first-order algorithm with \(b_i\) decreasing}
    \label{tab:theta-bdecrease}
    \begin{tabular}{@{}l@{\qquad}l@{}}
    \hline
    \noalign{\vskip 3pt}
        \(K\) & \multicolumn{1}{c}{\(\theta^{b_k\downarrow} \)}        \\
        \hline
         1 & \( \begin{pmatrix}
			                         1.5509
		                         \end{pmatrix}\)\\
        2 & \( \begin{pmatrix}
			                                                                                                     1.5478 & 0 \\ 1.7047 & 1.7954
		                                                                                                     \end{pmatrix}\)\\
        3 & \(\begin{pmatrix}
			                        1.5495 & 0 & 0 \\ 1.7043 & 1.7875 & 0\\ 1.7690 & 2.1161 & 1.8192
		                        \end{pmatrix} \)\\
        4 & \(\begin{pmatrix}
			                       0.0335 & 0 & 0 & 0 \\ 0.0276 & 1.5575 & 0 & 0\\ 0.0259 & 1.7218 & 1.8158 & 0\\ 0.0252 & 1.7925 & 2.1663 & 1.8749
		                        \end{pmatrix}  \)\\
                    5 & \(\begin{pmatrix}
			                         0.0111 & 0 & 0 & 0 & 0 \\ 0.0014 & 0.5086 & 0 & 0 & 0\\ 0.0005 & 0.4713 & 1.6643 & 0 & 0\\ 0.0002 & 0.4588 & 1.8883 & 1.8380 & 0\\ 0.0000 & 0.4531 & 1.9905 & 2.2202 & 1.8669
		                         \end{pmatrix}\) \\ \hline
    \end{tabular}
\end{table}

	We can observe that the stepsize at each iteration with the exact gradient oracle remains the same, while this is not the case with the inexact oracle. Even for the fixed inexactness level (second scenario), the optimized algorithm changes with the total iteration number. Another fact is that the optimized algorithm can identify a bad gradient estimate (see the first iteration of \(\theta^{b_k\downarrow}\)). It almost discards the first iteration since its inexactness is beyond some unknown threshold. Apart from the stepsize result, the optimized objective function value indicates that the initial condition \(R\) plays a role since \(\tau\) for the three scenarios has different values. This is different from the exact case \cite{Drori2014PerformanceFirstOrder, Kim2016OptimizedFirstOrder} as optimized stepsize is independent of \(R\) and \(L\).

    \begin{table}[hbt]
    \centering
    \renewcommand{\arraystretch}{1.2}
    \caption{Objective values for the optimized algorithm}
    \begin{tabular}{@{}c@{\qquad}c@{\quad}c@{\quad}c@{\qquad}c@{\quad}c@{\quad}}
    \hline
    \noalign{\vskip 3pt}
			\(K\) & \(\tau^{b=0}\) & \(\tau^{b=\bar{b}}\) & \(\tau^{b \downarrow}\) & \(\bar{b}^2\sum u^{b=\bar{b}}_k\) & \(\sum u^{b_k\downarrow}_k b^2_k\) \\
			\hline
			1   & 0.0955       & 0.1123             & 0.0185                & 0.0185                        & 0.0185                           \\
			2   & 0.0520       & 0.0743             & 0.0271                & 0.0273                        & 0.0271                           \\
			3   & 0.0331       & 0.0583             & 0.0340                & 0.0339                        & 0.0340                           \\
			4   & 0.0230       & 0.0504             & 0.0310                & 0.0396                        & 0.0310                           \\
				5   & 0.0170       & 0.0500             & 0.0320                & 0.0400                        & 0.0320 \\ \hline
		\end{tabular}
	\end{table}
\end{appendices}

\end{document}